\documentclass[10pt]{amsart}
\pdfoutput=1
\usepackage{amssymb}
\usepackage{latexsym}
\usepackage{amscd} 
\usepackage{amsmath}
\usepackage{epsfig}
\usepackage{eepic}
\usepackage{mathrsfs}
\usepackage[margin=3.00cm]{geometry}
\usepackage{ifpdf}
\usepackage[textures,bookmarks=false]{hyperref}
\usepackage{bm}

\usepackage{color}
\usepackage{graphicx}

\usepackage{amsthm}
\newtheorem{theorem}{Theorem}[section]
\newtheorem{introtheorem}{Theorem}
\newtheorem*{introconjecture}{Conjecture}

\newtheorem{corollary}[theorem]{Corollary}

\newtheorem{lemma}[theorem]{Lemma}
\newtheorem{definition}[theorem]{Definition}
\newtheorem{proposition}[theorem]{Proposition}
\newtheorem{remark}[theorem]{Remark}

\def\CC{\mathbb{C}}
\def\C{\mathbb{C}}

\def\H{{\mathbb{H}}}

\def\R{\mathbb{R}}
\def\P{\mathbb{P}}
\def\Q{\mathbb{Q}}

\def\N{\mathbb{N}}
\def\L{\mathbb{L}}
\def\Z{\mathbb{Z}}
\def\D{\mathbb{D}}

\def\QS{\Q / \Z}
\def\RZ{\R / \Z}
\def\CDC{\CC \setminus \overline{\D}}
\def\CP2{{\mathbb{CP}^2}}
\def\Nnot{\mathbb{N}_0}

\def\cA{{\mathcal{A}}}
\def\cB{{\mathcal{B}}}

\def\cI{{\mathcal{I}}}

\def\cL{{\mathcal{L}}}
\def\cM{{\mathcal{M}}}

\def\cO{{\mathcal{O}}}

\def\cT{{\mathcal{T}}}
\def\cV{{\mathcal{V}}}
\def\cW{{\mathcal{W}}}

\def\z{\zeta}
\def\vphi{\varphi}
\def\la{\lambda}
\def\veta{\vartheta}

\def\znot{{\z_0}}

\def\zone{{\z_1}}

\def\zp{\z^\prime}

\def\omegap{{\omega^\prime}}

\def\xip{{\xi^\prime}}
\def\bla{\la}
\def\itin{\mathit{it}}

\def\laurent{\C((\tau))}
\def\laurentm{\C((\tau^{1/m}))}
\def\puiseux{\C \langle\langle \tau \rangle\rangle }
\def\ponel{\P^1_\L}
\def\ponec{\P^1_\C}

\def\lval{|}
\def\rval{|_\mathrm{o}}

\def\crit{\operatorname{Crit}}

\def\fatouvphi{F(\vphi)}

\def\juliavphi{J(\vphi)}
\def\hatjuliavphi{\widehat{J(\vphi)}}
\def\filledvphi{K(\vphi)}

\def\lv{\ell}  %level

\def\mdl{\operatorname{mod}}

\def\dend{E} %dynamical end
\def\capdend{E^{\cap}}

\def\dpi{{P}} %berkovich piece
\def\dpill{{P_\lv}} %berkovich level l piece
\def\dd{{D}} %dynamical direction

\def\ddll{\dd_\lv}

\def\marked{\mathbf{M}}

\def\hl{{\H_\L}}
\def\pberl{\P^{1,an}_{\L}}

\def\dist{\operatorname{dist}}
\newcommand{\tang}[1]{\ensuremath{T_{#1} {\pberl}}}
\newcommand{\val}[1]{\ensuremath{|#1|_o}}
\newcommand{\tree}[1]{\ensuremath{\cT^{(#1)}}}
\newcommand{\treec}[1]{\ensuremath{\cT_c^{(#1)}}}
\newcommand{\trees}[1]{\ensuremath{\cT_s^{(#1)}}}
\newcommand{\arbol}[1]{\ensuremath{\cA^{(#1)}}}
\newcommand{\barbol}[1]{\ensuremath{\cB^{(#1)}}}
\newcommand{\barbolc}[1]{\ensuremath{\cB_c^{(#1)}}}
\newcommand{\barbols}[1]{\ensuremath{\cB_s^{(#1)}}}

\newcommand{\intervala}[1]{\ensuremath{\cI_\cA^{(#1)}}}
\newcommand{\intervalb}[1]{\ensuremath{\cI_\cB^{(#1)}}}
\newcommand{\intervalt}[2]{\ensuremath{\cI^{(#1)}_{#2}}}

\newcommand{\gam}[1]{\ensuremath{\Gamma^{(#1)}}}
\newcommand{\gamc}[1]{\ensuremath{\Gamma_c^{(#1)}}}
\newcommand{\why}[1]{\ensuremath{Y^{(#1)}}}
\newcommand{\lam}[1]{\ensuremath{\lambda^{(#1)}}}
\newcommand{\lamin}[1]{\ensuremath{\cL^{(#1)}}}

\newcommand{\cIpar}[1]{\ensuremath{\cI^{(#1)}}}
\newcommand{\cVpar}[1]{\ensuremath{\cV^{(#1)}}}
\newcommand{\cWpar}[1]{\ensuremath{\cW^{(#1)}}}
\newcommand{\cvpar}[1]{\ensuremath{v^{(#1)}}}
\newcommand{\cwpar}[1]{\ensuremath{w^{(#1)}}}

\begin{document}

%\frontmatter

\title{Puiseux series dynamics of Quadratic rational maps}
\author{Jan Kiwi}
\address{Facultad de Matem\'aticas,
 Pontificia Universidad Cat\'olica,
 Casilla 306, Correo 22, Santiago,
 Chile.}
\thanks{Partially supported by ``Fondecyt \#1060730'', Research Network on Low Dimensional Dynamics PBCT/CONICYT ACT17, Conicyt, Chile 
and  ECOS C07E01.}
\email{jkiwi@puc.cl}
\date{\today}
\keywords{Puiseux series, Julia sets}
%{\it Mots cl\'es:} S\'eries de Puiseux, ensemble de Julia}
\subjclass[2010]{Primary: 37P50; Secondary: 37F20}
\begin{abstract}
  We give a complete description for the dynamics of quadratic rational maps 
with coefficients in the completion of the field of formal Puiseux series.
\end{abstract}
\maketitle

\setcounter{tocdepth}{3}
%\setcounter{page}{4}
%\tableofcontents

%\mainmatter

\section*{Introduction}
\label{section|Introduction}
Consider 
 the completion $\L$ of the field of formal Puiseux series $\puiseux$ with complex coefficients. In this paper we study the dynamics  of quadratic rational maps with coefficients in $\L$.

The results 
obtained  are in the context of the recent developments in the study
of rational maps acting on the projective line over non-Archimedean
fields (e.g. see~\cite{BakerRumelyBook} and references therein). 
The main motivation  arises
from the interplay established in~\cite{KiwiFourier} between dynamics over $\L$
and the study of rational maps with complex coefficients acting on the Riemann sphere.

The main purpose of this paper is to give a full description for
 the dynamics of quadratic rational maps with coefficients in $\L$. 
The interplay with dynamics over $\C$ is also manifested here. More precisely, our description of the geometry and topological dynamics of a large class of maps is based on the construction of a ``model'' for the dynamics which finds its origin in the  study of the complex quadratic polynomial family 
$z \mapsto z^2 +c$.

  The natural space to study the 
dynamics of a rational map $\vphi \in \L(z)$ is the Berkovich projective line
$\pberl$ which contains the projective line $\ponel$ as a dense subset 
(see Section~\ref{section|BerkovichLine}). In contrast with $\ponel$, the Berkovich line $\pberl$ is a compact and arcwise connected topological space. Moreover,
$\pberl$ is a non-metric tree. 
As in complex dynamics, the Berkovich line is divided into two complementary sets: the Fatou and  Julia sets  (see~Section~\ref{section|DynamicsL}). 
The Julia set $\juliavphi$ is either a singleton or uncountable. If $J(\vphi)$ 
is a singleton, then we say that $\vphi$ is {\sf simple}.

Our description of the dynamical properties of quadratic rational maps $\vphi: \pberl \to \pberl$ is organized according to 
the character and number of periodic orbits in $\juliavphi \setminus \ponel$. 
We refer the reader to Section~\ref{sec:periodic.points} for the basics about periodic points.
A periodic point $x$  in $\juliavphi \setminus \ponel$ must be a topological branched point of $\pberl$. 
The action of $\vphi^p$ on the branches growing from a period $p$ point $x$ is encoded by a complex rational map
$T_x \vphi^p : \ponec \to \ponec$.

\begin{introtheorem}
  \label{ithr:periodic}
  Let $\vphi: \pberl \to \pberl$ be a quadratic rational map which is not simple. Then one and exactly one of the following holds:
  \begin{enumerate}
  \item $\juliavphi \setminus \ponel = \emptyset$.
  \item $\juliavphi \setminus \ponel$ is the grand orbit of an indifferent periodic orbit. 
  \item $\juliavphi \setminus \ponel$ is the grand orbit of one repelling periodic orbit $\cO$ of period $q \ge 2$. 
For all $x \in \cO$, the  map $T_x \vphi^q : \ponec \to \ponec$ is a quadratic rational map with a multiple fixed point. 
  \item $\juliavphi \setminus \ponel$ is the grand orbit of two distinct periodic orbits $\cO, \cO'$ of periods $q,  q' \ge 2$, respectively, where $q' > q$. 
The  map $T_x \vphi^q : \ponec \to \ponec$ is a  quadratic rational map with a multiple fixed point, for all $x \in \cO$.
The map  $T_x \vphi^{q'} : \ponec \to \ponec$ is (modulo choice of coordinates) a quadratic polynomial, for all $x \in \cO'$.
  \end{enumerate}
\end{introtheorem}

Next we describe the dynamics and geometry of Fatou set components.
Given a periodic orbit $\cO$, the set of points $x \in \pberl$ such that $\vphi^n(x)$ converges to
$\cO$ is called the basin of $\cO$. 
According to Rivera-Letelier's classification~\cite{riveratesis03}, every fixed Fatou component $U$
is either a component of the basin of an attracting fixed point $z \in \ponel$ or, $U$ is a component of the interior of an affinoid such that 
$\vphi: U \to U$ is a bijection. The above classification extends naturally to periodic Fatou components.
In the latter case, we say that $U$ is a ``Rivera domain''. 
For a Rivera domain $U$, fixed under $\vphi$, 
which is not a ball, the convex hull $\cA_U$ of $\partial U$ in $\pberl$ 
may be regarded as an invariant finite simplicial tree. If $\cA_U$ contains at most one branched point and exactly one fixed point, then we say that $U$ is an {\sf starlike Rivera domain}. (See Section~\ref{sec:PeriodicFatou}.)

It should be mentioned that, for dynamics over $\L$, the basin of a periodic orbit  $\cO \subset \juliavphi \setminus \pberl$ is always non-emtpy and contains
uncountably many Fatou components, even if $\cO$ is a ``repelling'' periodic orbit.

\begin{introtheorem}
  \label{ithr:fatou}
  Let $\vphi: \pberl \to \pberl$ be a quadratic rational map which is not simple. If $U$ is a Fatou component,
then $U$ is eventually periodic or, $U$ is a ball or an annulus  contained in the basin of a periodic orbit $\cO \subset \juliavphi \setminus \ponel$. 
Moreover, one and exactly one of the following holds:
\begin{enumerate}
  \item $\juliavphi \setminus \ponel = \emptyset$ and the Fatou set consists of one totally invariant component which is the basin of an attracting fixed point $z \in \ponel$. 
  \item $\juliavphi \setminus \ponel$ is the grand orbit of an indifferent fixed point.
Every periodic Fatou component is a  fixed Rivera domain consisting of an open ball.
\item  $\juliavphi \setminus \ponel$ is the grand orbit of an indifferent periodic orbit of period at least $2$.
There exists a unique periodic Fatou component. This Fatou component is a starlike Rivera domain.
  \item $\juliavphi \setminus \ponel$ is the grand orbit of one or two repelling periodic orbit. There is exactly one fixed Fatou component. This fixed Fatou component is a starlike Rivera domain whose boundary is the repelling orbit of least period in  $\juliavphi \setminus \ponel$. 
Moreover, periodic Fatou components of higher periods are open balls.
  \end{enumerate}
\end{introtheorem}

Let us agree that a Fatou component $U$ is called {\sf wandering} if $\vphi^n(U) \neq \vphi^m(U)$ for all non negative integers $n \neq m$ and $U$ is not contained in the basin of a periodic orbit. 
Benedetto~\cite{benedetto02} has shown that wandering Fatou components  do arise
in polynomial dynamics over the $p$-adic field $\C_p$. 
Our previous theorem rules out the existence of wandering Fatou components for quadratic rational maps over $\L$. 
For polynomials with coefficients in fields such as $\L$, 
Trucco~\cite{trucco09} has already ruled out the existence
of wandering Fatou components. Thus, there is some supporting evidence for the following.

\begin{introconjecture}
\label{con:wandering}
 {\em  Let $L$ be a field endowed with a non-Archimedean absolute value which is complete with respect to the induced metric.
  Assume that the residual field of $L$ has characteristic zero. Suppose that $L$ contains a field $F$ with discrete value group such that
the elements of $L$ which are algebraic over $F$ are dense in $L$. 
If $\varphi \in L(\z)$ is a rational map of degree at least $2$, then
every Fatou component of $\varphi$ is eventually periodic or is contained in the basin of a periodic orbit.  }
\end{introconjecture}

\bigskip
Finally we describe the dynamics over the Julia set. For quadratic rational maps with a repelling periodic orbit in $\juliavphi \setminus \ponel$ we will construct a model for the dynamics over the convex hull $\widehat{\juliavphi}$
of $\juliavphi$. These models arise from ``abstract $\alpha$-laminations'' as discussed in Section~\ref{section|laminations}.
An abstract $\alpha$-lamination is an equivalence relation in $\RZ$ closely related to the landing pattern of dynamical external rays of complex quadratic polynomials and, they  are a slight generalization of the $\alpha$-laminations defined by McMullen in~\cite{McMullenBook1}. From an abstract $\alpha$-lamination $\bla$ we 
build an inverse system of finite simplicial trees whose limit $\cT^\infty (\bla)$ is naturally endowed with a dynamics $m_2: \cT^\infty (\bla) \to   \cT^\infty (\bla)$ inherited from multiplication by $2$ acting on $\RZ$.

\begin{introtheorem}
\label{ithr:dynamics}
Let $\vphi: \pberl \to \pberl$ be a quadratic rational map which is not simple.
Then exactly one of the following holds:
\begin{enumerate}
\item $\juliavphi \setminus \ponel = \emptyset$
and $\vphi: \juliavphi \rightarrow \juliavphi$ is topologically conjugate to the full shift on two symbols.
\item
$\juliavphi \setminus \ponel$ contains an indifferent periodic orbit 
and  $\vphi: \juliavphi \rightarrow \juliavphi$ is topologically conjugate to a subshift of finite type.

\item 
$\juliavphi \setminus \ponel$ contains a repelling periodic orbit 
and the convex hull $\widehat{\juliavphi}$ of $\juliavphi$ 
is totally invariant under $\vphi$. Let $U_0$ be the unique fixed Rivera domain.
\begin{enumerate}
\item If no critical point eventually maps into $U_0$, then $\vphi: \widehat{\juliavphi} \to  \widehat{\juliavphi}$
is topologically conjugate to the action of $m_2$ on the tree of an abstract $\alpha$-lamination. 

\item If a critical point eventually maps into $U_0$, then 
$\vphi: \widehat{\juliavphi}\to  \widehat{\juliavphi}$
is topologically semiconjugate to the action of $m_2$ on the  tree of an $\alpha$-lamination. The topological semiconjugacy restricts to a topological conjugacy on $\juliavphi$. 
\end{enumerate}
\end{enumerate}
\end{introtheorem}

A more detailed description of the topological conjugacies mentioned in the above theorem are contained in Section~\ref{sec:WithoutRepelling} for cases (1) and (2) as well as in Section~\ref{sec:TopModel} for case (3).

\medskip
We briefly outline the organization of the paper. 

In Section~\ref{section|BerkovichLine} we summarize some basic facts about the field $\L$ and the Berkovich projective line $\pberl$. Section~\ref{section|DynamicsL} is devoted to the basics of iterations of rational maps on $\pberl$.

In Section~\ref{section|QuadraticDynamicsL} we establish two preliminary results about quadratic rational maps: 
a first rough classification according to their Julia periodic orbits in $\pberl \setminus \ponel$ (Proposition~\ref{classification}) and, the number and geometry of fixed Rivera domains is described in Proposition~\ref{star}.

Section~\ref{sec:WithoutRepelling} contains a detailed description of the Julia and Fatou dynamics for maps $\vphi$ without a repelling periodic orbit in 
$\juliavphi \setminus \ponel$. Namely, the case in which $\juliavphi \setminus \ponel = \emptyset$ is covered by Proposition~\ref{attracting.fixed.point}
and the case in which $\juliavphi \setminus \ponel$ is the grand orbit of an indifferent periodic orbit is covered by Proposition~\ref{pro:14}. Some examples are given at the end of this section. 

Section~\ref{sec:maps-with-repelling} is devoted to the more difficult case of maps possessing a repelling periodic orbit in $\juliavphi \setminus \ponel$. Here we study the geometry of these maps by introducing the ``filled Julia set'' and ``dynamical pieces''. We establish Proposition~\ref{repelling.filled} which is the last ingredient needed for the proofs of  Theorems~\ref{ithr:periodic} and~\ref{ithr:fatou}. These proofs are the content of Section~\ref{sec:proofs12}.

In Section~\ref{section|laminations} we introduce and discuss the basic properties of abstract $\alpha$-laminations. Then in Section~\ref{sec:TopModel} we discuss how these laminations are employed to describe the dynamics of quadratic rational maps possesing a repelling periodic orbit in $\juliavphi \setminus \ponel$. That is, we prove Propositions~\ref{bothin} and~\ref{pro:1} which are the last ingredient needed for the proof of Theorem~\ref{ithr:dynamics}.

\medskip
\noindent
{\bf Acknowledgements.} Preliminary versions of this work were the content of a series of lectures which I gave 
during the ``Session r\'sidentielle du CIRM autour de
la dynamique p-adique'' on June 2008 and at Kyoto University on March 2009, I would like to thank the audience for the many comments and suggestions as well as for their interest. I am grateful for related and valuable discussions, at different points of this work, with Laura De Marco, Charles Favre, Mary Rees, Juan Rivera-Letelier, Mitsu Shishikura, Eugenio Trucco and Jean Yves-Briend. The technical assistance of Martin Ugarte and Anibal Medina was of great help for this work, I would like to thank both of them for their generosity.

%%% Local Variables: 
%%% TeX-master: "main.tex"
%%% End: 

%\chapter{Preliminaries}

\section{Berkovich projective line over $\L$}
\label{section|BerkovichLine}
In this section, after introducing the field $\L$,  we summarize notations, definitions, and facts related to the Berkovich projective line over $\L$.
We must warn the reader that this is not a self contained  exposition.
We refer to~\cite{MonographBakerRumely,MonographBerkovich,BourbakiDucros,GazetteDucros,MonographRivera} and the references therein
for a more detailed account on Berkovich spaces. 

\subsection{The field of formal Puiseux series and extensions}
\label{sec:field}
The field of {\sf formal Laurent series} $\laurent$ with coefficients in $\C$ is naturally endowed with the valuation:
$$\lval \sum_{n \geq n_0} c_n \tau^n \rval = \exp(-\min \{ n \in \Z \mid c_n \neq 0 \}).$$
This valuation is non-Archimedean and restricts to the trivial valuation on $\C$.
An algebraic closure of $\laurent$ is the {\sf field of formal Puiseux series} $\puiseux$ (e.g. see~\cite{LibroCasas}). 
The field $\puiseux$ is  the injective limit of 
the fields $\laurentm$, where $m \in \N$ and the injective limit is taken with respect to the obvious inclusions.
The valuation $\lval \cdot \rval$ in $\laurent$ extends uniquely to $\puiseux$ (e.g. see~\cite{LibroCassels}). 
However, $\puiseux$ endowed with this valuation is not complete.
We let $\L$ be the field obtained as the completion of $\puiseux$. It follows that $\L$ is algebraically closed (e.g. see~\cite{LibroCassels}). 
Moreover, each element $\zeta$ of $\L$
can be represented by a series of the form:
$$\zeta = \sum_{n \geq 0} c_n \tau^{\lambda_n},$$
where $\lambda_n \in \Q$ and $\lambda_n \to +\infty$ as $n \to +\infty$. 
The value group $\lval \L^\times \rval$ of $\L$ is $\exp(\Q)$. 

\subsection{Balls and affinoids in $\L$ and in $\ponel$}
Basic geometric objects of $\L$ and/or $\ponel$ such as balls, annuli, and affinoids, will be relevant to our discussion.

\smallskip
Given $r > 0$ and $\z \in \L$ we let
$$B^-_r (\z) = \{\xi \in \L \mid  \lval \xi - \z \rval < r \},$$
and 
$$B_r (\z) = \{\xi \in \L \mid  \lval \xi - \z \rval \leq r \}.$$
If $ r \in \lval \L^\times \rval$, we say that $B^-_r (\z)$ 
is an {\sf open ball} in $\L$  and
 $B_r (\z)$ 
is a {\sf closed ball} in $\L$.
If $r \notin  \lval \L^\times \rval$, then $B^-_r (\z) = B_r (\z)$ is 
an {\sf irrational ball} in $\L$.
However, every ball is both open and closed in the topology of $\L$.
The closed unit ball $B_1(0)$ is a local ring $\mathfrak{O}_\L$, called the {\sf ring of integers of $\L$}, whose maximal
ideal is $\mathfrak{M}_\L = B^-_1(0)$. The {\sf residual field $\mathfrak{O}_\L / \mathfrak{M}_\L$ of $\L$}, usually denoted $\widetilde{\L}$, 
is canonically identified with $\C$.

Given two nested balls $B \subset B^\prime$ of radii $0< r < r^\prime$, respectively, 
we say that $A = B^\prime \setminus B$ is an {\sf annulus} in $\L$ with {\sf modulus} 
$$\mdl A = \log r^\prime - \log r.$$
The modulus of an annulus is preserved under affine transformations.

\medskip
As usual, we  regard $\L$ as a subset of $\ponel$ via the identification of $\z \in \L$ with $[\z:1] \in \ponel$.
Moreover, we let $\infty = [1:0]$. Thus, we may identify $\ponel$ with $\L \cup \{ \infty \}$. 

By definition, a subset of $\ponel$ is called an {\sf open
(resp. closed, irrational) ball} of $\ponel$ if it is an open
(resp. closed, irrational) ball contained in $\L$ or its complement is
a closed (resp. open, irrational) ball contained in $\L$.

A non-empty intersection $A$ of closed  balls of $\ponel$ is called
an {\sf affinoid of $\ponel$}. 
Projective transformations of $\ponel$ map an open (resp. closed, irrational) ball onto
a ball of the same type. Hence, the same holds for affinoids.
A subset $A^\prime$  of $\ponel$ which is the image under a projective transformation of an annulus $A$ of $\L$ is called an {\sf annulus}
of $\ponel$ and we let $\mdl A^\prime = \mdl A$.

\subsection{Berkovich projective line}

The aim of this section is to agree on some terminology and notation
regarding the Berkovich projective line $\pberl$ over $\L$.
 The Berkovich
projective line is a non-metric real tree (e.g., see
Favre-Jonsson~\cite{ValuativeTree}), and may be endowed with the strong or the weak
topology.  With respect to the weak topology, our default topology for $\pberl$,
the Berkovich projective line is compact, sequentially compact and Hausdorff but it is not metrizable (e.g. see~\cite{MonographBakerRumely,MonographRivera}).

\subsubsection{Points, balls, annuli, affinoids and basic open sets of $\pberl$}
\label{sec:PointsBallsAffinoids}
The classical projective line $\ponel$ is a
dense subset of $\pberl$. Points of $\ponel$ in the Berkovich projective line
$\pberl$ are called {\sf type I, or rigid points, or classical points}. Rigid points are endpoints
in the tree structure of $\pberl$.

A set $B \subset \pberl$ which is the closure in $\pberl$ of a closed
(resp. irrational) ball of $\ponel$ is called a {\sf closed
  (resp. closed irrational) ball of $\pberl$}.  Given a ball $B
\subset \pberl$ the topological boundary of $B$ in $\pberl$, denoted
by $\partial B$, is a singleton $\{ \xi \}$.  We say that $\xi$ is {\sf the
  point associated to $B$.}  Points associated to closed balls are
called {\sf type II or rational points} and are ramification points in
the tree structure.  Several closed balls may have the same associated point.
However, given a type II point $\xi$ there exists a unique
closed ball $B$ such that  $\xi$ is the point associated to
$B$ and $B \cap \ponel \subset \L$. The point associated to $\mathfrak{O}_\L$ is called the {\sf Gauss point}.
Points associated to irrational balls are called
{\sf type III or irrational points} and are regular points in the tree
structure.  Every point in $\pberl$ which is not of type I, II or III
is an endpoint in the tree structure and it is called a {\sf type IV
  or singular point}.

Other important geometric objects in $\pberl$ are described as
follows.   An {\sf open ball of $\pberl$} is the complement of a closed ball.
A connected component of the complement of finitely many points
is called a {\sf basic open sets of $\pberl$}.
An {\sf affinoid of $\pberl$} is a non-empty
intersection of finitely many Berkovich closed balls. 

Berkovich balls as well as affinoids and basic open sets are convex
and connected.  The collection of basic open sets form a basis
for the (weak) topology of $\pberl$.
The intersections of balls and affinoids  of $\pberl$ with the
classical projective line $\ponel$ are balls (of the same type) and affinoids  of $\ponel$, respectively. 
Given an affinoid or a basic open set $U$ the convex hull of $\partial U$ is a (finite) simplicial tree 
embeded in $\pberl$ called the {\sf skeleton of $U$} and denoted by $\cA_U$.
There is a canonical retraction $\pi_U : \pberl \to \cA_U$ that maps a point
$x \notin \cA_U$ to the boundary point of the connected component of $\pberl \setminus \cA_U$ which contains $x$.  

An {\sf annulus} $A$ of $\pberl$ is the intersection of two balls $B, B^\prime$ with
distinct boundary points such that $B \cup B^\prime  = \pberl$.  For any annulus $A
\subset \pberl$ we have that $A \cap \ponel$ is an annulus of
$\ponel$. The {\sf modulus of $A$} is, by definition, $\mdl (A \cap
\ponel)$.

\subsubsection{Set of directions  at points}
\label{sec:set-directions-at}
Following Favre and Jonsson~\cite{ValuativeTree}, given $\z \in \pberl$
we say that the set $T_\z \pberl$ formed by the connected components of $\pberl \setminus \{\z\}$
is the {\sf set of directions  at $\z$}.  
This definition coincides with the Favre-Jonsonn tree theoretical definition and with 
the intuition that each ``branch coming out'' of $\z$ corresponds to a direction at $\z$ (c.f.~\cite[Appendix B.6]{MonographBakerRumely}). 

At any given type I  or  IV point $\z$ there is only one
direction, namely  $\pberl \setminus \{\z\}$.  The directions at any
given type II or III point $\z$ are in one-to-one
correspondence with the open  balls with boundary point $\z$ (equivalently, with associated point $\z$).  
At a type III point $\z$, the set of directions
$T_\z \pberl$ consists of two directions which correspond
to the two open irrational balls whose boundary is $\{ \z \}$.  At a
type II point $\z$, the set of directions is isomorphic to
$\ponec$. That is, there exists a natural bijection  from
$T_\z \pberl$ onto $\ponec$ which is unique up to postcomposition by a
(complex) M\"oebius transformation.

For any $\z \in \pberl$, we denote by $D_\z : \pberl \setminus \{ \z \} \rightarrow \tang{\z}$
the map which assigns to a point $\xi$ its direction. 
When $\z$ is the Gauss point, the restriction of $D_\z$ to $\ponel$ is the usual reduction.
Thus, the direction at $\z$ containing a point $\xi \in \pberl \setminus \{ \z \}$ will be denoted
by $D_\z (\xi)$.
We will systematically abuse of notation by simultaneously regarding a direction $D$ at $\z$ as 
an element of $\tang\z$ and as a subset of $\pberl$.

\subsubsection{Hyperbolic space}
We say that $\H_\L :=\pberl \setminus \ponel$ is the {\sf hyperbolic space over $\L$.}  
That is, the points in $\hl$ are of types II, III or IV.
Hyperbolic space is endowed with a metric $\dist_\hl$ which is defined as follows.
Given two distinct type II or III points $\znot, \zone \in \hl$,
$$A(\znot,\zone)=D_\znot (\zone) \cap D_\zone (\znot)$$ is an annulus.
The {\sf hyperbolic distance between $\znot$ and $\zone$} is defined by
$$\dist_\hl (\znot,\zone) = \mdl A(\znot,\zone).$$
The  unique continuous extension of $\dist_\hl$  to $\hl \times \hl$
endows hyperbolic space with a metric which we also denote by $\dist_\hl$.
With this metric $\hl$ is a $\R$-tree.
The classical projective line $\ponel$ is the boundary at infinity of $\hl$.
For future reference we state and prove a simple fact.

\begin{lemma}
  \label{fine.sequence}
  Let $\{X_\ell \}_{\ell \in \N}$ be a decreasing sequence of closed subsets of
$\pberl$ such that $\partial X_\ell$ is a singleton contained in $\hl$
for all $\ell \in \N$.
If $$\sum_{\ell \in \N} \dist_\hl (\partial X_\ell, \partial  X_{\ell+1}) = \infty,$$
then 
$$X=\bigcap_{\ell \in \N} X_\ell$$
is a singleton contained in  $\ponel$.
\end{lemma}

\begin{proof}
The compactness of $\pberl$ implies that $X$ is not empty. 
Without loss of generality assume that, $\partial X_\ell \neq \partial
X_{\ell+1}$ for all $\ell \in \N$.  Given $\ell \in \N$, let
$B_{\ell+1}$ be the  ball which is the complement of the
direction at $\partial X_{\ell+1}$ containing $\partial X_\ell$.  It
follows that $\partial B_{\ell+1} = \partial X_{\ell+1}$ and
$B_{\ell+1} \supset X_{\ell+1}$. Moreover, $Y= \cap_{\ell} B_{\ell+1}$
is contained in $\ponel$.  Otherwise, for any $\z \in Y \cap \hl$ we
would have that $\dist_\hl (\z, \partial B_1) > \dist_\hl (\partial
X_{\ell+1}, \partial X_1) \rightarrow \infty$ as $\ell \rightarrow
\infty$. Therefore, $Y$ ($ \supset X$) is a singleton, since $Y$ is connected and
$\ponel$ is totally disconnected.
\end{proof}

%\subsection{Galois group}
% Rivera-Letelier [48, §2.2] and Thuillier [94, §3.2.3] have given algebraic constructions
% of multiplicities using ranks of modules1. Rivera-Letelier [81, §2]
% 1The constructions of multiplicities given in [48] and [94] generalize to higher dimensions
% and are based on the following fact: if f : Y → X is a finite flat map between
% smooth, separable Berkovich spaces and if y ∈ Y , then there are an affinoid neighborhood
% W of y and an affinoid neighborhood V of f(y) such that f(W) = V and y is the only
% preimage of f(y) in W, such that the affinoid algebra AW is free of finite rank over AV .
% This rank is the multiplicity.
% 249
% 250 9. MULTIPLICITIES
% and [84, §12] has given a different construction based on local mapping properties
% of ϕ(T); he calls the multiplicity the “local degree”.

\subsection{Action of rational maps on $\pberl$}
The action of a rational map $\vphi \in \L (\z)$ on $\ponel$ extends
continuously to an action on $\pberl$.  The extended action, which we
also denote by $\vphi : \pberl \rightarrow \pberl$, preserves the
Berkovich type of the points.  Moreover, it is an open map. Furthermore, at any given
$\xi \in \pberl$, the map $\vphi : \pberl \rightarrow \pberl$ has a
well defined (local) degree $\deg_\xi \vphi$.  
The number of preimages of any point $\z \in \pberl$,
under $\vphi$ (counting multiplicities), coincides with the degree of $\vphi :
\ponel \rightarrow \ponel$. (See~\cite[sections~2.3, 2.4, 9.1]{MonographBakerRumely}).

Also, for any rational map $\vphi: \pberl \to \pberl$,
the preimage of an affinoid (resp. basic open set) $A$ is the finite union of affinoids (resp. basic open sets) $A_1, \dots, A_k$ and
$\vphi: A_j \to A$ has a well defined degree $d_j$ for all $j$.

\subsubsection{M\"obius transformations}
From the above, it follows that the 
group of linear fractional transformations $\operatorname{PSL}(2,\L)$
acts on $\pberl$. The stabilizer of the Gauss point is 
$\operatorname{PSL}(2, \mathfrak{O}_\L)$.  The action of $\operatorname{PSL}(2,\L)$ is transitive on
type II points (see~\cite[Section~2.3]{MonographBakerRumely}).

\subsubsection{The action on the space of directions}
\label{ends-s}
A rational map  $\vphi : \pberl \rightarrow \pberl$ also
acts on the set of directions. 
More precisely, according to Rivera~\cite{riveratesis03} (compare with~\cite[Section~9.3]{MonographBakerRumely}), given a direction $D$ at a type II or III point $\z$ the following
hold:

\begin{enumerate}
\item
There exists a closed ball $B \subset D$
such that $\vphi (D \setminus B_0)$ is an annulus 
for all closed balls $B_0$ for which  $D
\supset B_0 \supset B$.

\item
There exists $k \in \N$ such that $\vphi: D \setminus B
\rightarrow \vphi(D \setminus B)$ is $k$-to-one.

\item
$\vphi (D \setminus B)$ is contained in one direction $D^\prime$ at $\vphi(\z)$. Moreover,
$D' \setminus \vphi (D \setminus B)$ is a closed ball.

\item
$\vphi(D) = D^\prime$ or $\vphi (D) = \pberl$. In the former case we say that $D$ is a {\sf good direction}, in the latter 
we say that $D$ is a {\sf bad direction}.

\item
For all $\xi \in  D \setminus B_0$, 
$$\dist_\hl (\vphi(\z), \vphi(\xi)) = k \cdot \dist_\hl (\z, \xi).$$
\end{enumerate}

It follows that $T_\z \vphi: D \mapsto D^\prime$ defines the {\sf tangent map}
$T_\z \vphi : T_\z \pberl \rightarrow T_{\vphi(\z)} \pberl$.
We say that $k$ is {\sf the degree of $T_\z \vphi$ at $D$}, and denote 
this number by $\deg_D T_\z \vphi$.

\begin{enumerate}
\item[(6)]
There exists a number $n \geq 0$ such that
any point in $\pberl \setminus T_\z \vphi (D)$ has $n$ preimages
(counting multiplicities) in $D$ and any point in $T_\z \vphi (D)$ has
$n+ \deg_D T_\z \vphi$ preimages (counting multiplicities) in $D$.
Following Faber~\cite{FaberRamificationI}, $n$ is called the {\sf surplus multiplicity}. 

\item[(7)]
At a type III  point $\z$, the tangent map $T_\z \vphi$ is a bijection. The degree of $T_\z \vphi$ 
is the same in both directions and coincides with the degree of
$\vphi$ at $\z$

\item[(8)]
At a type II  point $\z$, the tangent map $T_\z \vphi : T_\z \pberl \rightarrow
T_{\vphi(z)} \pberl $ is a rational map (in the correponding $\ponec$--structures) of degree at least $1$.  The degree of $T_\z \vphi$ at $D$
as a rational map over $\C$ and the degree of $T_\z \vphi$ in the direction $D$
(as defined above) coincide.  Moreover, the local degree of $\vphi$ at
$\z$ coincides with the degree of $T_\z \vphi$ as a rational map over $\C$.
\end{enumerate}

It is not difficult to show that points $x \in \pberl$ such that $\deg_x \varphi \ge 2$ are contained in the convex hull of the rigid critical points. For a quadratic rational map, it follows that the segment joining the two rigid critical points form the set of points where $\vphi$ is not locally injective (see~Lemma~\ref{degree-one-l}). 

\subsubsection{Counting critical points}
Since the residual field of $\L$ has characteristic zero the degree of a map on a ball is closely related to the number of critical points in it.
(Compare with~\cite[Appendix A.10]{MonographBakerRumely} and~\cite{FaberRamificationI})

\begin{lemma}
\label{riemann-hurwitz}
  Let $B \subset \pberl$ be a ball and $\vphi$ a rational map in $\L(\z)$. 
If $\vphi(B)$ is a ball, then the number of (rigid) critical points of
$\vphi$ in $B \cap \ponel$ is (counting multiplicities) $\deg_{\partial B} \vphi - 1$.
\end{lemma} 

\begin{proof}
Note that by the above discussion
$\deg_{\partial B}$ coincides with the degree of $\vphi: B \to \vphi(B)$.
  We may restrict to the case in which $B \cap \ponel$ and $\vphi(B) \cap \ponel$
are balls which contain the origin in $\L$. 
It follows that the Newton polygon (e.g. see~\cite{LibroCassels}) of $d\vphi/d\z$ is obtained from that
of $\vphi$ by translation towards the left. Thus,  substracting one to the number of zeros of $\vphi$ in $B\cap \ponel$ we 
obtain the number of zeros of $d\vphi/d\z$ in $\vphi$ in $B\cap \ponel$.
\end{proof}

%\subsection{Action on hyperbolic space}

%%% Local Variables: 
%%% TeX-master: "main.tex"
%%% End: 

\section{Rational dynamics over the Berkovich projective line}
\label{section|DynamicsL}
\subsection{The Julia and Fatou sets}
Let $\vphi : \pberl \rightarrow \pberl$ be a rational map of degree  $\deg \vphi \geq 2$. 
A point $\z \in \pberl$ belongs to the {\sf Julia set $\juliavphi$} if for all  neighborhoods  $U$ of
$\z$ we have that $\cup \vphi^n(U)$ omits at most two points. The complement of the Julia set
is the {\sf Fatou set $\fatouvphi$} (see~\cite[Section~10.5]{MonographBakerRumely}).

These definitions agree with the original ones by Hsia.
More precisely, $\z \in \ponel$ lies in the Julia set defined by Hsia~\cite{Hsia2000} if and only if $\z \in \juliavphi \cap \ponel$.

\subsection{Periodic points}
\label{sec:periodic.points}
Given a rational map $\vphi \in \L (\z)$ we say that $\znot \in
\pberl$ is {\sf periodic} if $\vphi^n (\znot) = \znot$ for some $n \geq 1$.
The minimal $n$ such that $\vphi^n (\znot) = \znot$ is called the
{\sf period} of $\znot$.  

When $\znot \in \ponel$ is a period $p$ rigid  point, the
{\sf multiplier} of $\znot$ is, by definition, $$\lambda = \frac{d\vphi^p}{d\z} 
(\znot) \in \L.$$  As usual the periodic point $\znot$ is called {\sf attracting
(resp. neutral or indifferent, repelling)} according to whether the multiplier
$\lambda$  belongs to $\mathfrak{M}_\L$ (resp. $\mathfrak{O}_\L \setminus \mathfrak{M}_\L$, $\L \setminus
\mathfrak{O}_\L$).  Type I repelling periodic points belong to the Julia set
and type I non-repelling periodic points belong to the Fatou set.
According to Benedetto~\cite{benedettothesis},
every rational map  $\vphi \in \L(\z)$ has at least one non-repelling fixed point in $\ponel$.

Following Rivera-Letelier, 
a non-rigid periodic point $\z \in \pberl$ of period $p$ is called {\sf repelling} 
if the local degree of $\vphi^p$ at $\z$ is at least $2$. Otherwise, is called {\sf neutral or 
indifferent}. 
Non-rigid repelling periodic points must be type II points. Also,
non-rigid periodic points which belong to the Julia set must be type II points (see~\cite{rivera03}).
The following result is the content of a personal communication 
with Juan Rivera-Letelier~\cite{Riv06Personal}.

\begin{theorem}[Rivera-Letelier]
   \label{CharacterizationJuliaPeriodicPointsTheorem}   
Let $\z$ be a  non-rigid period $p$ periodic point of a rational map $\vphi:\pberl \rightarrow \pberl$ of degree at least $2$.
The point $\z$ belongs to $\juliavphi$ if and only if one of the following hold:
\begin{enumerate}
\item $\deg_\z \varphi^p \geq 2$.
\item $\deg_\z \varphi^p =1$ and there exists a bad direction at $\z$ with infinite forward orbit
under $T_\z \vphi^p$.
\end{enumerate}
\end{theorem}

\begin{proof}
Since every non-rigid repelling periodic point belongs to $\juliavphi$ 
(see~\cite[Section~10.7]{MonographBakerRumely}),
we may assume that $p=1$ and  $\deg_\z \varphi =1$. Note that  $T_\z \vphi$ is a bijection. Thus a bad direction $D$
has infinite forward orbit  under $T_\z \vphi$  if and only if $D$  has infinite backward orbit  ($\{ D' \in \ponec \mid T_\z \vphi^n (D') =D \mbox{ for some } n\geq 0\}$).

In the case that there exists a bad direction $D$ at $\z$ which has infinite backward orbit,  every
neighborhood $U$ of $\z$ contains a direction $D'$ in the backward orbit of a bad direction. Hence, there exists $n \geq 0$ such that
$\pberl = \vphi^n (D') \subset \vphi^n (U)$. Therefore $\z \in \juliavphi$.

Now we consider the case in which  every bad direction has finite backward orbit.
That is, every bad direction is periodic.
Thus, after 
removing appropriate closed balls in the orbit of the bad directions, we obtain a neighborhood $U$
of $\z$ which is invariant under $\vphi$ (i.e. $\vphi(U)=U$). Hence, $\z \in \fatouvphi$.
\end{proof}

It follows that non-rigid indifferent periodic points may belong to the Julia set or the Fatou set.
However, indifferent periodic points in the Julia set must be of type II.

The {\sf basin of a periodic orbit $\cO$} is the interior of the set of points $\xi$ such that
the omega limit of $\xi$ is $\cO$. Non-rigid periodic orbits contained in the Julia set
always have non-empty basin. In fact,  
given a point $\z$ of type II  there are uncountably many directions $D$ at $\z$ such that its image
under $T_\z \vphi^n$ is a good direction, for all $n \ge 0$, since there are only finitely many bad directions. 
If $\z$ belongs to a periodic orbit $\cO$, say of period $p$, which lies
in $\juliavphi$, then for uncountably many directions $D$ we have that $\{ \vphi^{np}(D) \}$ are pairwise distinct directions at $\z$.
Thus, each of these directions is a component of the Fatou set contained in the basin of $\cO$.

\subsection{Periodic Fatou components}
\label{sec:PeriodicFatou}   
   A connected component $U$ of the Fatou set $\fatouvphi \subset
   \pberl$ is called a {\sf Fatou component}.  Since $\vphi$ is an open map, the image of a Fatou component
   is again a Fatou component.

If a Fatou component $U$ contains an attracting periodic point $\xi$,
then $U$ is a component of the basin of attraction
of the orbit of $\xi$. In this case, $U$ i\textsc{}s called the {\sf immediate basin of $\xi$}.

  We say that a period $p$ Fatou component $U$ is a {\sf Rivera domain}
 if $\vphi^p : U \rightarrow U$ is a bijection.

  \begin{theorem}[Rivera~\cite{riveratesis03}]
	Let $\vphi \in \L(\z)$ be a rational map of degree at least $2$. Then a
 periodic Fatou component is either the immediate basin of an attracting periodic point or
a Rivera domain.
  \end{theorem}

   \begin{theorem}[Rivera~\cite{riveratesis03}]
  \label{rivera.domain}     
	Let $\vphi \in \L(\z)$ be a rational map of degree at least $2$.
If $U$ is a Rivera domain, then $U$ is a component of the interior of an affinoid and every point in $\partial U$ is a
periodic orbit in $\juliavphi$.

	If $U$ is a component of an immediate basin of an attracting
        periodic orbit, then $U$ is an open ball or $\partial U$ is a Cantor set. 
\end{theorem}

It follows if $U$ is a Rivera domain, then the retraction
$\pi_U: U \to \cA_U$ conmutes
with $\vphi$ (see Section~\ref{sec:PointsBallsAffinoids} for the definition of 
$\pi_U$ and $\cA_U$). In particular, $\vphi: \cA_U \to \cA_U$ has finite order.

Let $\vphi \in \L(\z)$ be a degree $d \geq 2$ rational map. Assume
that $U$ is a fixed Rivera domain for $\vphi$.  Let
$\xi \in \partial U$ and denote by $D$ the direction at $\xi$ that contains
$U$.
We let $\eta (\xi) \geq 0$ be the multiplicity of $D$ as a fixed point of $T_\xi
\vphi$.

\begin{theorem}[Rivera~\cite{riveratesis03}]
 \label{residue.formula}  
Let $\vphi \in \L(\z)$ be a degree $d \geq 2$ rational map. 
Assume that $U$ is a fixed Rivera domain for $\vphi$. 
Let $N \geq 0$ be the number of fixed points in $U \cap \ponel$. Then
$$N = 2 + \sum_{\xi \in \partial U, \vphi(\xi)=\xi} (\eta(\xi) - 2).$$
\end{theorem}

Although the above results were originally proved in the context of $p$-adic dynamics, the techniques 
apply to the context of dynamics over $\L$.

%%% Local Variables: 
%%% TeX-master: "main.tex"
%%% End: 

%\chapter{Dynamics of quadratic rational maps}

\section{Julia periodic orbits in $\hl$ and fixed Rivera domains}
\label{section|QuadraticDynamicsL}
The aim of this section is to give a rough classification  of quadratic rational maps in $\L(\z)$ according to their Julia non-rigid periodic orbits
and to describe the geometry of fixed Rivera domains. 

\subsection{Julia periodic orbits in $\hl$}
As a first step towards proving Theorem~\ref{ithr:periodic} we study periodic orbits in $\hl$.

\begin{proposition}
  \label{classification}
  Let $\vphi: \pberl \rightarrow \pberl$ be a quadratic rational map which is not simple. Then one and only one of the following holds:
  \begin{enumerate}
    \item 
      There are no periodic points in $\juliavphi \cap \hl$.
      
    \item
      There exists exactly one indifferent  periodic  orbit $\cO$ in $\juliavphi \cap \hl$ and $\cO$ is the unique periodic orbit in $\juliavphi \cap \hl$.

    \item 
      There is at least one repelling periodic orbit in $\juliavphi \cap \hl$. 
  \end{enumerate}
\end{proposition}

We state and establish two necessary lemmas before proving this
proposition.   

\medskip
Recall  that 
a {\sf critical point} of a rational map $\vphi: \pberl \rightarrow \pberl$ 
is a rigid point $\z \in \ponel$ where the derivative of 
$\vphi: \ponel \rightarrow \ponel$ vanishes.
A {\sf critical value} is the image of a critical point under $\vphi$.

\begin{lemma}
\label{degree-one-l}
Let $\vphi \in \L(\z)$ be a quadratic rational map.
  Suppose that $\xi_0 \in \pberl$ and  $\vphi (\xi_0) = \xi_1$.
  Then the following are equivalent:
  \begin{enumerate}
    \item $\deg T_{\xi_0} \vphi  =1$.

    \item  The critical points of $\vphi$ belong to the same direction $D_0$ in  $T_{\xi_0} \pberl$.

    \item  The critical values of $\vphi$ belong to the same direction $D_1$ in  $T_{\xi_1} \pberl$.
  \end{enumerate}
If the above holds, then  $T_{\xi_0} \vphi (D_0) = D_1$ and $D_0$ is a bad direction.
\end{lemma}

\begin{proof}
The lemma holds trivially when $\xi_0$ is a type I or IV point.
Thus, suppose that $\xi_0$ is a type II or III point.

(1) $\implies$ (2).
  Assume that $\deg T_{\xi_0} \vphi =1$. Counting preimages of a typical
point with the aid of Section~\ref{ends-s} it follows that  there exists one and
  only one bad direction at $\xi_0$, say $D_0$.  Since good directions
are mapped injectively onto their image, 
the critical points of $\vphi$ must belong to $D_0$.

(2) $\implies$ (3).
   Now assume that the two critical points belong to the direction
  $D_0$ at $\xi_0$.  In view of Section~\ref{ends-s}, if $\z \notin D_1 =
  T_{\xi_0} \vphi (D_0)$, then $\z$ has at most one preimage (counting multiplicities) in $D_0$.
  In particular, $\z$ is not a critical value.

(3) $\implies$ (1).
  Assume that both critical values lie in the same direction $D_1$ at
  $\xi_1$ and proceed by contradiction. 
 If $\deg T_{\xi_0} \vphi =2$, then every direction at $\xi_0$ is a good direction.
Moreover, there are exactly two directions $D, D^\prime$ of degree two under $T_{\xi_0} \vphi$.
By Lemma~\ref{riemann-hurwitz}, each must contain a critical
point. Then both $D$ and $D^\prime$ are mapped under $\vphi$ onto $D_1$ with degree $2$.
 This is clearly impossible since points in $D_1$ do not have $4$ preimages.
\end{proof}

\begin{remark}
\label{degree.one.julia}
{\em   Suppose that there exists an open ball $D$ contained in the Fatou set
that contains both critical values. Then $\deg T_\xi \vphi  =1$, for all
$\xi \in \juliavphi$.}
\end{remark}

\begin{lemma}
\label{indifferent.critical}
Let $\vphi \in \L(\z)$ be a quadratic rational map.
  If $\vphi$ has an indifferent periodic orbit $\cO \subset \juliavphi \cap \hl$,
then both critical values of $\vphi$ belong to the same  Fatou component $D$ which is an open ball
 contained in the basin of $\cO$. Moreover, $\partial D \subset \cO$.
\end{lemma}

\begin{proof}
  Let $\xi_0$ be a point of the periodic orbit $\cO$. Say that the
  period of $\cO$ is $p$.  Consider a bad direction $D_0$ at $\xi_0$ for $\vphi^p$
with infinite forward
  orbit under $T_{\xi_0} \vphi^p$ (see Theorem~\ref{CharacterizationJuliaPeriodicPointsTheorem}). 
By Lemma~\ref{degree-one-l}, there is exactly one
bad direction at $\vphi^n(\xi_0)$, for all $n \ge 0$. This guarantees the 
existence of a smallest integer $k \geq 1$ such that $T_{\xi_0} \vphi^{n+k} (D_0)$ is a
  good direction for all $n \geq 0$.  Again by Lemma~\ref{degree-one-l}, 
$D = T_{\xi_0} \vphi^{k} (D_0)$ contains both critical values.
Therefore, $\{ \vphi^{np}
  (D) \}_{n \geq 0}$ is a collection of pairwise distinct directions
  at $\vphi^{k}(\xi_0)$. Hence, for all $\z \in D$, the 
  forward orbit of $\z$ under $\vphi^p$ converges to
  $\vphi^{k}(\xi_0)$. It follows that $D$ is a Fatou component 
 contained in the basin of $\cO$. Moreover,  $\partial{D} \subset \cO$.
\end{proof}

\begin{proof}[Proof of Proposition~\ref{classification}]
  Suppose that  (1) does not hold.  The presence of an indifferent periodic orbit in $\juliavphi \cap \H_\L$
implies that all the periodic orbits in  $\juliavphi \cap \H_\L$
are indifferent, by Remark~\ref{degree.one.julia} and  Lemma~\ref{indifferent.critical}. 
That is, (3) does not hold. Moreover, also
by Lemma~\ref{indifferent.critical}, there is at most one  indifferent periodic orbit in $\juliavphi \cap \H_\L$.
\end{proof}

\subsection{Starlike Rivera domain}
Here we establish that a fixed Rivera domain $U$ of a quadratic rational map $\vphi$ is an open ball or a
starlike domain. Recall that a starlike domain $U$ is 
the interior of an affinoid such that  $\cA_U$ contains at most one topological branched point and a
 unique fixed point. In this case, the action of $\vphi: \cA_U \to \cA_U$
is a ``rotation'' around this fixed point.

\begin{proposition}
\label{star}
  Let $\vphi \in \L(\z)$ be a quadratic rational map which is not simple. 
Then $\vphi$ has at most two fixed Rivera domains.

If $\vphi$ has exactly one fixed Rivera domain $U$, then $U$ is an open ball or a starlike domain. 
Moreover, the points in $\partial U$ form one periodic orbit.

If $\vphi$ has two  fixed Rivera domain,
then both are open balls which have the same boundary point.
\end{proposition}

\begin{proof}
Assume that a fixed Rivera domain $U$ is an open ball. Then $\{ \xi\} = \partial U$
is a fixed point. The degree of $\vphi$ at $\xi$ must be $1$, for
otherwise $\vphi$ would be simple. Since a M\"obius transformation
 $T_\xi \vphi: \ponec \rightarrow \ponec$ has at most two
fixed points, $\xi$ is in the boundary of at most two fixed Rivera
domains.  In view of Proposition~\ref{classification} the boundary of any
fixed Rivera domain must be $\{ \xi \}$, since $\{ \xi \}$ is the unique periodic orbit
of the Julia set in $\hl$.

If a fixed Rivera domain $U$ is not an open ball, then $\partial U$
consists of at least two points.  By Proposition~\ref{classification}, if
there is more than one orbit in $\partial U$, then all such orbits
must be repelling. Note that there is exactly one component $B$ of
$\pberl \setminus U$ that contains both critical points.  If a
periodic orbit in $\partial U$ is repelling, then this orbit must
contain a point of $B$. But $\partial U \cap B$ is a singleton.
Hence, there is only one such orbit.

We now show that a fixed Rivera domain $U$ which is not an open ball
has a starlike structure.  That is, we let $\cA_U$ be the convex hull
of $\partial U$ and proceed to show that it contains at most one
branched point and at most one fixed point $\veta$. Observe that $\vphi : \cA_U
\rightarrow \cA_U$ is a tree isomorphism and an isometry in the hyperbolic
metric of $\hl$.  Let $q$ be the period of
the periodic orbit  $\cO= \partial U$.  By
Theorem~\ref{residue.formula}, the domain $U$ contains no rigid periodic
point of period $k$ such that $1 < k < q$.

For us the vertices of $\cA_U$ consist of the topological branched points together 
with the endpoints of $\cA_U$.
We claim that every vertex of $\cA_U$ has period $1$ or $q$. In fact,
we proceed by contradiction and let $v$ be a vertex of period $k$ with
$1 < k < q$. Endow the vertices of $\cA_U$ with the partial order inherited from the tree 
structure where $v$ is a minimal element (i.e., the root).  Let $w$ be a maximal element among
the vertices of $ \cA_U$ with period $k$. It follows that $\vphi^k$
fixes at most one branch of $\cA_U$ at $w$.  Since $w$ is not an endpoint of $\cA_U$, we have that
$T_w \vphi^k$ must have finite order strictly greater than $1$ in order to permute the directions of the at least two branches
at $w$ which are not fixed by $\vphi^k$. That is, $T_w \vphi^k$ is conjugate to
$z \mapsto \eta z$ where $\eta \neq 1$ is a root of unity. Hence,
there exists a direction $D$ at $w$, which is disjoint from $\cA_U$,
that is fixed under $\vphi^k$. Without loss of generality, assume that $D \cap \ponel$ is the open unit ball.
In $D \cap \ponel$, we have that $\vphi^k (\z) = a_0 + a_1 \z + h.o.t$ where $\val{a_1-\eta} <1$  and $\val{a_0} <1$.
Therefore, $\vphi^k (\z)-\z$ has exactly one solution in $D \cap \ponel$ which is a contradiction with  the last sentence of the previous paragraph.

Finally, consider a point $\xi \in \partial U$ and let 
$v$ be a fixed point of $\vphi$ in $\cA_U$ such that $(v,\xi]$ contains no other fixed point. 
Let $k \geq 1$ be the smallest integer  such that 
$\vphi^k ((v,\xi]) \cap (v,\xi] \neq \emptyset$.
Then $k \geq 2$ since $(v,\xi]$  is fixed point free.
In fact, $k = q$, for otherwise, there is a vertex in $\cA_U$ of period $1 < k < q$. 
It follows that $\cA_U$ is either an interval (and $\xi$ has period $2$) 
or a tree with $v$ as the unique vertex which is not an endpoint.
\end{proof}

%\section{Dynamics}

\section{Maps without a repelling periodic orbit in $\H$}
\label{sec:WithoutRepelling}
The description of the dynamics in the first two cases of Proposition~\ref{classification} 
is rather simple. In Section~\ref{sec:attracting} we deal with maps 
not having Julia periodic orbits in $\hl$  and, in Section~\ref{sec:indifferent}, we consider maps with a Julia indifferent periodic orbit. At the end of Section~\ref{sec:indifferent} the reader may find examples of maps with indifferent periodic orbits in their Julia set. 

\subsection{Not simple with attracting fixed point}
\label{sec:attracting}
Let us first state what occurs in the absence of Julia periodic orbits in
$\hl$.
\begin{proposition}
    \label{attracting.fixed.point}
    Let $\vphi$ be a quadratic rational map which is not simple. 
Then the following are equivalent:
\begin{enumerate}
  \item $\juliavphi \subset \ponel$.

  \item $\vphi$ has an attracting fixed point $ \z \in \ponel$.

  \item There are no periodic orbits of $\vphi$ in $\juliavphi \cap \hl$.
\end{enumerate}

    Moreover, if (1)--(3) hold, then  the dynamics over $\juliavphi$ is topologically conjugate to the one-sided shift on two symbols
and $F(\vphi)$ is the inmediate basin of attraction of $\z$.
  \end{proposition}

  \begin{proof}
    First we show, by contradiction, that (3) implies (2).
    If there is no attracting fixed point in $\ponel$, then there
    exists an indifferent fixed point $\zp$ in $\ponel$
    (see~\cite{benedettothesis}). Therefore,
    $\zp$ belongs to a Rivera domain $U_{\zp}$ whose
    boundary, according to Theorem~\ref{rivera.domain}, contains a
    periodic orbit in $\hl \cap \juliavphi$.

    Now  assume that (2) holds: $\vphi$ has an attracting fixed point $\z \in
    \ponel$. We simultaneously show that (1) holds ($\juliavphi \subset \ponel$) and
that the dynamics is topologically equivalent to the one-sided shift on two symbols.  

Let $B$ be  the union of all open  balls $B^\prime$
    containing $\z$ such that $\vphi(B^\prime) \subset B^\prime$.  It
    follows that $B$ is a ball such that $\vphi (B) \subset B$.  Now
    let $\xi$ be the point associated to $B$. Since $D_{\xi}(\z)$ is contained in the
    basin of $\z$, we have that $\vphi(D_{\xi} (\zeta)) =
    D_{\vphi (\xi)} (\zeta)$. By Schwartz Lemma, $\xi$ is not
    a fixed point, for otherwise, $\vphi:D_{\xi} (\zeta) \rightarrow D_{\xi} (\zeta) $
would have degree $2$ and we would also have that  $\deg T_{\xi} \vphi =2$, but
  $\vphi$ is not simple. 
Hence $\xi$ is contained in a direction $D$ at $\vphi(\xi)$.

We claim that $\vphi^{-1} (D)$ is the disjoint union of two open
    balls $D_0, D_1$ compactly contained in $D$, each of which is mapped onto $D$
    in a one-to-one fashion. 
In fact, given a direction $D^\prime$ at $\xi$ such that $T_\xi \vphi (D^\prime) = D$
we have that  $\deg_{D^\prime}  {T_\xi \vphi}   =1$ and $\vphi(D^\prime) =D$;
for  otherwise, $\vphi (\pberl \setminus D^\prime) \subset \pberl
    \setminus D^\prime$ which, by Section~\ref{ends-s}, contradicts the definition of $B$.
    Therefore $\vphi^{-1} (D)$ is the disjoint union of two open
    balls $D_0, D_1$, compactly contained in $D$, and each of them is mapped onto $D$
    in a one-to-one fashion.  

Now we establish the conjugacy between $\juliavphi$ and the full shift on two symbols.
Observe that $\juliavphi \subset D_0 \cup D_1$. 
 Given $\varsigma \in \juliavphi$, let
$i_n (\varsigma)=0$ or $1$ if $\vphi^n (\varsigma) \in D_0$ or $D_1$,
respectively. It follows that itinerary map 
$$\juliavphi \ni \varsigma \mapsto (i_n (\varsigma))_{n \geq 0} \in \{0,1\}^{\N \cup \{0 \}}$$ 
is continuous. To check that the itinerary map is one-to-one and onto, given
an infinite symbol sequence $\mathbf{i}=(i_0, i_1, \dots )$, for each $n \geq 0$, let
$$ D_n (\mathbf{i}) = \{ \varsigma \mid \vphi^k (\varsigma) \mbox{ for all } 0 \leq k \leq n \}.$$
Now $\vphi^{n+1}$ maps the annulus $D_n(\mathbf{i}) \setminus D_{n+1}(\mathbf{i}) $ 
isomorphicaly onto $D \setminus D_{i_{n+1}}$, therefore 
$$\dist_\H (\partial D_n(\mathbf{i}),\partial D_{n+1}(\mathbf{i})) = 
\dist (\partial D, \partial D_{i_{n+1}}).$$
From Lemma~\ref{fine.sequence} we conclude that the
intersection of the 
nested sequence $D_0 (\mathbf{i}) \supset D_1 (\mathbf{i}) \supset \cdots$ 
is a singleton contained in $\ponel$. It follows that
the itinerary map is a conjugacy and that $\juliavphi \subset \ponel$.
Moreover,  $F(\vphi)= \pberl \setminus \juliavphi$ is connected and invariant.
Therefore, $F(\vphi)$ is the inmediate basin of $\z$.

Since (1) implies (3) is obvious, we have proven the proposition.
  \end{proof}

\subsection{Indifferent periodic orbit in $\hl$}
\label{sec:indifferent}
Now we describe the situation in the presence of an indifferent orbit in
the Julia set. 

Before stating our result let us denote the {\sf grand orbit} of a periodic orbit
$\cO$ by $GO(\cO)$. Also recall that the {\sf renewal shift} is the (non-compact) subshift of
the full shift on two symbols obtained after removing the grand orbit of a fixed
point from the full shift.

\begin{proposition}
\label{pro:14}
  Suppose that $\vphi$ has an indifferent periodic orbit 
$\cO$ of period $q \geq 1$ in $\juliavphi \cap \hl$. Then the following statements hold:

\begin{enumerate}
\item
$\juliavphi \cap \ponel$ is topologically conjugate to the renewal shift.  

\item
    $\juliavphi \cap \hl = GO (\cO)$.

  \item
    The dynamics over $\juliavphi$ is topologically conjugate to the subshift 
$\Sigma_q$
of finite type in $2q$ symbols $\{ X_0 , \dots X_{q-1} , Y_0 \dots Y_{q-1} \}$ (subscripts $\mod q$) with the following allowed transitions:

\begin{alignat*}{2}
X_j &\rightarrow X_{j+1}, \quad &\text{for}\ j=0, \dots, q-1. \\
X_{q-1} &\rightarrow Y_j, &\text{for}\ j=0, \dots, q-1. \\
Y_j &\rightarrow X_j, &\text{for}\ j=0, \dots, q-1. \\
Y_0 &\rightarrow Y_j, &\text{for}\ j=0, \dots, q-1.
\end{alignat*}

\noindent
Under the topological conjugacy, $\cO$ corresponds to the orbit of $(X_0 X_1 \dots X_{q-1})^{\infty}$.

\item
There exists at least one fixed Rivera domain. 
The boundary of every Rivera domain is $\cO$.
If $U$ is a periodic Fatou component, then $U$ is a fixed Rivera domain.

\item
If $U$ is a Fatou component which is not eventually periodic, then $U$ is an annulus or an open ball.
Moreover, there exists $n_0 \ge 0$ such that $\vphi^n(U)$ is an open ball with $\partial \vphi^n(U) \subset \cO$ for all $n \ge n_0$.
Furthermore, $U$ is contained in the basin of $\cO$.
\end{enumerate}
\end{proposition}

\begin{remark}
{\em  The above subshift factors onto the full shift on the symbols $\{ X,
  Y \}$, where the factor map $h$ is the one obtained by replacing the
  symbols $X_j$ with $X$ and $Y_j$ with $Y$, for all $j$.  The factor
  map $h$ is one-to-one except over the backward orbit of $X^{\infty}$
  where is $q$-to-one.  In fact, note that the forward (periodic)
  orbit of $(X_0 X_1 \dots X_{q-1})^{\infty}$ is $h^{-1}(X^{\infty})$.}
\end{remark}

  \begin{proof}
Let $V$ be the Fatou component that contains the critical values (see
Lemma~\ref{indifferent.critical}).  Enumerate $\cO$ by $\xi_0, \xi_1,
\dots, \xi_{q-1}$ respecting dynamics and so that $\partial V =
\xi_1$. Subindices will be modulo $q$ throughout this proof.
Note that $V$ is contained in the basin of $\cO$.

Let $U$ be the union of all fixed Rivera domains.  Recall that
$U$ is an open ball, or a union of two
open balls with the same boundary point,  or a starlike
domain (Proposition~\ref{star}).  In the first pair of cases, $q =1$ and, in the last, $q
>1$.  In particular, $L=\pberl \setminus U$ is a ``sphere'' $B_0$ or
the disjoint union of $q$ closed balls $B_0, \dots, B_{q-1}$ labeled
such that $\{\xi_j\} = \partial B_j$.  Now let $W$ be the direction at $\xi_0$
containing the critical points and observe that $\vphi(W) = \pberl$.
In particular $\vphi(B_0) = \pberl$ and $\vphi: B_0 \setminus W
\rightarrow B_1 \setminus V$ is a bijection.  Moreover, if $q >1$,
then $\vphi: B_j \rightarrow B_{j+1}$ is a bijection for all $j \neq
0$.

It follows that $\xi_1$ has two preimages: $\xi_0 \in \cO$ and $\xip_0 \in W$.
The degree of $\vphi$ at each one of these preimages is $1$. The bad direction at $\xi_0$ is
$W$, and denote the bad direction at $\xi'_0$ by $W^\prime$. Both directions
contain the critical points, and their image direction is $V$. Thus, $A=W\cap W^\prime$ is an
open annulus which maps onto $V$. Therefore $A$ is a Fatou component contained in the basin of $\cO$. It follows that each one of the closed balls $B=\pberl \setminus W$ and $B^\prime=\pberl \setminus W^\prime$
map isomorphically onto $\pberl \setminus V$. 
Denote by $\xi'_j$ the preimage of $\xi_{j+1}$ contained in $B^\prime$.

Now consider the following level zero set
$$L_0 = L \setminus  \bigcup_{n \geq 0} \vphi^n (V).$$
The set $L_0$ also has exactly $q$ connected components $Z_0, X_1 \dots, X_{q-1}$
where $Z_0 \subset B_0,  X_1 \subset B_1 \dots,  X_{q-1} \subset B_{q-1}$.
Moreover, $\vphi$ maps $X_j$ isomorphically onto its image
$X_{j+1}$ for all $1 \leq j \leq q-2$. Also, $X_{q-1}$ maps isomorphically onto $Z_0$.
However, $Z_0$ is the disjoint union of $X_0 = Z_0 \setminus W$, $A$
and $B^\prime$. The set $X_0$  maps isomorphically onto $X_1$, $A$ maps in a two-to-one fashion
onto $V$ and $B^\prime$ maps isomorphically onto $\pberl \setminus V$.

Let $L_1 = \vphi^{-1} (L_0) \subset L_0$.
The set $L_1$ has $2q$ connected components.
The components $X_0, \dots, X_{q-1}$ described above and
the components $Y_0, Y_1, \dots, Y_{q-1}$ contained in  $B^\prime \subset Z_0$
such that $\vphi(Y_0) = Z_0$ and $\vphi(Y_j) = X_{j} $ for $j=1, \dots, q-1$.
Note that $\partial Y_j = \{ \xi'_{j-1} \}$.

A key remark is that 
$$\dist_\hl (\partial X_{q-1}, \partial (X_{q-1} \cap \vphi^{-1} (Y_j))) \geq \mdl A,$$
for all $0 \leq j \leq q-1$ and,
$$\dist_\hl (\partial Y_{0}, \partial (Y_{0} \cap \vphi^{-1} (Y_0))) = \mdl A,$$
 
Since the complement of $L_0$ is contained in the Fatou set, 
we have that $\juliavphi \subset L_1$, and we may introduce the itinerary function:
$$\juliavphi \ni \z \mapsto (i_k (\z))_{k \geq 0} \in \Sigma_q$$
if $\vphi^k (\z) \in i_k (\z)$. It follows that the itinerary map is well defined 
and continuous. 

To check that the itinerary function is a bijection, consider a
sequence $\mathbf{i} = (i_0, i_1, \dots ) \in \Sigma_q$.
For all $\ell \geq 0$, let
$$C_\ell (\mathbf{i}) = \{ \z \in \pberl \mid \vphi^k (\z) \in i_k
\mbox{ for all } 0 \leq k \leq \ell \}.$$ Note that $C_\ell (\mathbf{i})$ is
a non-empty and closed set.  Moreover, for all $\ell$, $C_\ell (\mathbf{i})
\supset C_{\ell+1} (\mathbf{i})$ and $C_\ell (\mathbf{i})$ has a unique
boundary point which we denote by $\veta_\ell (\mathbf{i}) \in GO(\cO)$.

The arc $[\veta_\ell (\mathbf{i}), \veta_{\ell+1} (\mathbf{i})]$ is mapped
isometricaly, in the hyperbolic distance,
onto $[\partial i_\ell , \partial (i_\ell \cap \vphi^{-1}(i_{\ell+1}))]$ by
$\vphi^\ell$. Therefore, if $\mathbf{i}$ contains infinitely many symbols
in $\{ Y_0, \dots, Y_{q-1} \}$, then there are infinitely many $\ell \geq 0$
such that $i_\ell = X_{q-1}$ and $i_{\ell+1} = Y_j$ for some $j$ or, $i_\ell = i_{\ell+1} = Y_0$.
From Lemma~\ref{fine.sequence}, $\cap C_\ell(\mathbf{i})$ is a singleton $\{ \xi \}$ contained in $\ponel$
and $\xi$ is the unique point with itinerary $\mathbf{i}$. Moreover,
$\xi \in \juliavphi$ since it is the limit of $\veta_\ell(\mathbf{i}) \in GO(\cO) \subset
\juliavphi$. 
If $\mathbf{i}$ contains only finitely many symbols in $\{ Y_0, \dots, Y_{q-1} \}$,
then there exists $k$ such that
$$(i_\ell)_{\ell \geq k} = ({X_0X_1 \cdots X_{q-1}})^{\infty},$$
(the period $q$ symbol sequence).
Hence $\vphi^k$ maps $ \cap C_\ell(\mathbf{i})$ homeomorphically onto $ X$,
where $X$ consists of $B_0$ with all the directions in the positive and negative orbit 
of $W$ removed. Therefore, the unique point of $X$ in the Julia set is $\xi_0$
and every direction at $\xi_0$ contained in $X$ is a Fatou component contained in the basin of $\cO$. 
It follows that $\cap C_\ell(\mathbf{i})$ contains exactly one point of the Julia set. 
Moreover, every Fatou component contained in  $\cap C_\ell(\mathbf{i})$ is a ball contained in the basin of $\cO$. 
Any Fatou component is either contained in some $\cap C_\ell(\mathbf{i})$ or,
eventually maps isomorphically onto a fixed Rivera domain or, eventually maps isomorphically onto $A$ and statements (4) and (5) of the proposition follow.
\end{proof}

\subsubsection{Examples} Let us illustrate examples of maps possesing an indifferent orbit in $\hl$. 
Below, let us denote by $x_\alpha \in \hl$ the point associated to the ball 
$B_{\val{t}^{\alpha}}(0)$. 

The Gauss point $x_0$ is an indifferent fixed point of the map $$\vphi_1 (\z) =  \z +1 + \dfrac{\tau}{\z}.$$ 
 Moreover, since the bad direction at $x_0$ is the one pointing towards $\z=0$, we have that $x_0 \in J(\vphi_1)$. The unique fixed Rivera domain is the direction of $\infty$.

Given $\theta \in \R \setminus \Q$, the Gauss point $x_0$ is a Julia (indifferent) fixed point of
 $$\vphi_2 (\z) = \exp( i \theta) \z \dfrac{\z+1 + \tau}{\z+1}.$$
Here, the directions of $\z=0$ and $\z=\infty$ at $x_0$ are the fixed Rivera domains.

Finally, consider $$\vphi_3(\z) = \tau^{1/2} - \dfrac{1+\tau^2}{\z} + \dfrac{\tau}{\z^2}.$$
It follows that $\vphi_3$ has a starlike Rivera domain $U$ with skeleton $[x_{1/2}, x_{-1/2}]$. The unique fixed point in the skeleton is $x_0$. The extreme points form a Julia indifferent period $2$ orbit.

\section{Maps with a repelling periodic orbit in $\hl$}
\label{sec:maps-with-repelling}
The aim of this section is to describe the dynamics of quadratic rational maps $\vphi$ having a non-classical repelling periodic orbit. 

Section~\ref{sec:basics} contains Lemma~\ref{basics.repelling} which establishes the basic facts about these maps. Throughout this section as well as
in Section~\ref{sec:TopModel} we will freely employ the content and notation of this Lemma. In Section~\ref{sec:filled-julia-set} we introduce the filled Julia set
$\filledvphi$. We state Proposition~\ref{repelling.filled} which describes its connected components and whose proof is completed, after some work, in 
Section~\ref{sec:proofs12}. In Section~\ref{sec:puzzle} we introduce the basic structure to study $\filledvphi$, namely ``dynamical pieces'' of integer levels. Section~\ref{sec:when-fill-cont} deals with the easier case of maps for which only one critical point lies in $\filledvphi$. 
In Section~\ref{sec:both-critical-points} we deal with the case in which both critical points lie in $\filledvphi$, but before we need to introduce dynamical pieces of ``intermediate levels'' in Section~\ref{sec:piec-interm-level}. Finally, we establish Proposition~\ref{repelling.filled} in Section~\ref{sec:proofs12},
which also contains the proof of theorems~\ref{ithr:periodic} and~\ref{ithr:fatou}.

\subsection{The basics}
\label{sec:basics}
We start by stating the aforementioned lemma.

\begin{lemma}
\label{basics.repelling}
  Let $\vphi \in \L(\z)$ be a quadratic rational map which is not simple.
  Suppose that $\vphi$ has a repelling periodic orbit in $\hl$. Then the following
 statements hold:
  \begin{enumerate}
    \item
      $\vphi$ has exactly one fixed Fatou component $U_0$ which is
      a starlike Rivera domain. The boundary of $U_0$ is a repelling periodic orbit 
$\cO=\{ \xi_0, \xi_1 = \vphi(\xi_0), \dots, \xi_{q-1} = \vphi^{q-1}(\xi_0) \}$
of period $q >1$. The skeleton $\cA_{U_0}$ contains a unique fixed point $\veta_0$ which we 
will call de {\sf center of $U_0$}. The M\"obius transformation $T_{\veta_0} \pberl$ is a rotation of order $q$.

\item
  $L_0 = \pberl \setminus U_0$ is the disjoint union of $q$ closed balls $B_0, \dots, B_{q-1}$
labeled such that $\partial B_j = \{ \xi_j \}$. One of these balls, say  $B_0$, contains
both critical points of $\vphi$. Moreover, the following holds:
\begin{enumerate}
  \item
$$\deg_{\xi_j} \vphi = \left\{ \begin{array}{ll}
    1 & \mbox{if $j \neq 0$,} \\
     2 & \mbox{if $j=0$.}
\end{array} \right. $$

\item
For $j=1, \dots, q-1$, the bad direction at $\xi_j$ is the direction that contains $U_0$; and
$$\vphi : B_j \rightarrow B_{j+1}$$ is a bijection (indices $\mod q$).

\item
There exists a direction  $D_0 \subset B_0$ at $\xi_0$ such that $\vphi(D_0)$ is the direction of $U_0$ at $\xi_1$.
Furthermore,  $$\vphi : B_0 \setminus D_0 \rightarrow B_1$$ 
is two--to--one (counting multiplicities) and 
$$\vphi: D_0 \rightarrow
\pberl \setminus B_1$$ is a bijection.
\end{enumerate}

    \item
      For all $\xi \in \partial U_0$, the rational map
$$T_\xi \vphi^q: T_\xi \pberl \cong \ponec  \rightarrow T_\xi \pberl \cong \ponec$$ 
has a multiple fixed point in the direction of $U_0$. All these maps are M\"obius conjugate and will be denoted
by $T_\cO \vphi$.

\item
There exists a Fatou component containing a critical point which is an open ball $D$, contained in the basin of $\cO$, with  $\partial D \subset \cO$.

    \item
      Assume that  $\vphi$ has another non-rigid repelling periodic orbit $\cO^\prime$ of  period $q^\prime$.
Then $q^\prime > q$ and, for all $\xi \in \cO^\prime$,
$$T_\xi \vphi^{q^\prime}: T_\xi \pberl \cong \ponec  \rightarrow T_\xi \pberl \cong \ponec$$ 
is (in the appropriate coordinate) a degree two polynomial. Moreover, there
exists a critical point of $\vphi$ that belongs to a Fatou component which 
is an open ball $D^\prime$ with $\partial D^\prime \subset \cO^\prime$.
  \end{enumerate}
\end{lemma}

\begin{proof}
  From Section~\ref{sec:periodic.points} and
  Proposition~\ref{attracting.fixed.point} we conclude that $\vphi$ has at
  least one classical indifferent periodic point. Hence, there exists
  a fixed Rivera domain $U_0$.  It follows from
  Proposition~\ref{classification} that $\partial U_0$ is a repelling
  periodic orbit, say of period $q$.  Since $\vphi$ is not simple, $q >1$.  Now by
  Proposition~\ref{star}, $U_0$ is the unique fixed Fatou component and
  it is a  starlike domain.
This proves part (1).

Observe that by Lemma~\ref{degree-one-l} both critical points are
contained in the same complementary ball of $U_0$. Thus, in (2) we may assume that both critical points are in $B_0$.

We now prove part (2a). Again, by Lemma~\ref{degree-one-l}, we
have that $\deg_{\xi_j} \vphi =1$ for all $j \neq 0$. But
$\cO$ is repelling, 
thus $\deg_{\xi_0} \vphi =2$ and (2a) follows.

In order to prove (2b), let $j \neq 0$. By Lemma~\ref{degree-one-l},
there is only one bad direction at $\xi_j$ and this direction must be
the one of $U_0$, since it contains the critical points.  Moreover,
using the fact that $\deg_{\xi_j} \vphi =1$, it follows that every
direction contained in $B_j$ is a good direction and cannot map to the
direction $\pberl \setminus B_{j+1}$.  It follows that $\vphi : B_j
\rightarrow B_{j+1}$ is a bijection.

For (2c), from Section~\ref{ends-s} we conclude
 that the degree of $T_{\xi_0}\vphi$ in the direction $\pberl \setminus B_0$ at $\xi_0$ 
is $1$. Hence, there must exist another direction $D_0$ at $\xi_0$ 
which maps under $T_{\xi_0}\vphi$ onto $\pberl \setminus B_1$.
Since all the directions at $\xi_0$ are good directions, (2c) follows.

Observe that  (3) follows from
Theorem~\ref{residue.formula}, which says that the
  direction of $U_0$ under $T_{\xi_0} \vphi^q$ is a multiple fixed point.
That is, 
$$T_{\xi_0} \vphi^q: T_{\xi_0} \pberl \cong \ponec \rightarrow
T_{\xi_0} \pberl \cong \ponec$$ has a multiple fixed point in the
direction of $U_0$.

For (4), we invoke a result from complex dynamics which guarantees
that a map, such as $T_{\xi_0} \vphi^q$, with a multiple fixed point has a critical point, say $D^\prime$, with
infinite forward orbit (e.g. see~\cite{MilnorComplexBook}).
In particular,  $T_{\xi_0} \vphi^n (D^\prime)$ is a good direction at 
$\vphi^n (\xi)$ for all $n \geq 0$.  Hence, $\vphi^n (D^\prime) =T_{\xi_0} \vphi^n (D^\prime)$ for all $n \geq 0$.
We conclude that  $D^\prime$ is a  Fatou component, it contains a critical point of $\vphi$ and
 the omega limit of every point in $D^\prime$ is $\partial U_0$, since $\vphi^n(D') \neq \vphi^m(D')$ for all $n \neq m$.

Now for (5), note that another repelling periodic orbit 
$$\cO^\prime = \{ \xi'_0, \dots, \xi'_{q^\prime-1} \}$$
(indices respecting dynamics and $\mod q^\prime$)
is contained in $\pberl \setminus \overline{U_0}$.
For each $k=0, \dots, q^\prime -1$, let $B^\prime_k$ be the closed ball 
with boundary $\{ \xi'_k \}$ contained in  $\pberl \setminus \overline{U_0}$.

We claim that $\vphi (B^\prime_k) = B^\prime_{k+1}$. 
In fact, if $B^\prime_k \subset B_j$ for some $j \neq 0$, then this trivially follows from (2b).
If   $B^\prime_k \subset B_0 \setminus D_0$, then this also trivially follows but  now from (2c).
Finally if $B^\prime_k \subset D_0$, then $\vphi(B^\prime_k)$ is disjoint from the direction
at $\xi'_{k+1}$ containing $B_1$. Hence,  $\vphi(B^\prime_k) \subset B_j$ for some $j \neq 1$.
Since $\vphi(B^\prime_k)$ is a closed ball,  it follows 
that $\vphi(B^\prime_k) = B^\prime_{k+1}$.

From the previous paragraph, every direction at $\xi'_k$ contained in
$B^\prime_k$ is a Fatou component.  Therefore, the balls $B^\prime_0,
\dots, B^\prime_{q^\prime-1}$ are pairwise disjoint.  In view of part
(4), at most one of these balls, say $B^\prime_0$, contains a critical point.  It follows that
$\deg_{\xi'_0} \vphi^{q^\prime} =2$.  Finally, $\vphi^{q^\prime}
(B^\prime_0) = B^\prime_0$ implies that 
$$(T_{\xi'_0} \vphi^{q^\prime} )^{-1} (\pberl \setminus B^\prime_0) =
\pberl \setminus B^\prime_0.$$ That is, $T_{\xi_0} \vphi^{q^\prime} $
has a completely invariant direction, after putting this direction at $\infty$, we have that $T_{\xi_0}
\vphi^{q^\prime} $ is a quadratic polynomial, which establishes part (5).
\end{proof}

From the previous lemma it follows that for every non-classical repelling periodic orbit there exists a critical point contained in a Fatou component
whose boundary is a point in the periodic orbit. Therefore:

\begin{corollary}
\label{cor:3}
  If $\vphi : \pberl \to \pberl$ is a quadratic rational map, then $\vphi$ has at most two non-classical repelling periodic orbits.
\end{corollary}

\subsection{The filled Julia set} 
\label{sec:filled-julia-set}

\bigskip
{\sf
  The standing assumption for the rest of this section is that the quadratic rational map $\vphi$ 
has a fixed Rivera domain $U_0$ whose boundary is a period $q >1$ repelling
orbit $\cO$.}

We subdivide dynamical space according to whether a point eventually lands in the fixed Rivera domain or not.
More precisely, we say that
$$\filledvphi := \{ \z \in \pberl \mid \vphi^n (\z) \notin U_0 \,
\mbox{for all} \, n \geq 0 \}$$
is the {\sf filled Julia set of $\vphi$.}

\begin{proposition}
\label{repelling.filled}
  Let $C$ be a connected component of $\filledvphi$. Then one of the following holds:
  \begin{enumerate}
  \item $C = \{ \chi \}$ where $\chi$ is a rigid point.
  \item $C$ eventually maps onto a periodic closed ball $B$. Moreover, $\partial B \subset \cO'$ where $\cO' \neq \cO$ is a repelling periodic orbit.
  \item $C$ consists of a closed ball with infinitely many directions removed. Moreover, $C$ eventually maps onto a component $C'$ whose unique boundary point lies in $\cO$.   
  \end{enumerate}
\end{proposition}

To prove this proposition we consider two different cases. The easier
one is when there exists a critical point which is not in
$\filledvphi$. The proof in this easier case is contained in
Section~\ref{sec:when-fill-cont}. The other case, when both critical
points lie in $\filledvphi$ requires more control on the geometry of
$\filledvphi$ and the corresponding proof is at the end of
Section~\ref{sec:both-critical-points}.

Before introducing the basic combinatorial structure which will allow us to study $\filledvphi$ let us state and prove a basic fact.
 
\begin{lemma}
\label{lem:3}
        $\juliavphi = \partial \filledvphi$.
\end{lemma}

\begin{proof}
  The interior $W$ of $\filledvphi$ is contained in the Fatou set, 
since $\vphi^n(W) \cap U_0 = \emptyset$ for all $ n \geq 0$.
  If $\z \in \partial \filledvphi$ and $U$ is a neighborhood of $\z$, 
then $\vphi^n (U) \cap \cO \neq \emptyset$ for some
$n \geq 0$. Since $\juliavphi$ is closed and totally invariant, we conclude that $\juliavphi = \partial \filledvphi$.
\end{proof}

\subsection{The puzzle}
\label{sec:puzzle}
By definition we say that the {\sf dynamical level $0$ set} 
$L_0$ is the complement of $U_0$
and each of its connected components is a {\sf level $0$ dynamical piece}.

Recursively, for all $\lv \geq 0$ we let
$$L_{\lv+1} = \vphi^{-1} (L_\lv)$$ be the {\sf dynamical level
$\lv +1$ set}. Note that $L_{\lv+1} \subset L_\lv$.  Each
connected component of $L_\lv$ is called a {\sf level $\lv$
dynamical piece}. 

Given a point $\z \in L_\lv$ and an integer $\lv \geq 0$, we denote by $\dpi_\lv (\z)$ the level $\lv$ 
piece which contains $\z$. 

\begin{lemma}
\label{basics.pieces}
The following statements hold:
  \begin{enumerate}
    \item
      For all $\lv \geq 1$, each level $\lv$ piece is contained
in a unique level $\lv -1$ piece.

    \item
      For all $\lv \geq 1$, the image under $\vphi$ of
      each level $\lv$ piece $\dpi$ is
      a level $\lv -1$ piece $\dpi^\prime$. Moreover, $\vphi(\partial \dpi) = \partial \dpi^\prime$.

\item
Every piece of level $\lv$ is an  affinoid, for all $\lv \geq 0$.

\item
If $\lv \geq 0$ and $\dpill$ is a level $\lv$ piece, then $\partial \dpill \subset GO(\cO)$.

    \item
      $\filledvphi = \cap_{\lv \geq 0} L_{\lv}$.
  \end{enumerate}
\end{lemma}

\begin{proof}
  Note that (1) and (5) are an immediate consequence of the definitions. Moreover, (2)
  follows from the fact that $\vphi$ is an open map.  
  Statement (3) is a consequence of the fact that each connected component of the preimage
of an affinoid is an affinoid. Finally (4) follows from induction, since
  $\vphi$ must map the boundary of a piece onto the boundary of a piece of lower level.
\end{proof}

A decreasing sequence of pieces is called a {\sf dynamical end}.
More precisely, a dynamical end $\dend$ is  a collection
$$\dend = \{ \dpi_{\lv} \}_{\lv \geq 0}$$ such that, for all
$\lv$, we have that $\dpi_{\lv}$ is a level $\lv$ piece and
$\dpi_{\lv +1} \subset \dpi_{\lv}$.
The {\sf intersection $\capdend$ of the end} $\dend$ is by definition
$$\capdend = \cap \dpi_\lv.$$
Note that every connected component of $\filledvphi$ is the intersection of an end and vice verse.

Let us now discuss the geometry of the puzzle pieces.
Recall from Lemma~\ref{basics.repelling} that the level $0$ set  $L_0=\pberl \setminus U_0$ 
is the disjoint union of $q$ closed balls labeled by $B_0, \dots, B_{q-1}$ where
\begin{itemize}
  \item
    $B_0$ contains both critical points.

  \item
    For $j=0, \dots, q-1$, we have that 
$\vphi (\partial B_j) = \partial B_{j+1}$, subindices $\mod q$.
\end{itemize}
Hence, there are exactly $q$ level $0$ pieces and each of them is a closed ball.

\medskip
Preserving the notation of  Lemma~\ref{basics.repelling} and applying parts (2b) and (2c) we conclude that
there are exactly $2q-1$ pieces of level $1$. In fact, let $$B^\prime_1, \dots,
B^\prime_{q-1}$$
be the preimage of $B_2, \dots, B_q$ (respectively) under the bijection $\vphi: D_0 \rightarrow \pberl \setminus B_1$.
Hence,
$$B_0 \setminus D_0, B_1, \dots, B_{q-1}, B^\prime_1, \dots,
B^\prime_{q-1},$$ 
is a complete list of the level $1$ pieces.
With the exception of the first one, all of them are closed balls.
Nevertheless, observe that the boundary of each of them is a singleton.

%We say that a critical point $\omega^\prime$ is attached to $\cO$ if the positive orbit of $\omega^\prime$ is disjoint
%from the pieces $B_1^\prime, \dots, B_{q-1}^\prime$.  Note that at most one critical point is not attached to $\cO$. 

\medskip
Now we proceed to study the geometry of pieces of arbitrary level.

\begin{lemma}
\label{lem:15}
  Let $\dpi$ be a level $\lv \geq 0$ piece. Then $\dpi$ is an
affinoid and $\partial \dpi$ is a singleton contained in $GO(\cO)$.
\end{lemma}

\begin{proof}
  We proceed by induction on the level of the piece $\dpi$. 
  The claim is clearly true for  level $0$ pieces.
  Now suppose that the claim is true for level $\lv-1$ pieces. Let
$\dpi$ be a level $\lv$ piece which maps onto $\dpi^\prime$.
Denote by $\chi^\prime$ the unique boundary point of  $\dpi^\prime$.
By the induction hypothesis, every direction at $\chi^\prime$ is either contained in or disjoint from $\dpi^\prime$.
We may assume that $\chi^\prime$ has two preimages, for otherwise the claim is trivial.
Hence, there exists  a  direction $D^\prime_v$ at   $\chi^\prime$ containing both critical values.
Observe that since the critical values of $\vphi$ belong to distinct directions at $\xi_1 \in \cO$,
$D^\prime_v$ must contain $U_0$ and therefore is disjoint from $\dpi^\prime$. Consider  $\chi \in \partial \dpi$.
If a direction at $\chi$ contains the critical points, then it  is disjoint from $\dpi$, by Section~\ref{ends-s} combined with Lemma~\ref{degree-one-l}.
If a direction $D$ at $\chi$ is critical point free, then $D$ is a good direction and either $\vphi(D) \subset \dpi^\prime$
or $\vphi(D) \cap \dpi^\prime = \emptyset$. Hence, directions at $\chi$ are either contained in $\dpi$ or disjoint from
$\dpi$. It follows that $\partial \dpi = \{ \chi \}$.
\end{proof}

For any $\z \in L_\lv$
we will denote by  $\ddll (\z)$ 
the direction at $\partial \dpi_\lv (\z)$ that contains $\z$.
From the previous lemma it follows that for all  $\z \in L_\lv$, 
$\ddll (\z) \subset \dpi_\lv (\z)$.  Moreover, $\vphi (\ddll (\z)) = \dd_{\lv-1} (\vphi(\z))$
provided that $\lv \geq 1$.

\begin{lemma}
\label{closed.ball.l}
Consider an integer $\lv \geq 0$ and two nested pieces $\dpi_{\lv +1} \subset \dpi_\lv$  
of levels $\lv+1$ and $\lv$. If $\partial \dpi_{\lv +1} \neq \partial \dpi_\lv$, then
\begin{enumerate}
\item $\dpi_{\lv +1}$ is a closed ball and,   

\item
$\partial \vphi (\dpi_{\lv +1}) \neq \partial \vphi (\dpi_\lv)$.

\end{enumerate}
\end{lemma}

\begin{proof}
  Since every direction at $\partial \dpi_{\lv}$ contained in $\dpi_{\lv}$ is a good 
direction, (2) follows.
  For (1) we proceed by induction. In view of our discussion of the geometry of level $0$ and $1$
pieces, (1) holds when $\ell =0$.  By (2) we may assume that  $\vphi (\dpi_{\lv +1})$ is the ball whose complement is the direction at
$\partial \vphi (\dpi_{\lv +1})$ containing $U_0$. Then the inductive step is a consequence of the observation that 
if $\dpi$ is a piece such that $\partial \dpi \neq \partial B_0$, then the image of every direction at $\partial \dpi$ disjoint from 
$U_0$ is a direction disjoint from $U_0$. 
\end{proof}

\begin{corollary}
\label{cor:1}
  Let $\dend$ be a dynamical end. Then $\capdend$ is either a nested intersection of closed balls or a closed ball with 
infinitely (countable)  many directions removed. In the latter case,
$\partial \capdend$ is a point in $GO(\cO)$.
\end{corollary}

\begin{proof}
  Let $\dend = \{ \dpill \}$. By the previous lemma, if $\capdend$ is not the nested intersection of closed balls, then there
exists $\ell_0$ such that $\partial \dpi_\lv = \partial \dpi_{\lv_0}$ for all $\lv \ge \lv_0$.
Since $\partial \dpi_{\lv_0}$ is a point in $GO(\cO)$, we have that $\capdend$ eventually maps
onto the intersection of $\dend(\xi_0)$ (recall that $\xi_0 \in \cO$). 
The grand orbit of the direction of $U_0$ under
$T_{\xi_0} \vphi^q$ is infinite, since this direction is a multiple fixed point of $T_{\xi_0} \vphi^q$ (see~Lemma~\ref{basics.repelling}). Therefore,
$\capdend(\xi_0)$ is a ball with a countably infinite set of directions removed and the same occurs
for $\capdend$.
\end{proof}

\subsection{When $\filledvphi$ contains only one critical point}
\label{sec:when-fill-cont}
In this case, we show that the intersections of ends which have infinitely many closed balls are singletons in the classical line.

\begin{lemma}
\label{lem:4}
  Assume that there exists one critical point $\omega$ of $\vphi$ such that $\omega \notin \filledvphi$.
  Let $\dend = \{ \dpi_\lv (E)\}$ be a dynamical end such that
$\dpi_{\lv} (E)$ is a closed ball for infinitely many $\lv$.  
Then $\capdend = \{ \z \} \subset \ponel$.
\end{lemma}

\begin{proof}
Let $\lv_0 \geq 1$ be such that $\omega$ is not a level $\lv_0$ point.
Let $$a = \min \{ \dist_\hl (\partial \dpi_{\lv_0}(\z), \partial \dpi_{\lv_0 +1}(\z)) \mid \partial \dpi_{\lv_0}(\z) \neq \partial \dpi_{\lv_0 +1}(\z) \mbox{ and } \z \in L_{\lv_0+1} \}.$$

Consider  $\lv \geq \lv_0+1$ and suppose that $\dpi_\lv$
and $\dpi_{\lv +1}$ are pieces of levels $\lv$ and $\lv+1$ such that
$\partial \dpi_\lv \neq \partial \dpi_{\lv+1}$.  
We claim that $$\dist_\hl (\partial \dpi_\lv,\partial \dpi_{\lv+1})
= \dist_\hl (\partial \vphi(\dpi_\lv),\partial \vphi(\dpi_{\lv+1})).$$
In fact, the direction  
at  $\partial \dpi_{\lv}$ that contains $\dpi_{\lv+1}$
maps isomorphically onto its image.
Therefore, $\vphi$ maps   $]\partial \dpi_\lv,\partial \dpi_{\lv+1}[$ isometrically
onto $]\partial \vphi(\dpi_\lv),\partial \vphi (\dpi_{\lv+1})[$.
It follows, that $$\dist_\hl (\partial \dpi_\lv,\partial \dpi_{\lv+1}) \geq  a.$$

Since there are infinitely many $\lv \geq 0$ such that  
$\partial \dpi_\lv (E) \neq \partial \dpi_{\lv+1} (E)$, we conclude that
there are infinitely many levels such that 
 $\dist_\hl (\partial \dpi_\lv(E),\partial \dpi_{\lv+1}(E)) \geq  a$.
From Lemma~\ref{fine.sequence}, we conclude that $\capdend$ is a singleton contained in $\ponel$.
\end{proof}

\subsection{Pieces of intermediate level and marked grids}
\label{sec:piec-interm-level}
It will be convenient to simultaneously consider another collection of pieces.

Recall that $\veta_0$ is the center of the starlike fixed Rivera domain $U_0$. That is, the unique fixed point in the skeleton $\cA_{U_0}$ (see~Lemma~\ref{basics.repelling}).
Let $\veta_{1/2}$ be the unique preimage of $\veta_0$ in $D_0$.
Consider the closed ball $B_{1/2} \subset D_0$ defined by
 $B_{1/2}=\pberl \setminus \dd_{\veta_{1/2}} (\veta_0)$.
That is, the ball obtained after removing the direction of $U_0$ at $\veta_{1/2}$ from the Berkovich line.
We define the {\sf level $1/2$ set} as
$$L_{1/2} = B_{1/2} \cup B_1 \cup \cdots \cup B_{q-1}.$$
 Recursively, for $\Z \ni \ell \ge 0$ we let
$$L_{\ell+3/2} = \vphi^{-1}(L_{\ell+1/2}).$$
Each connected component of $L_{\ell+1/2}$ is called a {\sf level $\ell+1/2$ dynamical piece.}

To avoid confusion let us be precise about the notation related to natural numbers:
\begin{eqnarray*}
  \N    & = & \{ n \in \Z \mid n \ge 1 \}, \\
  \Nnot & = & \{ n \in \Z \mid n \ge 0 \}, \\
  \N/2 & = & \{ n/2 \in \Q \mid n \in \N \}, \\
  \Nnot/2 & = & \{ n/2 \in \Q \mid n \in \Nnot \}. 
\end{eqnarray*}

Note that for all $\ell \in \Nnot/2$, 
$$L_{\ell+1} \subset L_{\ell+1/2} \subset L_\ell.$$
Hence every piece of level $\ell+1/2$ is contained in a unique piece of level $\ell$.
Consider a  piece $\dpi_{\lv+1/2}$  of level $\ell +1/2$. It follows that
 $\partial \dpi_{\lv+1/2} \subset GO(\cO \cup \{ \veta_0 \})$. 
Moreover, if  $\lv \ge 1$, then $\vphi(\dpi_{\lv+1/2})$ is a level $\lv -1/2$ piece.

\begin{lemma}
\label{lem:7}
  Let $\ell \ge 0$ be an integer. Consider an intermediate piece $\dpi_{\lv+1/2}$ of level $\lv+1/2$. Let $\dpi_\lv$ be the level $\lv$ piece containing $\dpi_{\lv+1/2}$.
  \begin{enumerate}
  \item 
If $\partial \dpi_\lv = \partial \dpi_{\lv+1/2}$, then
$\dpi_\lv = \dpi_{\lv+1/2}$.
\item

If $\partial \dpi_\lv \neq \partial \dpi_{\lv+1/2}$, then $\dpi_{\lv+1/2}$ is a closed ball.
  \end{enumerate}
\end{lemma}

\begin{proof}
By definition of level $1/2$ pieces, both assertions are clearly true if $\lv=0$.
We proceed by induction.

For (1), given $\ell \geq 1$,  if  $\partial \dpi_\lv = \partial \dpi_{\lv+1/2}$, then 
$\partial \vphi(\dpi_\lv) = \partial \vphi(\dpi_{\lv+1/2})$. By the inductive hypothesis,  $\vphi (\dpi_\lv) = \vphi (\dpi_{\lv+1/2})$.
Therefore, $\dpi_\lv = \dpi_{\lv+1/2}$.

For (2), consider $\lv \ge 1$, assume that  
$\partial \dpi_\lv \neq \partial \dpi_{\lv+1/2}$. 
Then $\dpi_{\lv+1/2}$ is contained in a direction $\dd$ at $\partial \dpi_\lv$ which must 
be a good direction. Hence, $\partial \vphi(\dpi_\lv) \neq \partial \vphi(\dpi_{\lv+1/2})$.
By the inductive hypothesis,  $\vphi(\dpi_{\lv+1/2})$ is a closed ball. 
Given $\z \in \partial \dpi_{\lv+1/2}$, every direction at $\z$, contained in $\dd$, must map into
$\vphi(\dd)$, and therefore into the ball $\vphi(\dpi_{\lv+1/2})$. Hence $\dpi_{\lv+1/2}$ contains
all the directions at $\z$, with the exception of the one that is not contained in $\dd$. 
\end{proof}

\begin{lemma}
\label{lem:16}  Assume that both critical points are level $\ell \in \N/2$ points.
Let $\dpi_{\lv-1/2}$ be a level $\lv-1/2$ piece and $\dd$ be a direction at
$\partial \dpi_{\lv-1/2}$ contained in $\dpi_{\lv-1/2}$ but not contained in $L_\lv$.
Then the following holds:

\begin{enumerate}
\item
$\dd$ contains a unique level $\lv$ piece $\dpi_\lv$,

\item
 $\vphi: \dd \setminus \dpi_\lv \to \vphi(\dd \setminus \dpi_\lv)$ 
is a degree one or two map between annuli. The degree is two if and only if
$\dpi_\lv$ contains a critical point. Equivalently, the degree of $\vphi$ at the boundary point
of $\dpi_\lv$ is two.
\end{enumerate}
\end{lemma}

\begin{proof}
   We assume that both critical points are level $k \ge \lv$ points
   and proceed by induction on $\lv$. For $\lv=1/2$ and $\lv=1$, the
   statement follows from the description of the pieces up to level
   $1$. Assume that (1) and (2) are true for $\lv -1$. Thus $\vphi(\dd)$ is a
   direction at $\partial \dpi_{\lv-3/2}$ contained in
   $\dpi_{\lv-3/2}$ but not contained in $L_{\lv-1}$.  Let $\dpi$ be
   the unique piece of level $L_{\lv-1}$ contained in $\vphi(\dd)$. By the previous lemma $\dpi$
is a closed ball.

   If $\vphi: \dd \to    \vphi(\dd)$ has degree one, then the preimage of $\dpi$ is a unique
   closed ball $\dpi'$. Moreover, $\dd$ is critical point free and so
   is $\dpi'$. 

 If $\vphi: \dd \to \vphi(\dd)$ has degree two, then
   $\partial \dpi$ has a unique preimage $\z$ in $\dd$, since $\dd
   \setminus \dpi$ is critical value free and therefore $\dpi$
   contains one critical value. It follows that the preimage of $\dpi$
   is the unique closed ball $\dpi' \subset \dd$ with boundary point $\z$.
Moreover, $\dpi'$ contains a critical point and $\dd \setminus \dpi'$ maps 
under $\vphi$ onto its image as a degree two map.
\end{proof}

\begin{definition}
  \label{def:5}
Denote by $\omegap$ a critical point that belongs to a Fatou component 
$U(\omegap)$ which is an open ball in the basin of $\cO$ such that $\partial U(\omegap) = \partial B_0$. 
Denote by $\omega$ the other critical point. If there are two choices for $\omegap$ just 
make an arbitrary one. We say that $\omega$ is the {\sf active critical point}.  
\end{definition}

Note that, 
$\partial U(\omegap) = \partial B_0$ and, therefore, 
$\partial \dpi_\lv(\omegap) = \partial B_0$ for all $\ell \ge 0$.
We emphasize that there might be a choice involved.

\begin{definition}
\label{def:4}
Let $\ell \in \N/2$.
Assume that the active critical point $\omega$ is a level $\ell$ point.
Let $\z$ be a point of level $\ell$. 
For   $m \in  \Nnot/2$, $n \in \Nnot$ such that $m +n \leq \lv$,
we let
$$M_{n,m}(\z) = 
 \begin{cases}
2 & \text{if  }  \dpi_m (\vphi^{n}(\z)) = \dpi_m(\omega), \\
1 & \text{otherwise}.
\end{cases}
$$
The {\sf marked grid of $\z$ of level $\lv$} is the array of $1$'s and $2$'s given by:
$$\marked_\lv (\z) = \left( M_{n,m} (\z) \right).$$
If $M_{n,m} (\z) =2$ we say that $(n,m)$ is a {\sf marked position}, otherwise we say that it is an {\sf unmarked position}. 
We say that $\marked_{\lv} (\omega)$ is the {\sf critical marked grid}.

If $\omega, \z \in \filledvphi$, then  we may define $M_{n,m} (\z)$ as above for all $n \in \Nnot$ and $m \in \Nnot/2$.
We say that the infinite array $\marked (\z) = \left( M_{n,m} (\z) \right)$
is the {\sf marked grid of $\z$}. 
\end{definition}

From the definition it follows that:

\begin{enumerate}
\item[(Ma)] If $M_{n,\lv}$ is marked, then $M_{n,j}$ is marked for all $j \leq \lv$.

\item[(Mb)] If  $M_{n,\lv}$ is marked, then $M_{n+j,\lv-j} = M_{j,\lv-j}(\omega)$
for $0 \leq j \leq \lv$.
\end{enumerate}

\smallskip
Marked grids are introduced to help us keep track of the hyperbolic distance
between the boundary of the pieces.

\begin{lemma}
\label{lem:8}
Assume that both critical points are level $\ell$ points. Let $\z$ be a level $\ell$ point and
$\z_n = \vphi^{n}(\z)$. If $m +n \leq \ell$ and $m-3/2 \geq 0$, then
$$M_{n,m} (\z) \cdot \dist_\hl (\partial \dpi_{m-1/2} (\z_n), \partial \dpi_{m} (\z_n))  
= \dist_\hl (\partial \dpi_{m-3/2} (\z_{n+1}), \partial \dpi_{m-1} (\z_{n+1})).$$
\end{lemma}

\begin{proof}
  First observe that $\partial \dpi_{m-1/2} (\z_n)= \partial \dpi_{m} (\z_n)$ if and only if
$\partial \dpi_{m-3/2} (\z_{n+1})= \partial \dpi_{m-1} (\z_{n+1})$. Now if 
$\partial \dpi_{m-1/2} (\z_n) \neq \partial \dpi_{m} (\z_n)$, the required identity follows
from (2) of the previous lemma.
\end{proof}

We must warn the reader that the degree $d$ of $\vphi : \dpill(\z_n) \to \dpi_{\ell-1} (\z_{n+1})$, in general, does not agree with 
$M_{n,\ell}(\z)$. However, always $M_{n,\ell} (\z) \le d$ but strict inequality might hold. That is, it might occur that $\omega \notin \dpi_\ell(\z_n)$ while
$d=2$. Nevertheless, for future reference we record, without proof, a situation in which the degree and the value of the marked grid coincide.

\begin{lemma}
  \label{lem:18}
  Assume that both critical points are level $\ell'$ points. Let $\z$ be a level $\ell'$ point.
 If there exist $\ell \le \ell'$ such that $\dpi_{\ell}(\z)$ is a critical point free  closed ball,
then $M_{\ell',0} (\z) = 1$ and the degree of $\vphi: \dpi_{\ell'}(\z) \to \dpi_{\ell'-1} (\vphi(\z))$ is also $1$.
\end{lemma}

\subsection{Both critical points in $\filledvphi$}
\label{sec:both-critical-points}
Let us assume that both critical points belong to $\filledvphi$. 
It will be convenient to say that $\ell_0 \in \N_0/2$ is the depth of the $n_0$-th column of a marked grid $(M_{n,\ell})$
if $M_{n_0, \ell_0}$ is marked but $M_{n_0, \ell_0+1/2}$ is not marked.

We distinguish  marked grids $\marked$ as follows:
\begin{enumerate}
\item  If there exists $p$ such that $M_{pn,\lv}$ is marked for all $\lv \geq 0$ and all $n \in \N_0$,
then  we say that 
$\marked$ is {\sf periodic of period $p$}. 

\item
Assume that $\marked$ is not periodic.
If there exists $k \geq 1$ and $p \geq 1$ such that for all $n \geq 0$  we have that
$M_{np+k,\lv}$ is marked for all $\lv \geq 0$, then  we say that 
$\marked$ is {\sf strictly preperiodic with eventual period $p$}.

\item
If a marked grid is periodic or strictly preperiodic, then we say that it is {\sf eventually periodic}.

\item
If the critical marked grid $\marked (\omega)$ has columns of arbitrarily large depth, then we say
that $\marked (\omega)$ is {\sf critically recurrent}.
\end{enumerate}

The marked grids involved in our study of quadratic rational maps over $\L$ have levels in $\N_0 /2$ in contrast with
the marked grids that appear in the literature related to  complex quadratic and cubic polynomial dynamics where levels are indexed by $\N_0$ (see~\cite{BrannerHubbardCubicI,MilLC}). 
The introduction of mid-level positions allow us to avoid the so called ``semi-critical'' positions which arise in complex quadratic polynomial
dynamics. However, ``degenerate annuli'' do appear in our context, just as in the study of quadratic polynomials.

Given $\zeta \in \filledvphi$ we are interested on deciding whether
$$S(\z):=\sum_{\ell \in \N/2} \dist_\hl (\partial \dpi_{\lv-1/2}(\z) , \partial \dpi_{\lv}(\z)),$$
converges or diverges. 

If $ \partial \dpi_{\lv-1/2}(\z) = \partial \dpi_{\lv}(\z)$ for all $\ell \geq \ell_0$, then the above sum clearly converges.
Otherwise, $\partial \dpi_{\lv-1/2}(\z) \neq \partial \dpi_{\lv}(\z)$ for infinitely many values of $\lv \in \N/2$ and the following result establishes exactly when
the corresponding series is convergent.

\begin{theorem}
  \label{thr:yoccoz}
  Assume that both critical points belong to $\filledvphi$. Consider $\z \in \filledvphi$.
  If  $\partial \dpi_{\lv-1/2}(\z) \neq \partial \dpi_{\lv}(\z)$ for infinitely many values of $\lv \in \N_0/2$, then $S(\z)$
converges if and only if $\marked (\z)$ is eventually periodic.
\end{theorem}

The proof of the theorem is outlined after the two lemmas below.
We will adapt the original techniques of Branner-Hubbard and Yoccoz (e.g. see~\cite{BrannerHubbardCubicI,MilLC}).
The only difference with the usual techniques is that, in our context, only a weaker version of the so called ``third tableau rule'' holds:   

\begin{lemma}[Weak third rule]
  \label{lem:21}
  Let $\z \in \filledvphi$.
Consider $\ell \in \N/2$ and assume that $\partial \dpi_{\ell-1/2}(\omega) \neq \partial \dpill(\omega)$.
For all  $k \in \N$ and $\ell_0 \in \N/2$ such that $\ell_0 \ge \ell+k$, the following holds:

If $M_{j,\ell_0-j}(\omega)$ is not marked for all $j=1, \dots, k-1$, $M_{k, \ell_0-k+1/2}(\omega)$ and $M_{n,\ell_0}(\z)$ are marked, then $M_{n+k,\ell_0-k+1/2}(\z)$ is not marked.
\end{lemma}

\begin{proof}
  We proceed by contradiction and assume that $M_{n+k,\ell_0-k+1/2}(\z)$ is  marked. That is, $$\dpi_{\ell_0-k+1/2}(\z_{n+k}) = \dpi_{\ell_0-k+1/2} (\omega).$$ 
By Lemma~\ref{lem:18}, $\vphi^k: \dpi_{\ell_0}(\omega) \to \dpi_{\ell_0-k}(\omega)$ has degree $2$, since $\dpi_{\ell+k-j} (\vphi^j(\omega))$ is a critical point free ball for $1 \le j <k$. Now $\dpi_{\ell_0+1/2}(\omega)$ maps with degree $2$ in $k$ iterates
onto $\dpi_{\ell_0-k+1/2} (\omega)$.
Therefore, the unique level ${\ell_0+1/2}$ piece contained in 
$\dpi_{\ell_0}(\omega) = \dpi_{\ell_0}(\z_n)$ which maps in $k$ iterates 
onto $\dpi_{\ell_0-k+1/2} (\omega)$ is $\dpi_{\ell_0+1/2}(\omega)$. 
We conclude that $\dpi_{\ell_0+1/2}(\omega) =
\dpi_{\ell_0+1/2}(\z_n)$ which contradicts 
the fact that $M_{n, \ell_0+1/2}(\z)$ is unmarked.
\end{proof}

The key to prove the theorem is the following result.

\begin{lemma}
  \label{lem:22}
  If $\marked (\omega)$ is recurrent but not periodic, then $S(\omega)$
diverges.
\end{lemma}

\begin{proof}
Since  $\marked (\omega)$ is not periodic, there exists $\ell' \in \N/2$ such that 
$\partial \dpi_{\ell'} (\omega) \neq \partial B_0$. Hence, there exists $\ell_0 \in \N/2$ 
such that  $\partial \dpi_{\ell_0-1/2} (\omega)\neq \partial \dpi_{\ell_0} (\omega)$.

Given $d \in \N/2$ such that 
$\partial \dpi_{d-1/2} (\omega)\neq \partial \dpi_{d} (\omega)$, 
we claim that there exist integers $m > k > 0$ such that
$\dpi_{d+k} (\omega)$ maps under $\vphi^{k}$ onto $\dpill (\omega)$ with degree $2$
and $\dpi_{d+m} (\omega)$ maps under $\vphi^{m}$ onto $\dpill (\omega)$ with degree $2$.
Thus, we will obtain that:
$$\dist_\hl (\partial \dpi_{d-1/2}(\omega) , \partial \dpi_{d}(\omega)) = 
\dist_\hl (\partial \dpi_{d+k-1/2}(\omega) , \partial \dpi_{d+k}(\omega))
+ \dist_\hl (\partial \dpi_{d+m-1/2}(\omega) , \partial \dpi_{d+m}(\omega)).$$

Let $k \ge 1$ be such that the $k$-th column of $\marked(\omega)$ is the first one with 
depth at least $d$. 
For $0 \le j < k$, we have that 
$\partial \dpi_{d-1/2+k-j} (\omega_j)\neq \partial \dpi_{d+k-j} (\omega_j)$
since, under $\vphi^{k-j}$, these singletons map onto the singletons
$\partial \dpi_{d-1/2} (\omega)\neq \partial \dpi_{d} (\omega)$.
By Lemma~\ref{lem:18}, it follows that $k$ has the desired properties.

To obtain $m$ we follow the original proof, as presented in~\cite[Lemma~1.3 (b)]{MilLC}, and
we outline how the ``weak third rule'' maybe applied instead of the ``third rule''.  
Denote the depth of the $k$-th column  by $d'-1/2$. Consider the largest integer $n$ such that $d' - k n > d$. 
Taking into account
that $\partial \dpi_{d} (\omega) \neq \partial \dpi_{d-1/2} (\omega)$, we may apply $(n-1)$-times Lemma~\ref{lem:21}, 
just as in~\cite{MilLC}, to prove that $M_{d'-j,k+j}(\omega)$ is unmarked for all $j$ such that  $1< j \leq k n$. 
Now let $m$ be the smallest integer greater than $(n+1)k$ such that the $m$-th column has depth at least $d$.
It follows that $m$ has the desired properties.

The level $d+k$ and $d+m$ critical pieces are usually called children of $\dpi_d (\omega)$.
Repeating the argument one obtains $2^t$ descendants of the $t$-th generation such that
the total contribution of each generation to the sum is $\dist_\hl (\partial \dpi_{d-1/2}(\omega) , \partial \dpi_{d}(\omega))$.
Thus, the sum $S(\omega)$ diverges.   
 \end{proof}

\begin{proof}[Proof of Theorem~\ref{thr:yoccoz}]
Let $\z \in \filledvphi$ be as in the statement of the theorem.

First we consider the case in which there exists $\ell_0 \in \N/2$ such that
$\partial \dpi_{\ell_0}(\omega) = \partial \dpi_\ell (\omega)$ for all $\ell \ge \ell_0$.
If $\partial \dpi_{d-1/2} (\z) \neq \partial \dpi_{d} (\z)$, then $\partial \dpi_{d-1/2-j} (\z_j) \neq \partial \dpi_{d-j} (\z_j)$
for all $j \le d-1/2$. Thus, taking $k \in \N$ such that $d-k = \ell_0+1/2$ or $d-k = \ell_0 +1$, we obtain that $\dpi_{d-j} (\z_j) \neq \dpi_{d-j}(\omega)$
for all $j$ such that $j \leq k$. Therefore, $\dist_\hl (\partial \dpi_{d-1/2} (\z), \partial \dpi_{d} (\z))$ takes one of the finitely many 
positive values $\dist_\hl (\partial P, \partial P')$ where $P$ (resp. $P'$) is a level $d-k-1/2$ (resp. $d-k$) piece such that
$P \supset P'$ and $\partial P\neq \partial P'$. Hence the sum $S(\z)$ diverges.

Now we consider the case in which there exists $\ell_0$ such that all the positions of depth $\ell \ge \ell_0$ in $\marked (\z)$ are unmarked.
For all $d > \ell_0$,  taking  $k \in \N$ such that $d-k = \ell_0+1/2$ or $d-k = \ell_0 +1$, we obtain that $\dpi_{d-j} (\z_j) \neq \dpi_{d-j}(\omega)$
for all $j$ such that $j \leq k$. Again, if $\partial \dpi_{d-1/2} (\z) \neq \partial \dpi_{d} (\z)$, 
then $\dist_\hl (\partial \dpi_{d-1/2} (\z), \partial \dpi_{d} (\z))$ takes finitely many 
positive values. Thus, the sum $S(\z)$ diverges. 

Next we consider the case in which 
$S (\omega)$
diverges. 
By the above we may assume that $\marked(\z)$ has columns of arbitrarily large depth.
For each $\ell$, let $n(\ell)$ be the first column in $\marked(\z)$ with depth at least $\ell$. Since all the positions,
$M_{\ell+n(\ell)-j,j}(\z)$ are unmarked, for $j=0,\dots, n(\ell)$, we obtain that
$$\dist_\hl ( \dpi_{\ell-1/2}(\omega), \dpi_{\ell}(\omega))=\dist_\hl ( \dpi_{\ell+n(\ell)-1/2}(\z), \dpi_{\ell+n(\ell)}(\z)).$$
It follows that the sum diverges.

Note that if  $S(\omega)$ converges, then either $\marked(\omega)$ is periodic or
there exists $\ell_0 \in \N/2$ such that
$\partial \dpi_{\ell_0}(\omega) = \partial \dpi_\ell (\omega)$ for all $\ell \ge \ell_0$.
Since we have already taken care of the latter, to finish the proof we have to consider 
the case in which $\marked(\omega)$ is periodic and $\partial \dpi_{\ell -1/2}(\omega) \neq \partial \dpi_\ell (\omega)$ 
for infinitely many $\ell$. 
It is not difficult to prove that if $\marked(\omega)$ is periodic, then $S(\omega)$ converges. Thus, if 
$\marked(\z)$ is eventually periodic, then $S(\z)$ converges.
It only remains to show that if $\marked(\z)$ is not eventually periodic, then $S(\z)$ diverges.
This is the other instance in which the ``(weak) third rule'' is applied. We follow the exposition of Milnor
(\cite[Theorem~2.4, Case 2]{MilLC}).
Consider $N$ such that all columns of $\marked(\omega)$ with depth at least $N$ have infinite depth. That is,
these columns are multiples of $p$, where $p$ is the period of $\marked(\omega)$.
Without loss of generality we may assume that $\partial \dpi_{N-1/2} (\omega)\neq \partial \dpi_{N} (\omega)$.
There are infinitely many pairs $(m,d)$ with $d \in \N /2$, $d \ge N$, $m \in \N$ such that
the $m$-th column of 
$\marked(\z)$ has depth exactly $d-1/2$ and the $m$-th column is the first to have depth at least $d-1/2$.
Taking into account that  $\partial \dpi_{N-1/2} (\omega)\neq \partial \dpi_{N} (\omega)$ we may apply the ``weak third rule''
to conclude that the columns $m+jp$ have depth $d-1/2-jp$ as long as $d-1/2-jp \ge N$. Let $k$ be the largest interger such that
$d-1/2-kp \ge N$. Now consider the smallest integer $m' > m+kp$ such that $M_{m',N}(\z)$ is marked. It follows that
all the positions $M_{m'+N-j, j}(\z)$ are unmarked for all $j=0, \dots, m'-1$. Thus, 
$$\dist_\hl ( \partial \dpi_{m'+ N-1/2} (\z),  \partial \dpi_{m'+N} (\z))=\dist_\hl ( \partial \dpi_{N-1/2} (\omega),  \partial \dpi_{N} (\omega)).$$
Note that $m' +N > m+d -1/2 -p$. So we may recursively choose pairs $(m,d)$ to obtain infinitely many values of $\dist_\hl ( \partial \dpi_{\ell-1/2} (\z),  \partial \dpi_{\ell} (\z))$ which agree with $\dist_\hl ( \partial \dpi_{N-1/2} (\omega),  \partial \dpi_{N} (\omega)).$
Thus $S(\z)$ diverges.

\end{proof}

\begin{corollary}
\label{lem:9}
  Assume that both critical points are in $\filledvphi$.
  Let  $\z \in \filledvphi$ and suppose that $\partial \dpi_\lv(\z) \neq \partial \dpi_{\lv+1/2}(\z)$ for infinitely many values of $\lv$. Then one of the following holds:
  \begin{enumerate}
  \item $\cap \dpi_\lv(\z) = \{ \z \}$ and $\z$ is a rigid point.
  \item $\marked (\z)$ is eventually periodic and $\cap \dpi_\lv(\z)$ is an eventually periodic
closed ball $B$. Moreover, $\marked(\omega)$ is periodic and $\cap \dpi_\lv (\omega)$ is a periodic closed ball $B'$. The point associated to $B'$ lies in a non-rigid repelling periodic orbit $\cO'$ and $B$ eventually maps onto $B'$. 
  \end{enumerate}
\end{corollary}

\begin{proof}
  In view of  Corollary~\ref{cor:1}, $\cap \dpi_\lv(\z)$ is either $ \{ \z \}$  or a closed ball $B$. From the previous theorem and Lemma~\ref{fine.sequence},
if the latter holds, then  $\marked (\z)$ is eventually periodic. In this case, $\z$ eventually maps to 
a periodic ball containing the critical point $\omega$.
\end{proof}

% \begin{corollary}
% Assume that both critical point belong to $\filledvphi$.
% The critical point  $\omega$ lies in $\juliavphi$ if and only if $\marked(\omega)$ is not eventually periodic and
% $\partial \dpi_\lv(\omega) \neq \partial \dpi_{\lv+1}(\omega)$ for infinitely many $\lv \in \N$.
% \end{corollary}

\subsection{Proofs of Proposition~\ref{repelling.filled},
theorems~\ref{ithr:periodic} and~\ref{ithr:fatou}}
\label{sec:proofs12}

\begin{proof}[Proof of Proposition~\ref{repelling.filled}]
From Corollory~\ref{cor:1} and Lemma~\ref{lem:4} the proposition follows when there exists a critical point $\omega \notin \filledvphi$. 
Now,  corollories~\ref{cor:1} and~\ref{lem:9} establish the proposition when both critical points are contained in $\filledvphi$.  
\end{proof}

\begin{proof}[Proof of theorems~\ref{ithr:periodic} and~\ref{ithr:fatou}]
  According to Proposition~\ref{classification} we have three possibililies:
  \begin{itemize}
  \item[(a)] $\juliavphi \cap \hl$ is periodic point free.
  \item[(b)] There exists an indifferent periodic orbit $\cO$ in $\juliavphi \cap \hl$.
  \item[(c)] There exists a repelling periodic orbit $\cO$ in $\juliavphi \cap \hl$.
  \end{itemize}
From Proposition~\ref{attracting.fixed.point}, (a) is equivalent to $\juliavphi \subset \ponel$. In this case, $F(\vphi)=\pberl \setminus \juliavphi$ is connected and coincides with the basin of an attracting fixed point. Thus, (1) of theorems~\ref{ithr:periodic} and~\ref{ithr:fatou} hold.

Assume that (b) holds. From Proposition~\ref{pro:14} (2),  we have that $\juliavphi \cap \hl = GO(\cO)$. That is, Theorem~\ref{ithr:periodic} (2) holds.
Moreover,  Proposition~\ref{pro:14} (4) implies that every periodic Fatou component is a fixed Rivera domain with boundary $\cO$. From Proposition~\ref{star},
Rivera domains are either  balls or starlike. In the former, the boundary must be a fixed point and Theorem~\ref{ithr:fatou} (2) holds. In the latter, this starlike domain is unique and Theorem~\ref{ithr:fatou} (3). Also observe that Proposition~\ref{pro:14} (5) implies that Fatou components are eventually periodic or
balls/annuli contained in the basin of $\cO$.

Now assume that (c) holds. From Proposition~\ref{repelling.filled} and Lemma~\ref{lem:3} we conclude that  every point in $\juliavphi \cap \hl$ is contained in the grand orbit of a non-rigid repelling orbit. According to Corollary~\ref{cor:3} there are at most two such orbits which, by Lemma~\ref{basics.repelling} (3) and (5) are such that Theorem~\ref{ithr:periodic} (3) or (4) hold. Hence, we have established Theorem~\ref{ithr:periodic}.
If $U$ is a Fatou component, then either $U$ eventually maps onto the fixed Rivera domain or $U$ is contained in a connected component $C$ of $\filledvphi$.
According to  Proposition~\ref{repelling.filled}, $U$ must be a direction at
the unique Julia set point $\z$ in $C$. Moreover, $\z$ belongs to the grand orbit of a non-rigid repelling periodic orbit. Thus, $U$ is either an eventually periodic open ball (of period at least $2$) or it is contained in the basin of a periodic orbit. That is, Theorem~\ref{ithr:fatou} (4) holds. 

From the above, we conclude that Fatou components are eventually periodic or balls/annuli contained in the basin of a periodic orbit. Thus, we have also established the statement of Theorem~\ref{ithr:fatou}.
\end{proof}

\subsection{Examples}
Let us give examples of maps having one or two repelling periodic orbits in $\hl$. For convenience, below we denote by $x_\alpha$ the point associated to the ball $B_{\val{t}^\alpha}(0)$. 

First consider $$\vphi_1 (\z) = -\tau - \dfrac{1+\tau^2}{\z} + \dfrac{\tau}{\z^2}.$$
Note that $\vphi_1$ has a fixed starlike Rivera domain with skeleton $[x_{1}, x_{-1}]$. The Gauss point is the unique fixed point in the skeleton. The extreme points form a period $2$ repelling orbit. In the coordinates for $T_{x_1} \pberl$
given by $z t + t \mathfrak{M}_\L \mapsto z$, we have  
$$R(z) = T_{x_1} \vphi^2_1 (z) = -1 + \dfrac{z^2}{z-1}.$$
Since $R$ has a triple fixed point at $\infty \in \ponec$, it follows that the  critical points $z=0,2$ of $R$ have infinite forward orbit converging to $\infty$ under iterations of $R$. Thus both critical points are in the basin of $\cO= \{ x_{-1}, x_1 \}$ and $\cO$ is the unique repelling periodic orbit.

Now consider $$\vphi_2 (\z) = \tau - \dfrac{1+\tau^2}{\z} + \dfrac{\tau}{\z^2} - a \tau^5$$ where $a \in \C$.
Also, $\vphi_2$ has a fixed starlike Rivera domain with skeleton $[x_{1}, x_{-1}]$ and the extreme points form a period repelling $2$ orbit. In the coordinates used above
$$T_{x_1} \vphi^2_2 (z) = 1 + \dfrac{z^2}{z-1}$$
which has a double fixed point at infinity. 
Moreover, $x_3$ is a period $3$ repelling periodic point, and 
$$T_{x_3} \vphi^3_2 (z) = z^2 +a,$$
in the appropriate coordinates.

\section{Quadratic Laminations and  trees}
\label{section|laminations}
Invariant laminations were introduced in complex polynomial dynamics by Thurston~\cite{thurstonlamination} and have been widely used to
describe  the dynamical space as well as the parameter space of complex polynomials. 
In~\cite[Section 8.5]{McMullenBook1}, McMullen defines the $\alpha$-lamination of a quadratic polynomial.
Here we introduce  a variation of these $\alpha$-laminations which will allow us to describe quadratic rational maps over $\L$.
More precisely,
we will introduce {\it abstract $\alpha$-laminations} and construct a tree
associated to each such abstract $\alpha$-lamination.
In a certain sense, 
this tree is related both to the Yoccoz puzzle used to study the dynamics of 
complex quadratic polynomials (e.g. see~\cite{Hubbard3Yoccoz,MilLC}), as well as to the 
obstructions which arise in the ``mating'' of two complex quadratic polynomials~\cite{LeiMating}.

In Section~\ref{sec:define.alpha} we introduce abstract $\alpha$-laminations
as well as illustrate the definition with those that arise from complex quadratic polynomial dynamics. Their geometric representation  as ``laminations'' in $\ponec$ is discussed in Section~\ref{sec:geo}. The trees of integer levels associated to an abstract $\alpha$-lamination and their basic properties are discussed in Section~\ref{sec:treel}. Then, in Section~\ref{sec:tree.full}, we introduce the associated infinite tree with its corresponding dynamics. 
In Section~\ref{section:symbolic} we construct all abstract $\alpha$-laminations via symbolic dynamics.
Finally, in Section~\ref{sec:branched-maps-trees}, we define  degree $2$ branched maps of trees and prove a lifting lemma.

\smallskip
{\sf Throughout, we  denote the multiplication by $2$ map acting on $\RZ$ by  $m_2$.}

\subsection{Abstract $\alpha$-laminations}
\label{sec:define.alpha}
Let $\{ t_0, \dots, t_{q-1}\} \subset \RZ$ be a period $q \geq 2$ periodic orbit under $m_2$ labeled $\mod q$ and respecting cyclic order in $\RZ$. 
Such a periodic orbit has {\sf rotation number $p/q \in \RZ$} if $m_2(t_j) = t_{j+p}$ for all $j$.
For example $\{ 1/3, 2/3 \}$ and $\{1/7,2/7,4/7\}$ have rotation numbers $1/2$
and $1/3$, respectively.
According to Bullet and Sentenac~\cite{BulSen}, for each pair of co-prime 
positive integers $p$ and $q$, there is exactly one periodic orbit of $m_2$ with
rotation number $p/q$.

\begin{definition}
\label{alphaLamination|definition}
An equivalence relation $ {\lambda}$  in $\RZ$  is called an {\sf abstract $\alpha$-lamination} if the following hold:

\begin{enumerate}
\item 
{\sf Invariant.} If $A$ is a class, then $m_2(A)$ is a class.

\item
{\sf Finite.} Every class contains finitely many elements.

\item
{\sf Unlinked.} If $A$ and $B$ are two distinct classes, then $A$ is contained in a connected component of $\RZ \setminus B$.

\item
{\sf Consecutive preserving.} If $A$ is a class and $]t,s[$ is a connected component of $\RZ \setminus A$, then $]m_2(t), m_2(s)[$ is a connected component
of $\RZ  \setminus m_2(A)$.

\item
{\sf $\alpha$-supported}. There exists a class $A_0$ with at least two elements such that $m_2(A_0)=A_0$ and, for all classes $A$ with at least two elements, there exists $\ell \geq 0$ such that $A_0= m^\ell_2 (A)$.
\end{enumerate}
We say that  $A_0$ is the {\sf fixed class of $ {\lambda}$}. It follows that $A_0$ as above has a well defined rotation number, say $p/q$, and we say that $ {\lambda}$ is an  {\sf abstract $\alpha$-lamination in the $p/q$-limb.}
\end{definition}

Throughout, we let  $1 \le p < q$ be relatively prime integers.

\begin{definition}
\label{def:9} 
Let $A_0 \subset \RZ$ be the unique set of $q$ arguments such that $m_2 (A_0) =A_0$  and $A_0$ has rotation number $p/q$.
According to  Milnor~\cite{MilnorPOP}, among  the connected components of $\RZ \setminus A_0$ there is one with smallest length, say 
$I_{p/q} = ]\theta_0, \theta_1[$. The interval $I_{p/q}$ is called the {\sf $p/q$-characteristic interval.} 
\end{definition}

\subsection{$\alpha$-laminations  that arise from complex dynamics}
\label{sec:alpha-lamin}
Examples of abstract $\alpha$-laminations arise in the context of iteration of quadratic
complex polynomials. 
 We refer the reader to~\cite{MilnorComplexBook}
for a  detailed exposition about  the basics of  iteration of
complex rational and polynomial maps.

Consider the quadratic family
$$Q_c(z) =z^2 +c$$
where $c \in \C$. The {\sf Mandelbrot set $\cM$} consists of all parameter $c$ for which 
the corresponding Julia set $J_c$ is connected. 
Let $A_0 \subset \R/\Z$ be the periodic orbit under $m_2$ with rotation number $p/q$.
The {\sf $p/q$-limb of $\cM$}, denoted $L_{p/q}$, consists of all parameters $c \in \cM$ such that
the external rays with arguments in $A_0$ land at a fixed point of $Q_c$.
This fixed point is called the {\sf $\alpha$-fixed point of $Q_c$.}

Recall that the  Douady-Hubbard map $\Phi: \CDC \to \C \setminus \cM$  is the unique conformal isomorphism tangent to the identity at infinity.
The image of $]1,+\infty[ \exp(2 \pi i t)$ under $\Phi$ is called {\sf the parameter ray $R^t_\cM$ at argument $t$}.
Parameter rays at arguments $t \in \QS$ have a well defined limit (i.e. land) as they approach the Mandelbrot set.
If $t$ is periodic under $m_2$, then the landing point $c$ is such that $Q_c$ has a multiple periodic point.
If $t$ is strictly preperiodic under $m_2$, then the landing point $c$ is such that the critical point is preperiodic under $Q_c$
(e.g. see~\cite{MilnorPOP}). Moreover, denote by  $I_{p/q} = ]\theta_0, \theta_1[$ the $p/q$-characteristic interval.
Then, the parameter rays $R_\cM (\theta_0)$ and $R_\cM (\theta_1)$ land at the same point $c_0$ which is the unique parameter in the boundary
of the main cardiod for which $Q_{c_0}$ has a fixed point with multiplier $\exp(2 \pi i p/q)$~\cite{OrsayNotes}.
The limb $L_{p/q}$ can be described as follows. 
These parameter rays together with $c_0$ cut the complex plane into two sectors (connected components). Denote by $S$ the sector 
not containing the main cardiod, see Figure~\ref{fig:1}.
It follows that $L_{p/q} = \overline{S} \cap \cM = (S \cap \cM) \cup \{ c_0 \}$ (see~\cite{AtelaWakes,MilnorPOP}). 
There is a unique parameter $c_{p/q}$ in $L_{p/q}$ with a period $q$ critical orbit~\cite{OrsayNotes}.
This parameter  $c_{p/q}$ is called the {\sf center of  $L_{p/q}$}.

\begin{figure}
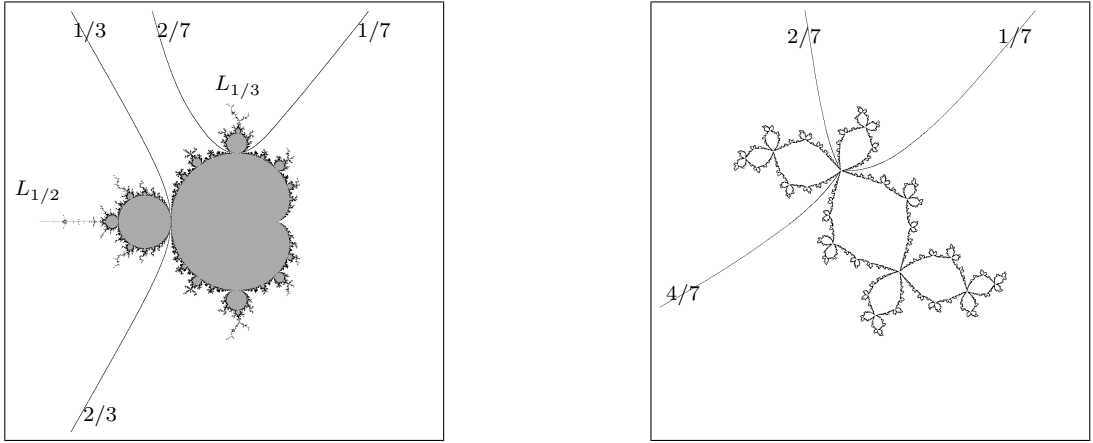

  \begin{center}
    \begin{minipage}[c]{7cm}
      \center{\input{limbs.pspdftex}}
    \end{minipage}\hfill
    \begin{minipage}[c]{7cm}
      \center{\input{rabbit.pspdftex}}
    \end{minipage}
  \end{center}
  \caption{Left: the parameter rays with arguments in $\partial I_{1/2}$ and $\partial I_{1/3}$, and the  limbs $L_{1/2}$ and $L_{1/3}$
 of the Mandelbrot set. Right: the Julia set of the center $c_{1/3}$ of $L_{1/3}$ together with the rays landing at the $\alpha$-fixed point.}
  \label{fig:1}
\end{figure}

\medskip
A large class of examples of abstract $\alpha$-laminations in the $p/q$-limb arise as the equivalence relations that encode
the landing pattern of external rays at the grand orbit of the $\alpha$-fixed point of $Q_c$ with $c \in L_{p/q}$. 

\begin{definition}
\label{def:6}
  For $c \in L_{p/q}$ we let {\sf the $\alpha$-lamination $ {\lambda}_{\alpha}(Q_c)$ 
of $Q_c$} be the equivalence relation in
$\RZ$ that identifies two distinct arguments $t,s$ if and only if the external rays of $Q_c$ with arguments $t,s$ land at a common point in the grand orbit of
the $\alpha$-fixed point of $Q_c$. 
\end{definition}

From~\cite{McMullenBook1}, it follows that  $ {\lambda}_{\alpha}(Q_c)$ is an 
abstract $\alpha$-lamination in the $p/q$-limb.
We will show that not all abstract $\alpha$-laminations arise in quadratic polynomial dynamics (see Remark~\ref{rem:1}). If $ {\lambda} =  {\lambda}_\alpha (Q_c)$ for some $c$, 
we simply say that $ {\lambda}$ is an {\sf $\alpha$-lamination}.

The $\alpha$-lamination of a quadratic polynomial $Q_c$ such that the critical point $z=0$ eventually maps onto its $\alpha$-fixed point is called a {\sf critically prefixed $\alpha$-lamination.}
For reasons that will be apparent after Proposition~\ref{pro:3} and Remark~\ref{rem:1}, 
  an abstract $\alpha$-lamination which is not the $\alpha$-lamination of a quadratic polynomial is called an {\sf almost critically prefixed lamination}.

For our purpose, it will be also convenient to construct $\alpha$-laminations via symbolic dynamics (see Section~\ref{section:symbolic} below).

\subsection{Geometric representation of an $\alpha$-lamination}
\label{sec:geo}
Following Thurston it is useful to have a geometric representation of laminations.
We identify $\RZ$ with the boundary of the unit disk $\D \subset \C$  via $t \mapsto \exp(2 \pi i t)$. 
Given a finite set $A \subset \RZ$ denote by $\operatorname{Convex} (A) \subset \overline{\D}$ the convex hull of $A \subset \partial \D$
 with respect to the hyperbolic metric.
From the unlinked property of $\alpha$-laminations, it follows that convex hulls of distinct classes of an abstract $\alpha$-lamination are disjoint.

Throughout, we identify $\ponec$ with $\C \cup \{\infty\}$ via $[z:1] \mapsto z$ and $[0:1] \mapsto \infty$.

\begin{definition}
  \label{def:8}
  Consider an equivalence relation $ {\lambda}$ in $\RZ$ such that equivalence classes are finite and pairwise unlinked (see Definition~\ref{alphaLamination|definition}). We say that the {\sf geometric lamination $\mathcal{L}_0 ( {\lambda}) \subset \ponec$ of $ {\lambda}$ centered at the origin} is the  set formed by the union of all  $\operatorname{Convex}(A)$ such that  $A$  is a non-trivial $ {\lambda} -\mbox{class}$.

  Consider the reflection $M(z)
=1/z$ around the unit circle. We say that {\sf $\mathcal{L}_\infty ( {\lambda}) = M(\mathcal{L}_0 ( {\lambda})) \subset \ponec$ is the
 geometric lamination  of
  $ {\lambda}$ centered at infinity}.
\end{definition}

For our construction  of  a tree associated to an abstract $\alpha$-lamination $ {\lambda}$ we will ``saturate'' the ``support'' of $ {\lambda}$ by finite sets
$A_\ell$. More precisely, let $A_0$ be the fixed class of $ {\lambda}$ 
and for all $\ell \geq 1$, let
$$A_\ell = m^{-\ell}_2 (A_0).$$
For all $\ell \ge 1$, note that $A_{\ell-1} \subset A_\ell$ and that each $ {\lambda}$-class is either contained or disjoint from $A_\ell$.
We say that the restriction $\lam\ell$ of $ {\lambda}$ to $A_\ell$ is the {\sf level $\ell$ restriction of $ {\lambda}$} and that  $A_\ell$ is the {\sf level $\ell$ support of $ {\lambda}$}.

% It is also useful to represent laminations also
% in the complement of the unit disc (compare with the mating 
% construction~\cite{DouadyBourbaki82}).  
% For this purpose we let $M(z)
% =1/z$ and, for each $\ell \geq 0$ we define the {\sf level $\ell$
%  geometric lamination $\mathcal{L}_\infty^{(\ell)} ( {\lambda}) \subset \ponec$ of
%   $ {\lambda}$ centered at the infinity} as the set formed by the union
% of all $M(\operatorname{Convex}(A))$ such that $A$ is a $\lam\ell
% -\mbox{class}$.

If $ A_\ell$ is the level $\ell$ support of a lamination in $p/q$-limb, then 
$- A_\ell = M(A_\ell)$ is the level $\ell$ support of a lamination in the $-p/q$-limb.
We will work under the agreement that the $-p/q$-limb denotes the $(q-p)/q$-limb since  $-p/q=(q-p)/q \in \RZ$.

The following definition is closely related to the ``mating construction''~\cite{DouadyBourbaki82}. More precisely, to the ``Levy cycle'' 
of ``obstructed matings''~\cite{LeiMating}.

\begin{definition}
  Let $ {\lambda}$ be an abstract $\alpha$-lamination in the $p/q$-limb and
$ {\lambda}_*$  be the $\alpha$-lamination of the center of the $-p/q$-limb. Denote by $\cL_0( {\lambda}^{(\ell)})$ the level $\ell$ geometric lamination of $ {\lambda}$ centered at the origin and $\cL_\infty( {\lambda}^{(\ell)}_*)$  the level $\ell$ geometric lamination of  $ {\lambda}_*$  centered at infinity.
We say that the {\sf level $\ell$ geometric lamination} of $ {\lambda}$ is:
$$\mathcal{L} ( {\lambda}^{(\ell)}) = 
\cL_0( {\lambda}^{(\ell)}) \cup \cL_\infty( {\lambda}^{(\ell)}_*).$$
\end{definition}

See the lower left of Figure~\ref{fig:2} for the illustration of a geometric lamination.

\begin{figure}
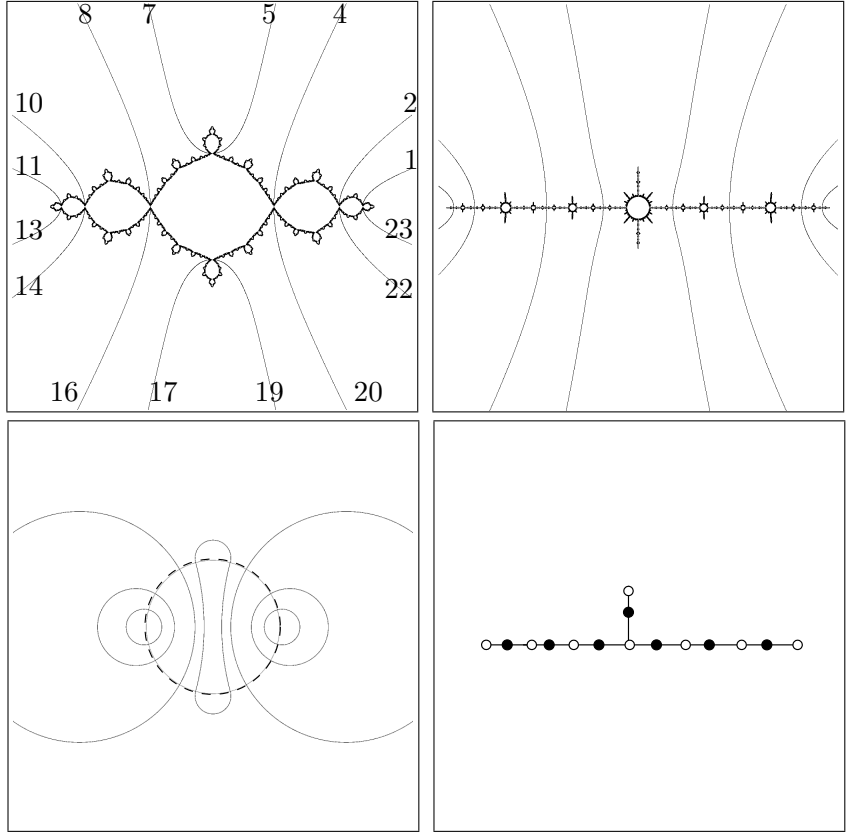

   {\input{airbasLTL.pspdftex}}
  {\input{airbasTR.pspdftex}}

  {\input{airbasBL.pspdftex}}
  {\input{airbasBR.pspdftex}}
  \caption{Top left: the Julia sets of the center of the $-1/2=1/2$ limb (the basilica). Top right: the Julia set of $Q_{c}$ for $c=-1.754877666\dots$ (the airplane: $z=0$ has period $3$). In these two figures, we illustrate external rays that land at  points which map onto the $\alpha$-fixed point in at most $3$ iterates
with angles labelled so that one unit corresponds to $1/24$ of a turn around $\RZ$.
Bottom left: the geometric lamination of level $\ell=3$ of $\lambda_\alpha (Q_c)$ (as a reference the unit circle is illustrated by a broken line). Bottom right: the associated tree of level $3$. Dots corresponding to $\Gamma$-vertices are filled while $Y$-vertices are illustrated with unfilled dots.
}
  \label{fig:2}
\end{figure}

\subsection{The tree of level $\ell$ of an abstract $\alpha$-lamination}
\label{sec:treel}
The dual tree to a level $\ell$ geometric lamination is the tree of level $\ell$ of the $\alpha$-lamination. (Compare with the discussion and definitions of~\cite[Section~8.5]{McMullenBook1}.)

\begin{definition}
  Let $ {\lambda}$ be an abstract $\alpha$-lamination. 
Each connected component $C$ of $\mathcal{L}( {\lambda}^{(\ell)})$ is called a {\sf $\Gamma$-vertex of level $\ell$}. The collection of all
$\Gamma$-vertices of level $\ell$ is denoted by $\Gamma( {\lambda}^{(\ell)})$.
Each connected component $U$ of $\ponec \setminus \mathcal{L}( {\lambda}^{(\ell)})$ is called a {\sf $Y$-vertex of level $\ell$}.  The collection of all
$Y$-vertices of level $\ell$ is denoted by $Y( {\lambda}^{(\ell)})$.
The {\sf level $\ell$ tree $\mathcal{T}( {\lambda}^{(\ell)})$} is the finite
graph with vertices $\Gamma( {\lambda}^{(\ell)}) \cup
Y( {\lambda}^{(\ell)})$ such that an edge of the tree joins the vertices $U \in
Y( {\lambda}^{(\ell)})$ and $C \in \Gamma( {\lambda}^{(\ell)})$ if and only
if
$$C \cap  \partial U \neq \emptyset.$$
No edges join vertices of the same type.
\end{definition}

Note that, by the Jordan curve theorem, $\mathcal{T}( {\lambda}^{(\ell)})$ is a tree.

It follows that $\cT( {\lambda}^{(0)})$ has exactly one $\Gamma$-vertex, namely $\operatorname{Convex}(A_0) \cup M(\operatorname{Convex}(-A_0))$, $q$ $Y$-vertices,
and $q$ edges joining the $Y$-vertices to the $\Gamma$-vertices. That is, $\cT( {\lambda}^{(0)})$ is a starlike tree.

For all $\ell \geq 0$, we will consider two maps $\pi_{\ell+1}$ and $m_2$ from $\mathcal{T}( {\lambda}^{(\ell+1)})$ onto
$\mathcal{T}( {\lambda}^{(\ell)})$. One is of topological nature while the
other is of dynamical origin. 
Both will be defined on vertices and extended to the whole tree.
For short we say that a vertex of $\mathcal{T}( {\lambda}^{(\ell)})$ is a {\sf vertex of level $\ell$}.

\medskip
To simplify notation, vertices of level $\ell$ will be regarded as elements of a tree and as subsets of $\ponec$, according to convenience.
Also when $ {\lambda}$ is clear from context, we let 
$$\lamin\ell= \cL ( {\lambda}^{(\ell)}), \tree\ell= \mathcal{T}( {\lambda}^{(\ell)}), \dots$$

\begin{lemma}
\label{lemma:piell}
Let $ {\lambda}$ be an abstract $\alpha$-lamination. Consider $\ell \geq 0$ and  let $v$ be a vertex of level $\ell+1$.
Then the following hold:

\begin{enumerate}
\item There exists a unique vertex $\pi_{\ell+1}(v)$ of level $\ell$ such that $v \subset \pi_{\ell+1}(v)$.
\item There exists a unique vertex $u$ of level $\ell$ such that $m_2(v \cap \RZ) = u \cap \RZ$.
We say that $u=m_2(v)$.

\noindent
\hspace{-1.5cm}
Moreover, with the above notation:

\item If $v$ and $w$ are endpoints of an edge of level $\ell+1$, then
  $\pi_{\ell+1}(v)= \pi_{\ell+1}(w)$ or  $\pi_{\ell+1}(v)$ and $ \pi_{\ell+1}(w)$ are endpoints of
  an edge of level $\ell$.
\item
  If $v$ and $w$ are endpoints of an edge of level $\ell+1$, then $m_2(v)$ and $m_2(w)$ are endpoints of an edge of level $\ell$.
\end{enumerate}

Thus, $\pi_{\ell+1}$ and $m_2$ extend to maps from $\mathcal{T}( {\lambda}^{(\ell+1)})$ onto
$\mathcal{T}( {\lambda}^{(\ell)})$. Furthermore, $\pi_{\ell+1}$ is monotone  and,
$$m_2 \circ \pi_{\ell+2} = \pi_\ell \circ m_2.$$  
\end{lemma}

\begin{proof}
  (1) Since vertices of
  level $\ell$ are pairwise disjoint,  uniqueness of $\pi_{\ell+1}(v)$ follows. 
For the existence, note that $\Gamma$-vertices of
  level $\ell$ are also $\Gamma$-vertices of level $\ell +1$. Hence,
  a $\Gamma$-vertex $v$ of level $\ell+1$ is already a vertex of
  level $\ell$ or is disjoint from the $\Gamma$-vertices of level
  $\ell$. Thus, $\pi_{\ell+1} (v) = v$ in the former case, and 
$v$ is contained in a connected component $v'=\pi_{\ell+1} (v)$ of $\ponec \setminus \lamin\ell$, in the latter.
Similarly, every $Y$-vertex of level $\ell+1$ is contained in a $Y$-vertex of level $\ell$.

Let us prove statement (2). From the invariance property of
$\alpha$-laminations, if $v \in \gam{\ell+1}$, then $m_2(v \cap \RZ)$
is the intersection of a vertex in $\gam\ell$ with the unit
circle. Now, if $U \in \why{\ell+1}$, then the connected components in $U \cap
\RZ$ can be joined by  a sequence of hyperbolic geodesics contained in $U$ (either
through $\D$ or through $\ponec \setminus \overline{\D}$) arbitrarily close to the
boundary of $U$.  Assume that two such components of $U \cap
\RZ$, say  $]a_0, b_0[$ and
$]a_{1},b_{1}[$, can be joined by a hyperbolic geodesic inside $\D$ or
outside $\overline{\D},$ without crossing $\lamin{\ell+1}$, and  with endpoints
arbitrarily close to $b_0$ and $a_{1}$.  By the consecutive preserving
property of $ {\lambda}$ and $ {\lambda}_*$,  the intervals $]2a_0, 2b_0[$ and $]2a_{1},2b_{1}[$ may be joined  by a
hyperbolic geodesic inside $\D$ or outside $\overline{\D}$ without crossing
$\lamin{\ell}$. It follows that $m_2(U \cap \RZ)$ is
contained in a $Y$-vertex of level $\ell$, say $V$.  Moreover, the
boundary of $V$ is formed by geodesics having as endpoints the image
of endpoints of geodesics in the boundary of $U$. Hence, $m_2(U \cap
\RZ)= V \cap \RZ$.

(3) and (4), as well as the functional relation between $m_2$ and $\pi_\ell$, 
are a straightforward consequence of (1) and (2).
 \end{proof}

To show that there is a natural inclusion of trees of lower levels into the trees
of higher levels we need to establish the following property of the $\alpha$-lamination of the center of a limb.

\begin{lemma}
  \label{lem:19}
  Let $ {\lambda_*}$ be the $\alpha$-lamination of the center of a limb.
  For all $\ell \geq 0$, if $A$ is a level $\ell+1$ class which is not a level $\ell$ class, then the connected component of $\RZ \setminus A$ 
of length greater than $1/2$ contains 
all level $\ell+1$ non-trivial classes $B$ such that $B \neq A$.
\end{lemma}

\begin{remark}
  \label{rem:2}
 {\em  From Lemma A.6 in~\cite{KiwiRatLam}, given a non-trivial class $A$ of an abstract $\alpha$-lamination there exists a connected component of $\RZ \setminus A$ of length greater than $1/2$ or $m_2 :A \to m_2(A)$ is two-to-one.  Thus the existence of a connected component as in the lemma is automatically guaranteed.}
\end{remark}

\begin{proof}
  Let $Q_c$ be a quadratic polynomial which is the center of a limb.
 
 Note that the Fatou component $V$ of $Q_c$ containing the critical point
$z=0$ contains the $\alpha$ fixed point in its boundary.
  Given a non-trivial 
class $A$ of  $ {\lambda_*} = \lambda_\alpha (Q_c)$, denote by $I_A$
the connected component of $\RZ \setminus A$ of length greater than $1/2$. We claim that the fixed class $A_0$ is contained in $I_A$, provided that $A \neq A_0$. In fact, 
the extreme points of $I_A$ correspond to the arguments of external rays
which together with their common landing point $z_A$ cut the complex plane into
two regions. The one containing all the rays with arguments in $I_A$ also contains the critical point (e.g. see~\cite{MilnorPOP}), 
therefore it contains the Fatou component $V$ as well as the $\alpha$ fixed point and the rays landing at it. Thus, $A_0 \subset I_A$.
Similarly, if $A \neq A_0 +1/2$, then $A_0 +1/2 \subset I_A$.

 To prove the lemma we proceed by induction. Since the non-trivial level $1$ classes are the fixed class $A_0$
and the prefixed class $A_0+1/2$, the statement clearly holds for $\ell =0$.
  We assume that the statement is true for $\ell \geq 0$ and, by contradiction,
we prove that the statement is true for $\ell+1$. 
  That is, let $A$ be a non-trivial class of level $\ell+2$ which is not of level $\ell+1$. Assume
that there exists a level $\ell+2$ class $B \neq A$ such that
$B \subset \RZ \setminus I_A$.  It follows that $m_2$ is injective
on $ \RZ \setminus I_A$, hence $m_2(B)$ is a class of level $\ell+1$  contained in 
$ m_2(\RZ \setminus I_A)$. But since $A_0$ and $A_0 +1/2$ are contained in $I_A$ we must have that $A_0$ is contained in the interval
$\RZ \setminus m_2(\RZ \setminus I_A)$. Thus this interval is $I_{m_2(A)}$ which
is a contradiction with the inductive hypothesis since $m_2(B)$ is not contained in  $I_{m_2(A)}$.
\end{proof}

Now we show that the map $\pi_{\ell+1}$ is, in fact, a deformation retract. 
We will also prove that edges are not ``subdivided'' as we increase the level.

\begin{lemma}
\label{lemma:iell}
  Let $ {\lambda}$ be an abstract $\alpha$-lamination.  Let $\ell \geq 0$ and  $v$ be a $Y$-vertex of level $\ell$. Then there exists a unique $Y$-vertex $\iota_\ell (v)$ of level $\ell+1$ such that:
  \begin{eqnarray*}
    \pi_{\ell+1}(\iota_\ell (v)) & =& v,\\
    \partial v &\subset& \partial \iota_\ell (v).
  \end{eqnarray*}
  Given a $\Gamma$-vertex $v$ of level $\ell$ define $\iota_\ell (v) =v \in \Gamma^{(\ell+1)}$.
  Then $\iota_\ell$ extends to an injective map $\iota_\ell : \mathcal{T}( {\lambda}^{(\ell)})\to\mathcal{T}( {\lambda}^{(\ell+1)})  $ which maps edges onto edges. 
Moreover $ \pi_{\ell+1} \circ \iota_\ell = id $
and $    \iota_\ell \circ \pi_{\ell+1} $ 
is homotopic to the identity $ \operatorname{rel} \mathcal{T}( {\lambda}^{(\ell)}).$
\end{lemma}

\begin{proof}
 Consider a $Y$-vertex $v$ of level $\ell$.
From the previous lemma, it follows that two connected components of $\partial v$ cannot be separated by a connected component
of $\lamin{\ell+1}$. Hence, there exists a $Y$-vertex $\iota_\ell (v)$ of level $\ell+1$ which is obtained from $v$ after removing
finitely many topological disks completely contained in $v$. This vertex $\iota_\ell (v)$ has the desired properties.

  Observe that given any two endpoints $v \in \why\ell$ and $w\in \Gamma^{(\ell)}$ of an edge in $\tree\ell$, by definition we have
that $\partial v \cap w \neq \emptyset$. From the previous paragraph,   $\partial \iota_\ell (v) \cap \iota_\ell (w) = 
\partial \iota_\ell (v) \cap w \neq \emptyset$. Therefore, $\iota_\ell (v)$ and  $\iota_\ell (w)$ are endpoints of an edge in 
$\tree{\ell+1}$. That is, $\iota_\ell$ extends to a map between the corresponding trees. It is not difficult to see
that $\iota_\ell$ has the desired properties.
\end{proof}

Note that  $\iota_\ell$ is an inclusion of trees. So one may regard the trees $\mathcal{T}^{(\ell)}$ as an increasing sequence of trees.
It is important to stress that edges of level $\ell$ are not subdivided when included in the tree of level $\ell+1$ via the map $\iota_\ell$.

\subsection{The full tree of an $\alpha$-lamination}
\label{sec:tree.full}
From Lemma~\ref{lemma:piell}, we conclude that
 $$(\mathcal{T}( {\lambda}^{(\ell+1)}), \pi_{\ell+1})$$ is an inverse system of (contractible) topological spaces.

 \begin{definition}
   Let $ {\lambda}$ be an abstract $\alpha$-lamination.
   The {\sf full tree of  $ {\lambda}$} is
$$\mathcal{T}^{\infty} ( {\lambda}) = \lim_{\longleftarrow} (\mathcal{T}^{(\ell+1)}( {\lambda}), \pi_{\ell+1}).$$
 \end{definition}

From Lemma~\ref{lemma:piell}, the action of $m_2$ extends to a continuous map 
$$m_2 : \mathcal{T}^{\infty} ( {\lambda}) \to \mathcal{T}^{\infty} ( {\lambda}).$$

We will show that in a great variety of cases the action of a quadratic rational map $\varphi : \pberl \to \pberl$
on the (Berkovich) convex hull of its Julia set is topologically conjugate to 
$m_2 : \mathcal{T}^{\infty} ( {\lambda}) \to \mathcal{T}^{\infty} ( {\lambda})$ for an appropriate $ {\lambda}$. 
The Julia set of $\varphi$ will correspond to the inverse limit of $Y$-vertices.

\subsection{Construction of $\alpha$-laminations via symbolic dynamics}
\label{section:symbolic}
Our aim here is to show how via symbolic dynamics of $m_2: \RZ \to \RZ$ 
we may construct abstract $\alpha$-laminations in the $p/q$-limb. 
Recall that  $I_{p/q}$ denotes the characteristic interval  of the $p/q$-limb
(see Definition~\ref{def:9}).
For each argument $\theta \in \overline{I_{p/q}} $  we will produce at least one and at most three abstract
$\alpha$-laminations.

For any $\theta \in \overline{I_{p/q}}$ we consider two partitions $\{ I_0^-(\theta), I_1^-(\theta)\}$ and $\{ I_0^+(\theta), I_1^+(\theta)\}$ 
of $\RZ$ into semicircles where
$$I_0^+ (\theta)=  [\theta/2 +1/2, \theta/2[, \quad\quad  I_0^- (\theta) = ]\theta/2+1/2, \theta/2],$$
and
$$I_1^+ (\theta) = [\theta/2, \theta/2 + 1/2[, \quad\quad  I_1^- (\theta) = ]\theta/2, \theta/2 + 1/2].$$
It follows that $A_0$ is contained in exactly one element of each partition, for all $\theta \in I_{p/q}$.

We will construct abstract $\alpha$-laminations using the itinerary of $m_2$-orbits according to each one of the two partitions of $\RZ$.
Let $\Sigma = \{ 0, 1 \}^{\N \cup \{0\}}$ and, for $\epsilon = +$ or $-$,
define
$$\itin^\epsilon_\theta (t) = (i_0, i_1, \dots) \in \Sigma$$
if $m_2^k(t) \in  I_{i_k}^\epsilon (\theta)$.
The equivalence relation $ {\lambda}^\epsilon (\theta)$ is the relation that identifies two distinct arguments
$s$ and $t$  if and only if $s, t$ belong to the $m_2$-grand orbit of $A_0$ and $\itin^\epsilon_\theta (t) = \itin^\epsilon_\theta (s)$.

One may produce a new equivalence relation from $ {\lambda}^{\pm}(\theta)$. Namely, we let $ {\lambda} (\theta)$ be the smallest equivalence relation
that contains both $ {\lambda}^+(\theta)$ and $ {\lambda}^-(\theta)$. 
If $m_2^n (\theta) \notin A_0$ for all $n \geq 0$, then it is fairly easy to check that $ {\lambda}(\theta) =  {\lambda}^+(\theta)=  {\lambda}^-(\theta)$.

\begin{lemma}
  If $\theta \in I_{p/q}=]\theta_0,\theta_1[$, then  $ {\lambda}(\theta)$, $ {\lambda}^+(\theta)$ and $ {\lambda}^-(\theta)$ 
are abstract $\alpha$-laminations in the $p/q$-limb.
Moreover, denote by $ {\lambda}_*$ the $\alpha$-lamination of the center of the $p/q$-limb. Then $ {\lambda}_* =  {\lambda}^+(\theta_1) =  {\lambda}^-(\theta_0)$.
\end{lemma}

\begin{proof}
We start proving that $ {\lambda}^\pm(\theta)$ are abstract $\alpha$-laminations, for all $\theta \in I_{p/q}$.
For $j=0,1$, the map $m_2: I_j^{\pm}(\theta)\to \RZ$ preserves cyclic order and
is onto. Hence, $ {\lambda}^\pm(\theta)$ satisfies the  invariant and consecutive preserving property.
Thus, if $m_2^{\ell}(t)$ is periodic, for some $\ell \ge 0$, then the same holds for all arguments $s$ which are
$ {\lambda}^\pm(\theta)$-equivalent to $t$. It  follows that $ {\lambda}^\pm(\theta)$-classes are finite.
A pair of distinct classes $B_1$ and $B_2$ 
are either contained in the same half circle or unlinked.
In the former case, $B_1$ and $B_2$ are unlinked
if and only if $m_2(B_1)$ and $m_2(B_2)$ are unlinked.
The unlinked property for $ {\lambda}^\pm(\theta)$ follows.

Now we show that $ {\lambda}(\theta)$ is an abstract $\alpha$-laminations, for all $\theta \in I_{p/q}$.
The only relevant case is when $\theta \notin A_0$ and $m_2^{\ell_0}(\theta) \in A_0$ for some
$\ell_0 \geq 1$.  The $ {\lambda}(\theta)$-class of
$\theta/2$ is the union of its $ {\lambda}^+(\theta)$-class with the
$ {\lambda}^+(\theta)$-class of $\theta/2+1/2$. It follows that a
$ {\lambda}(\theta)$-class consists of points with the same itinerary or
it eventually maps onto the class of $\theta/2$ through a cyclic order
preserving map. Therefore, we may apply a similar reasoning to
conclude that $ {\lambda}(\theta)$ is an abstract $\alpha$-lamination. (According
to~\cite{bfh,OrsayNotes} the lamination $ {\lambda}(\theta)$ is the
$\alpha$-lamination of a quadratic polynomial where the critical
point eventually maps to the $\alpha$ fixed point. That is, a critically prefixed lamination.)

Finally, the $\alpha$-lamination $ {\lambda}_*$ of the quadratic polynomial in the $p/q$-limb with a period $q$ critical orbit
is
$ {\lambda}_* =  {\lambda}^{+}(\theta_1)= {\lambda}^{-}(\theta_0)$, according to 
Poirier's description~\cite{PoirierCP}.
\end{proof}

With little more work, given an abstract $\alpha$-lamination $ {\lambda}$, it is possible to find $\theta$ such that
one of the laminations $ {\lambda}(\theta),   {\lambda}^\pm(\theta)$ coincides with $ {\lambda}$:

\begin{proposition}
  \label{pro:3}
   Let $ {\lambda}$ be an abstract $\alpha$-lamination in the $p/q$-limb. Denote by $I_{p/q}$ the corresponding characteristic 
interval. 
   Then there exists $\theta \in \overline{I_{p/q}}$ such that $ {\lambda} =  {\lambda}(\theta)$ or,  $ {\lambda} =  {\lambda}^+(\theta)$ or,   $ {\lambda} =  {\lambda}^-(\theta)$ 
\end{proposition}

First we show that given an abstract $\alpha$-lamination $ {\lambda}$ and a level $\ell$, we may  find
$\theta \in \RZ$ such that the level $\ell$ restriction of $ {\lambda}(\theta)$ and $ {\lambda}$ coincide:

\begin{lemma}
\label{lem:1}
  Let $ {\lambda}$ be an abstract $\alpha$-lamination. Denote by $\lam\ell$ its
  level $\ell$ restriction.  Consider $$I^{(\ell)} = \{ \theta \in \RZ
  \mid \operatorname{Convex}(\{\theta/2, \theta/2+1/2\}) \cap
  \cL^{(\ell)}= \emptyset \}.$$
Then exactly one of the following holds:

\begin{enumerate}
\item
 $I^{(\ell)} \neq \emptyset$. In this case,  $$ {\lambda}^+ (\theta)^{(\ell)} =  {\lambda}^- (\theta)^{(\ell)} = \lam\ell$$
for all $\theta \in  I^{(\ell)}$.
\item
$I^{(\ell)} = \emptyset$ and there exists a $\lam\ell$-class $B$ such that $m_2: B \to m_2(B)$ is two-to-one. In this case,
for all $\theta \in m_2(B)$, $$ {\lambda} (\theta)^{(\ell)} = \lam\ell.$$
\end{enumerate}
\end{lemma}

\begin{proof}
From Lemma A.6 in~\cite{KiwiRatLam}, either $I^{(\ell)} \neq \emptyset$ or there exists a $\lam\ell$-class $B$ such that $m_2: B \to m_2(B)$ is two-to-one.

In case (1), since   $\operatorname{Convex}(\{\theta/2, \theta/2+1/2\})$ is disjoint from $\lamin\ell$, we have that each $\lam\ell$-class is 
contained in $]\theta/2,\theta/2+1/2[$ or in $]\theta/2+1/2,\theta/2[$. Therefore, $\lam\ell \subset  {\lambda}^\pm (\theta)^{(\ell)}$.
Taking into account that every class of $\lam\ell$ as well as every class of $ {\lambda}^\pm (\theta)^{(\ell)}$ has exactly $q$ elements, 
we conclude that $\lam\ell =  {\lambda}^\pm (\theta)^{(\ell)}$.

\medskip
In case (2), since $\{ \theta/2, \theta/2+1/2 \} \subset B$, each  $\lam\ell$-class different from $B$ is contained  in $]\theta/2,\theta/2+1/2[$ or in $]\theta/2+1/2,\theta/2[$. Hence, as in the previous case, 
every $\lam\ell$-class that is not eventually mapped onto $B$ is a $ {\lambda} (\theta)^{(\ell)}$-class.

Note that $\itin_\theta^+(\theta/2) = \itin_\theta^-(\theta/2+1/2)$. 
Since  $\itin_\theta^+(\theta/2)$ coincides with
$\itin_\theta^+(t)$ for all $t \in B \cap I^+_1(\theta)$ and $\itin_\theta^+(\theta/2+1/2)$ coincides with
$\itin_\theta^+(t)$ for all $t \in B \cap I^+_0(\theta)$, it follows that $B$ is contained in a $ {\lambda} (\theta)^{(\ell)}$-class.
Taking into account that $ {\lambda} (\theta)^{(\ell)}$-classes have at most $2q$ elements, it follows that $B$ is a $ {\lambda}(\theta)^{(\ell)}$-class.

Now we assume that $A \neq B$ is a $ {\lambda}^{(\ell)}$-class such that $m_2(A)$ is a $ {\lambda} (\theta)^{(\ell)}$-class. We claim that $A$ is a $ {\lambda} (\theta)^{(\ell)}$-class.
By the unlinked property of $ {\lambda}$, the class $A$ is contained in the interior of $I^+_0(\theta)$ or of $I^+_1(\theta)$. 
Again by the unlinked property, if $A \ni t$ and  $t \in I^+_0(\theta)$ (resp.  
$t \in I^+_1(\theta)$) then the $ {\lambda} (\theta)^{(\ell)}$-class $A'$ of $t$ is contained in $I^+_0(\theta)$ (resp. $I^+_1(\theta)$).
By the invariance property of both laminations, $m_2(A') = m_2(A)$. Hence, $A' = A$. 

Since every $\lam\ell$-class eventually maps to $B$ or it is $ {\lambda} (\theta)^{(\ell)}$-class, it follows that $\lam\ell =  {\lambda} (\theta)^{(\ell)}$.
\end{proof}

\begin{proof}[Proof of Proposition~\ref{pro:3}]
Given an abstract $\alpha$-lamination $ {\lambda}$ in $p/q$-limb, for all $\ell \ge 0$, consider the decreasing collection $\{I^{(\ell)}\}$ of subset of $\RZ$ as in the previous lemma.
If $I^{(\ell)} \neq \emptyset$ for all $\ell$, choose  $\theta_\ell \in I^{(\ell)}$ so that the sequence $\{ \theta_\ell \}$ is monotone.
Without loss of generality we assume that this sequence converges to $\theta$ and, to fix ideas, suppose that it is increasing (with respect to the order in the interval $]\theta-1,\theta[$). It follows that, for all $t \in \RZ$,
 $$\lim_{\ell \to \infty} it^+_{\theta_\ell} (t) = it^+_\theta (t).$$
In particular, $ {\lambda}^+ (\theta)^{(\ell)} = \lam\ell$ for all $\ell \ge 0$. Therefore, $ {\lambda}^+(\theta) =  {\lambda}$.

Now if  $I^{(\ell_0)} = \emptyset$ for some $\ell_0$, then $B$ as in part (2) of the previous lemma coincides for all $\ell \ge \ell_0$. 
Choosing $\theta \in B$  it follows that $ {\lambda}(\theta)=  {\lambda}$.
\end{proof}

\begin{remark}
  \label{rem:1}
  {\em According~\cite{OrsayNotes}, given $\theta \in I_{p/q}$, if $c \in \cM$ is the accumulation point of the parameter ray $R^\theta_{\cM}$, then 
$\lambda_\alpha(Q_c) =  {\lambda}(\theta)$. Moreover, every $\alpha$-lamination is the $\alpha$-lamination of a quadratic polynomial which is
the accumulation point of some parameter ray. 
  It follows that an abstract $\alpha$-lamination $ {\lambda}$   is not 
an $\alpha$-lamination if and only if $ {\lambda} =  {\lambda}^+(\theta)$
or $ {\lambda} =  {\lambda}^-(\theta)$ for some $\theta \in I_{p/q}$ which eventually maps into the fixed class.}
\end{remark}

We will also need to establish the following properties of the 
$\alpha$-lamination of the center of the $p/q$-limb.

\begin{lemma}
  \label{lem:2}
  Let $ {\lambda}_*$ be the lamination of the center of the $p/q$-limb.
  For $\ell \ge 0$, denote by $\tree\ell_*$ its tree of level $\ell$ and by $\cvpar\ell_*(x)$ the vertex containing $x \in \C \cup \{ \infty \}$. 
  Then, for all $\ell \ge 0$, 
$$\cvpar{\ell+1}_* (0) = \cvpar{\ell+1}_* (\infty).$$
  Moreover, for  all $\theta \in m_2(\cvpar{\ell+1}_* (0))$, we have that
$$\tree{\ell+1}_* = \tree{\ell+1}(\theta).$$
\end{lemma}

\begin{proof}
First we observe that 
$A$ is a class of the  $\alpha$-lamination of the center of the $-p/q$
if and only if $-A$ is a class of $\bla_*$. It follows that
$\cL(\bla^{(\ell+1)}_*)$ is invariant under $1/\bar{z}$.

Consider $\theta \in m_2(\cvpar{\ell+1}_* (0))$. 
Since $\theta/2$ and $\theta/2 +1/2$ lie in $\cvpar{\ell+1}_* (0)$, 
there is a curve $\gamma \subset \cvpar{\ell+1}_* (0)$ connecting $\theta/2$ and $\theta/2 +1/2$ formed by concatenating geodesics 
in $\D$ or outside $\overline{\D}$. 
By the invariance of $\cL(\bla^{(\ell+1)}_*)$ under $1/\bar{z}$, we may replace
the geodesics outside $\overline{\D}$ in $\gamma$ by geodesics contained in
$\D$. Thus, we can connect $\theta/2$ and $\theta/2 +1/2$ through a sequence
of geodesic paths contained in $\cvpar{\ell+1}_* (0) \cap \D$. By convexity, it follows that 
the diameter
joining $\theta/2$ and $\theta/2 +1/2$ is disjoint
from $\cL(\bla^{(\ell+1)}_*)$. By $1/\bar{z}$-invariance, $\cvpar{\ell+1}_* (0) = \cvpar{\ell+1}_* (\infty)$. 
From Lemma~\ref{lem:1} we obtain that $\bla^{(\ell+1)}_* =\bla^{(\ell+1)} (\theta)$.  
\end{proof}

\subsection{Branched maps of trees}
\label{sec:branched-maps-trees}
Multiplication by $2$ on $\RZ$ acts on level $\ell$ trees as a two-to-one map ramified over an interval. 

\begin{definition}
\label{def:3}
  Consider finite simplicial trees $\cT, \cT'$ and let $g: \cT \to \cT'$ be a simplicial map.
Let $I \subset \cT'$ be a simplicial subcomplex homeomorphic to a closed  interval or a singleton.
We say that $g$ is a {\sf degree two map branched over $I$} if the following holds: there exists a
tree automorphism $\gamma: \cT \to \cT$ of order $2$ (i.e. a simplicial involution) 
such that $I=g(\operatorname{Fix}(\gamma))$, and $g(x) = g(y)$ if and only if  $x=y$ or $y=\gamma(x)$. We say that $J=g^{-1}(I)$ is the {\sf critical interval of $g$} and $I$ is the {\sf critical value interval of $g$}.
\end{definition}

As mentioned above our definition is tailored to include the action of multiplication
by $2$ on trees associated to $\alpha$-laminations.

\begin{lemma}
\label{lem:6}
  Let $ {\lambda}$ be an abstract $\alpha$-lamination and for $\ell \geq 0$, let 
$\tree\ell$ be the associated level $\ell$ tree. 
Denote by $v_{\ell}(0)$ and $v_{\ell}(\infty)$ the vertices of $\tree\ell$ containing 
$0$ and $\infty$, respectively.
Then $m_2: \tree{\ell+1} \to \tree\ell$ is a degree $2$ map branched map over the interval 
$I=[m_2(v_{\ell+1}(0)),m_2(v_{\ell+1}(\infty))]$.
\end{lemma}

\begin{proof}
  Recall that the vertices of $\tree\ell$ are subsets of $\ponec=\C
  \cup \{\infty\}$.  For all $\ell \ge 0$, we have that $v$ is a
  vertex of $\tree{\ell+1}$ if and only $-v$ is a vertex of
  $\tree{\ell+1}$. Moreover, $m_2(v) = m_2(-v)$.

Now let $\gamma: \tree{\ell+1} \to \tree{\ell+1}$ be the involution induced by 
$v \mapsto -v$. We claim that given a level $\ell+1$ vertex $v$, we have that $v=-v$ if and only if one of the following occurs:
\begin{enumerate}
\item $\partial v$ separates $0$ and $\infty$.
\item $0 \in v$.
\item $\infty \in v$.
\end{enumerate}
The claim is rather immediate when $\{ 0, \infty \} \cap v \neq \emptyset$ so we assume that  $\{ 0, \infty \} \cap v = \emptyset$.
We let $W$ denote the connected component of $\ponec \setminus v$ that contains $\infty$. 
Now $v=-v$ if and only if $\partial W = \partial (-W)$ which implies that $0$ and $\infty$ are separated by  $\partial W$.
Conversely, if $\partial W$ separates $0$ and $\infty$, then $\partial W \cap \partial (-W) \neq \emptyset$. Since $\partial W$ and $\partial (-W)$ are
equal or disjoint the claim follows. 

The vertices that separate $0$ from $\infty$ together with the ones that contain $0$ or $\infty$ form a (possibly degenerate)
interval, and the lemma follows.
\end{proof}

We will need the following ``lifting'' property in the process of
establishing a topological conjugacy between maps acting on trees
associated to abstract $\alpha$-laminations and the dynamics of some degree two
rational maps acting on the convex hull of their Julia sets in
$\pberl$.

\begin{lemma}
  \label{lifting.lemma}
  Consider simplicial trees $\cA'$ and $ \cT'$ with subtrees $\cA \subset \cA'$ and $\cT \subset \cT'$.
  Suppose that $\vphi: \cA' \to \cA$ and $m: \cT' \to \cT$ are degree two branched maps over $\cI_\cA$ and $\cI_\cT$, respectively,
 such that  $\vphi(\cA) \subset \cA$ and $m(\cT) \subset \cT$.
  Also, suppose that $\vphi: \cA \to \vphi(\cA)$ and $m: \cT \to m(\cT)$ are degree two branched maps.

  Assume that $h: \cA \to \cT$  is a tree isomorphism such that the following hold:
  \begin{enumerate}
  \item The following diagram conmutes:
  $$
\begin{CD}
\cA @>\vphi>> \cA\\
@VhVV   @VhVV     \\
\cT @>m>> \cT.
\end{CD}
$$
  \item $h (\cI_\cA) = \cI_\cT$.
  \end{enumerate}
Then there exists a tree isomorphism $h': \cA'\to \cT'$ such that the following hold:
\begin{enumerate}
\item The following diagram conmutes:  
$$
\begin{CD}
\cA' @>\vphi>> \cA\\
@Vh'VV   @VhVV     \\
\cT' @>m>> \cT.
\end{CD}
$$
\item $h' \big|_\cA = h.$
\end{enumerate}
\end{lemma}

\begin{proof}
  Without loss of generality we may assume that $\cA$ and $\cT$ contain $\vphi^{-1}(\cI_\cA)$ and $m^{-1}(\cI_\cT)$, since both
$\vphi$ and $m$ are bijections over these sets.
  Now consider two connected components $\cA_0, \cA_1$ of $\cA' \setminus \cA$ such that $\vphi(\cA_0) = \vphi(\cA_1)$. Let $\{ x_j \} = \overline{\cA_j} \cap \cA$
for $j=0,1$. 
Note that $m$ has two inverse branches $m_0$ and $m_1$ defined on $h(\vphi(\overline{\cA_0}))$ and $h(\vphi(\overline{\cA_0}))$ which are continous
bijections into $\cT'$ such that $h(x_j) = m_j(h(\vphi(x_j)))$. Define $h'$ on $\cA_j$ as $h' = m_j \circ h \circ \vphi$. It is not difficult to check that
$h'$ has the desired properties.
\end{proof}

\section{Topological model}
\label{sec:TopModel}
For any subset $X \subset \pberl$ we denote by $\widehat{X}$ its convex hull in $\pberl$. Namely $\widehat{X}$ is the union of all arcs $[x_0,x_1]$ where
$x_0, x_1 \in X$.

Here we give a complete description of $\hatjuliavphi$ for quadratic rational maps possessing a non-classical repelling periodic orbit.
We will show that $\hatjuliavphi$ is completely invariant under $\vphi$ (see~Lemma~\ref{lem:10} below).
Moreover, we will describe the structure and dynamics over $\hatjuliavphi$ with the aid of the tree associated to an appropriate  abstract 
$\alpha$-lamination (see Section~\ref{section|laminations}).

The (simultaneous) proof of the following two propositions is given in Section~\ref{sec:TopConj}

\begin{proposition}
\label{bothin}
Let $\vphi$ be a quadratic rational map over $\L$ which is not simple.
Assume that $\vphi$ has a non-classical repelling periodic orbit and
that both critical points belong to the filled Julia set.
  Then there exists an abstract
  $\alpha$-lamination $ {\lambda}$ which is not critically prefixed such that $\vphi : \hatjuliavphi \to \hatjuliavphi$ is topologically
  conjugate to $m_2 : \tree\infty ( {\lambda}) \to
  \tree\infty( {\lambda})$.
\end{proposition}

\begin{definition}
  \label{def:11}
Let $\vphi$ be a quadratic rational map over $\L$ which is not simple.
Assume that $\vphi$ has a non-classical repelling periodic orbit
and that one critical point $\omega$ is not in $\filledvphi$.
  Denote by $\pi_{U_0}: U_0 \to \cA_0$ the projection of the fixed Rivera domain $U_0$ into its skeleton (see Section~\ref{sec:PointsBallsAffinoids}).
  Then,
    $$\cO_0 = \pi_{U_0} (\cO(\omega) \cap U_0)$$ is a periodic orbit 
(see Section~\ref{sec:PeriodicFatou}).
We define $\sim$ to be the equivalence relation in $\hatjuliavphi$ that identifies two distinct elements
$\z_1, \z_2$ if and only if $\z_1, \z_2$ lie in the same connected component of $\vphi^{-n}(\widehat{\cO_0})$ for some 
$n \ge 0$. 
\end{definition}

\begin{proposition}
\label{pro:1}
Let $\vphi$ be a quadratic rational map over $\L$ which is not simple.
Assume that $\vphi$ has a non-classical repelling periodic orbit
and that one critical point $\omega$ is not in $\filledvphi$.
Then the following map is well defined:
$$\begin{array}{cccc}
\psi : & \hatjuliavphi/\sim & \to & \hatjuliavphi/\sim \\
       &  [\z]              & \mapsto & [{\vphi(\z)}]
\end{array}
$$
where $[\z]$ denotes the $\sim$-class of $\z$.

Moreover, there exists a critically prefixed $\alpha$-lamination $\lambda$  and 
 a homeomorphism $h: \hatjuliavphi/\sim \to \tree\infty (\lambda)$ 
which conjugates the action of $\psi$ with that of $m_2$ (i.e. $m_2 \circ h = h \circ \psi$).
\end{proposition}

\begin{remark}
  {\em Note that $\sim$ classes are either trivial or contained in a Fatou component which eventually maps onto $U_0$. Thus $h$ is a conjugacy over the Julia set. Sometimes $\widehat{\cO_0}$ will be simply the singleton $\{\veta_0\}$ (see Lemma~\ref{basics.repelling}). In that case, $\sim$ is the trivial relation and the above proposition gives a topological model for the dynamics over $\hatjuliavphi$.} 
\end{remark}

In Section~\ref{sec:Julial} we show that $\hatjuliavphi$ is completely invariant and simultaneously introduce an increasing sequence of trees $\arbol\ell \subset \hatjuliavphi$ whose union is dense and such that $\vphi: \arbol{\ell+1} \to \arbol{\ell}$ is a degree two branched map. The construction of the (semi)conjugacies will rely on choosing a decreasing sequence of subsets $\cvpar\ell_\cM$ of $\RZ$ so that for all $\theta \in \cvpar\ell_\cM$ we are able to construct (semi)conjugacies with domain $\arbol\ell$ and range $\tree (\lambda(\theta)^{(\ell)}$. In Section~\ref{sec:collapsed} we prepare to treat at once the cases in which $\crit(\vphi) \subset \filledvphi$ and $\crit(\vphi) \not\subset \filledvphi$. To spread conjugacies defined on $\arbol\ell$ to $\arbol{\ell+1}$ we employ Lemma~\ref{lifting.lemma} but before we need some control over the branched intervals. This control is achieved in Section~\ref{sec:postcritical} by introducing appropriate subsets of the trees involved which contain the ``post-branched'' points. Finally, Section~\ref{sec:TopConj} contains the proof of propositions~\ref{bothin} and~\ref{pro:1} as well as the proof of Theorem~\ref{ithr:dynamics}.

\subsection{Trees of level $\ell$ and $\hatjuliavphi$}
\label{sec:Julial}
The proofs of propositions~\ref{bothin} and~\ref{pro:1} rely on the following basic facts about $\hatjuliavphi$. 

{\sf We will continue to work under the standing assumption that the quadratic rational map $\vphi$ 
has a fixed Rivera domain $U_0$ whose boundary is a period $q >1$ repelling
orbit $\cO$.
We will use the notation of Lemma~\ref{basics.repelling}. Also we let $\omega$ and $\omega'$ be the critical points of $\vphi$ where
$\omega$ is the active critical point (see Definition~\ref{def:5}). Furthermore, let $\crit (\vphi) = \{ \omega, \omegap \}$.}

\begin{lemma}
\label{lem:10} 
Let $\arbol0$ be the skeleton of $U_0$.
For $\ell \ge 0$ recursively define,
$$\arbol{\ell+1} = \vphi^{-1}(\arbol\ell).$$
The following statements hold for all $\ell \ge 0$: 
\begin{enumerate}
\item $\arbol\ell \subset \arbol{\ell+1}$.
\item $\arbol\ell$ is the convex hull of $\vphi^{-\ell}(\cO)$, in particular, $\arbol\ell$ is connected.
%% \item There exists a unique retraction  $\pia\ell: \arbol{\ell+1} \to \arbol\ell$  between these trees.
\item If $x \in \vphi^{-(\ell+1)}(\cO) \setminus \vphi^{-\ell}(\cO)$, then $x$ is an endpoint of $\arbol{\ell+1}$.
\item For all $\z \in \arbol\ell \setminus GO(\cO)$, if $\dd$ is a direction at $\z$ disjoint from $\arbol\ell$,
then $\vphi^\ell(\dd) \subset U_0$.
\item $$\widehat{\juliavphi} = \overline{\bigcup_{\ell \geq 0}  \arbol\ell} \supset \overline{GO(\cO)}= \juliavphi.$$
\item $\hatjuliavphi$ is completely invariant.
\end{enumerate}
\end{lemma}

\begin{proof}
  Statement (1) is a consequence of the forward invariance of the intial skeleton (i.e. $\vphi(\arbol0) = \arbol0$).

For (2) we claim that $\arbol\ell$ is connected for all $\ell$ and prove this claim by induction.
Since $\vphi$ is an open map, every connected component of $\arbol{\ell+1} =\vphi^{-1}(\arbol\ell)$ maps
onto $\arbol\ell$, but the preimage of the singleton $\partial B_1 \subset \arbol0$ is the singleton $\partial B_0$. Hence, $\arbol{\ell+1}$ is connected
and the claim follows. Now the endpoints of $\arbol0$ are contained in $\cO$, hence the endpoints of
$ \arbol\ell$ are contained in $\vphi^{-\ell}(\cO)$. Since every closed and  connected subset of $\pberl$ is convex, we have established (2).

For (3) we proceed by contradiction.
Suppose that  $x \in \vphi^{-(\ell+1)}(\cO) \setminus \vphi^{-\ell}(\cO)$
is not an endpoint of $\arbol{\ell+1}$. Let $C$ be a connected component
of $\arbol{\ell+1} \setminus \{ x \}$ not containing $\arbol\ell$.
Then $\overline{C}$ contains a component $\cA_U$ of $\vphi^{-(\ell+1)}(\arbol0)$ 
such that $x \in  \cA_U$. All the endpoints of $\cA_U$ lie
in $\arbol{\ell+1} \setminus \arbol\ell$. Hence,
$\vphi^{\ell}(\cA_U) \subset \arbol{1} \setminus \arbol{0}$ but then
$\arbol{0} = \vphi^{\ell+1}(\cA_U)$ would be a subset of $\vphi(\arbol{1} \setminus \arbol{0}) = \arbol{0} \setminus \{ \xi_1 \}$, which is a contradiction.

We prove (4) by induction. In fact, for $\ell =0$ the statement trivially holds. For $\ell \ge 1$,  given a direction $\dd$, as in (4),
 we have that $\dd$ does not contain $\xi_0$, and therefore is a good direction at $\z$. It follows that $\vphi(\dd)$ is a direction disjoint from $\arbol{\ell-1}$ at $\vphi(\z) \in \arbol{\ell-1} \setminus GO(\cO)$. By the inductive hypothesis, $\vphi^{\ell-1}(\vphi(\dd))$ is contained in $U_0$.

For (5), note that by (2) we have $$\widehat{\juliavphi} \supset \overline{\bigcup_{\ell \geq 0}  \arbol\ell} \supset \overline{GO(\cO)}= \juliavphi,$$
since the closure of a convex set is convex. Taking the convex hulls of these three sets,
the first $\supset$ above may be replaced by equality.

In view of (1), $\overline{\cup  \arbol\ell}$ is completely invariant. Hence part (6) follows from (5).
\end{proof}

\begin{definition}
  \label{def:7}
  For all $\ell \ge 0$, let $\cwpar\ell \in \arbol\ell$ be such that
$$[\xi_0, \cwpar\ell] = [\omega',\omega] \cap \arbol\ell.$$
\end{definition}

 By Lemma~\ref{degree-one-l}, given  $x \in \arbol\ell$ we have  that $\deg_x \vphi =2$ if and only if $x \in [\xi_0, \cwpar\ell]$. Throughout, we let $\cvpar\ell = \vphi(\cwpar{\ell+1})$.

\begin{lemma}
  \label{lem:36}
If $\vphi$ has a critical point $\omega \notin \filledvphi$, let $\cO_0 = \pi_{U_0} (\cO(\omega)\cap U_0)$, otherwise let $\cO_0$ be the emptyset.

  \begin{enumerate}
  \item $\arbol\ell$ is a topological tree such that its topological ramification points as well as endpoints are contained in 
$$\cVpar{\ell}= \left(GO(\cO) \cup GO(\veta_0) \cup GO(\cO_0)\right) \cap  \arbol\ell.$$
  \item For all $\ell \ge 0$, regarding $\arbol\ell$ as simplicial trees with vertices $\cVpar{\ell}$, the map $$\vphi: \arbol{\ell+1} \to \arbol\ell$$
is a degree two branched map over  
$$\cIpar\ell_\cA = [\xi_1,\cvpar\ell]=[\vphi(\omegap),\vphi(\omega)]\cap\arbol\ell,$$
where $v^{(\ell)} = \vphi(\cwpar{\ell+1}) \in \cVpar{\ell}$.
  \item $\intervala\ell \subset \intervala{\ell+1}$. Moreover, if $\intervala\ell \subsetneq \intervala{\ell+1}$, then $\cvpar{\ell+1} \in GO(\cO_0)\cup GO(\veta_0)$ or $\cvpar{\ell+1}$ is an endpoint of $\arbol{\ell+1}$.
  \item If $\crit(\vphi) \subset \filledvphi$, then $\cvpar\ell \in GO(\cO)$ for all $\ell$. In this case, there exists a minimal $\ell_0 \ge 0$ such that $\cvpar\ell = \cvpar{\ell_0}$ for all $\ell \ge \ell_0$ if and only if both critical points belong to a Fatou component $U$ such that $\partial U = \{ \cvpar{\ell_0} \}$. 
  \item If $\crit(\vphi) \not\subset \filledvphi$, then there exists $\ell$ such that $\cvpar\ell \in GO(\cO_0)$. In this case, if $\ell_0$ is the minimal $\ell$ such that  $\cvpar\ell \in GO(\cO_0)$,
then $\cvpar\ell = \cvpar{\ell_0}$ for all $\ell \ge \ell_0$.
  \end{enumerate}
\end{lemma}

\begin{proof}

  Observe that $\omega$ lies in a direction $\dd$ at $\cwpar\ell$ which is disjoint from $\arbol\ell$.
 First we prove that $\cwpar\ell \in \cVpar\ell$. Indeed, if $\cwpar\ell \notin GO(\cO)$, then $\vphi^\ell(\dd) \subset U_0$, by Lemma~\ref{lem:10} (4).
Therefore, $\vphi^{\ell}(\cwpar\ell) \in \cO_0$ and $\cwpar\ell \in \cVpar\ell$.

  Assertion (1) is true for $\ell=0$ and we proceed by induction to establish the assertion for arbitrary $\ell$.  
Suppose that $x$ is a topological ramification point of
  $\arbol{\ell+1}$ and $\vphi(x)$ is not a topological ramification point of
  $\arbol\ell$. Then the degree of $\vphi$ at $x$
  must be two. Thus, the critical points belong to different directions at $x$.
  At least one of these directions is disjoint from $\arbol{\ell+1}$, otherwise
  there would be at least three directions at $\vphi(x)$ containing points of $\arbol\ell$.
  It follows that $x$ is an endpoint of $[\omega, \omegap] \cap \arbol{\ell+1}$.
  Therefore, $x= \xi_0$ or $x = \cwpar{\ell+1}$ which belong to $\cV^{(\ell+1)}$.

  For (2), let $\gamma: \arbol{\ell+1} \to \arbol{\ell+1}$ be the involution defined by 
$\gamma(x) = x'$ if $\{x,x'\} = \vphi^{-1}(\vphi(x))$. It is not difficult to check that $\gamma$ respects the simplicial structure of $\arbol{\ell+1}$.
Moreover, $\gamma(x) = x$ if and only if the degree of $\vphi$ at $x$ is $2$, which is equivalent to $x \in [\omegap,\omega] \cap \arbol{\ell+1} = [\xi_0, \cwpar{\ell+1}]$. Since $\vphi$ restricted to $[\omegap,\omega]$ is a bijection and 
$\cwpar\ell \in \cVpar\ell$,
assertion (2) follows.

  From (2) we have that $\intervala\ell \subset \intervala{\ell+1}$. To prove (3), assume that $\cvpar{\ell+1} \neq \cvpar\ell$ and that
$\cvpar{\ell+1} \in GO(\cO)$. Hence, $\cvpar{\ell+1} \notin \arbol\ell$. Thus $\cvpar{\ell+1} \in \vphi^{-(\ell+1)}(\cO) \setminus \vphi^{-\ell}(\cO)$. Therefore, $\cvpar{\ell+1}$ is an endpoint of $\arbol{\ell+1}$, by Lemma~\ref{lem:10} (3).

  For (4), if $\crit (\vphi) \subset \filledvphi$, then $\cwpar{\ell+1} \notin GO(\cO_0) \cup GO(\veta_0)$, by Lemma~\ref{lem:10}(4). Thus, $\cwpar{\ell+1} \in GO(\cO)$ and $\cvpar\ell = \vphi(\cwpar{\ell+1}) \in GO(\cO)$. In this case, assume that
 $\cvpar\ell = \cvpar{\ell_0}$ for all $\ell \ge \ell_0$. Let  $\dd$ be the direction at at $\cvpar{\ell_0}$ containing a critical value. 
It follows that there are no points of $GO(\cO)$ in $\dd$, for otherwise, $\cvpar\ell \neq \cvpar{\ell_0}$ for some $\ell$.
Since the closure of $GO(\cO)$ is the Julia set, it follows that $\dd$ is a Fatou component. The converse is also straightforward.

  For (5), assume that there exists a  critical point $\omega \notin \filledvphi$. 
Let $\ell_0$ be the minimal $\ell$ such that $\vphi^\ell(\omega) \in U_0$. 
Denote by $\dd$ the direction at $\cwpar{\ell_0}$ containing $\omega$.
Since $\pi_{\arbol\ell} \circ \vphi = \vphi \circ \pi_{\arbol{\ell+1}}$, for all $\ell$, we have that 
the direction  $\dd$ is contained in a Fatou component which eventually maps onto $U_0$.
Therefore $\dd$ is free of $GO(\cO)$ elements and $\cwpar\ell= \cwpar{\ell_0}$ for all $\ell \ge \ell_0$.  
\end{proof}

\subsection{The collapsed trees}
\label{sec:collapsed}
To deal with the case $\crit(\vphi) \subset \filledvphi$ and  $\crit(\vphi) \not\subset \filledvphi$ simultaneously it is convenient to make the following agreement.

\begin{definition}
  \label{def:10}
  If $\crit(\vphi) \not\subset \filledvphi$, then let $\sim$ be the  relation in $\hatjuliavphi$ given by Proposition~\ref{pro:1}. Otherwise, let $\sim$ be the trivial equivalence relation (no distinct points are identified). We denote the $\sim$-class of $x$ by $[ x ]$.
\end{definition}

\begin{lemma}
  \label{lem:31}
  For all $\ell \ge 0$, if $x \in \arbol\ell$, then $[x] = \{ x \}$ or $[x]$ is a subtree of $\arbol\ell$.
Moreover, $\vphi([x]) = [\vphi(x)]$. 
\end{lemma}

\begin{proof}
Since $\vphi:\arbol{0} \to \arbol{0}$ is a 
bijection leaving $\widehat{\cO_0}$ invariant and $\vphi$ is an open map, it follows 
that $\vphi([x]) = [\vphi(x)]$ for all $x \in \hatjuliavphi$.

Now let  $]v,v'[$ be an edge of $\arbol\ell$.
Since $\vphi^{\ell}: \arbol\ell \to \arbol{0}$ is a simplicial map  and $\vphi:\arbol{0} \to \arbol{0}$ is a 
bijection leaving $\widehat{\cO_0}$ invariant, we have  that $\vphi^\ell(]v,v'[) \subset \widehat{\cO_0}$ 
or $\vphi^n (]v,v'[)$ is disjoint from $\widehat{\cO_0}$ for all $n \ge 0$. 
By Lemma~\ref{lem:10}, the endpoints of $\arbol\ell$ lie in $GO(\cO)$. Thus, endpoints have trivial $\sim$-classes.
Moreover, by definition, classes are connected and, using the fact that $\vphi$ is an open map, classes are also closed.
Therefore, for all $x \in \arbol\ell$, the class $[x]$ is a closed, connected and simplicial subset of $\arbol\ell$. That is, $[x]$ is
a subtree of $\arbol\ell$.
\end{proof}

It follows that $$\barbol\ell = \arbol\ell/ \sim$$
is a naturally endowed with a simplicial tree structure 
with vertices $\cWpar\ell = \cVpar\ell/\sim$. Moreover, the induced map
$$\begin{array}{cccc}
\psi : & \barbol{\ell+1} & \to & \barbol\ell \\
       &  [{\z}]              & \mapsto & [{\psi(\z)}]
\end{array}
$$
is a well defined degree two branched map over the interval $\intervalb\ell=\intervala\ell/\sim$.

\subsection{The ``postcritical'' trees}
\label{sec:postcritical}

{\sf In what follows we abuse of notation and drop the brackets to write the $\sim$-equivalence classes. Thus, we simply write $\veta_0$ for $[\veta_0]$,
$\xi_j$ for $[\xi_j]$, $\cwpar\ell$ for $[\cwpar\ell]$, etc. }
With this notation,
 $$\cWpar\ell = GO(\cO) \cup GO(\veta_0).$$

\begin{lemma}
  \label{lem:11}
  For $\ell \ge 0$, let
$$\barbolc\ell = \{ x \in \barbol\ell \mid \cwpar\ell \notin [ \veta_0 , \psi^k(x) [  \mbox{ for all } k \ge 0 \}. $$
Then, for all $\ell \ge 0$,  the following statements hold:
\begin{enumerate}
\item For all $x \in \barbol\ell$, we have that $\psi ([\veta_0, x]) \subset [\veta_0,\xi_1] \cup [\veta_0,\psi(x)]$.
\item $\cwpar\ell \in \barbolc\ell \subset \barbolc{\ell+1}$.
\item $\psi (\barbolc\ell) \subset \barbolc\ell.$
\item $\barbolc\ell$ is a (connected) subtree of $\barbol\ell$.
\item If $\cwpar\ell \neq w^{(0)}$, then $\intervalb\ell \subset \barbolc\ell$. Moreover, if $\cvpar\ell \in GO(\cO)$ and
$v$ is a vertex of $\barbol\ell$ such that $[\cvpar\ell,v]$ is an edge, then $v \in \barbolc\ell$.
\item If $\cwpar{\ell+1} = \cwpar\ell$, then $\psi(\barbolc{\ell+1}) \subset \barbolc\ell$.
\item If $\cwpar{\ell+1} \neq \cwpar\ell$ and $\cwpar{\ell+1} \in GO(\cO)$, then $\barbolc{\ell+1} = \barbol{\ell+1}$.
\end{enumerate}
\end{lemma}

\begin{proof}
  Statement (1) is trivial for $\ell=0$ so consider $\ell \ge 1$ and let $\veta'_0$ be
the unique point, different from $\veta_0$, such that  $\psi(\veta'_0)=\veta_0$. 
Note that there are two intervals in $\barbol\ell$ mapping onto $[\veta_0,\psi(x)]$.
Namely, $[\veta_0, z]$ and $[\veta'_0,z']$ where $\psi(z)=\psi(z')=\psi(x)$.
Thus, $x=z$ or $x=z'$. In the former case,  $\psi([\veta_0, z=x])=[\veta_0,\psi(x)])$. In the latter, 
$$\psi([\veta_0, z=x']) \subset \psi([\veta_0, \veta_0']) \cup \psi([\veta_0, z])
= [\veta_0,\xi_1] \cup [\veta_0,\psi(x)]).$$

To establish (2), recall that $\psi$ is a simplicial map. Thus, 
for all $k \ge 0$, the number of edges contained
in $[\veta_0, \psi^{k}(\cwpar{\ell})[$ is bounded above by the number of edges in
$[\veta_0, \cwpar{\ell}[$. It follows that $\cwpar{\ell} \notin [\veta_0, \psi^{k}(\cwpar{\ell})[$. Therefore, $\cwpar{\ell} \in \barbolc\ell$. The rest of statement
(2) as well as statement (3) follows directly from the definition.

For (4), observe that $\barbol{0} \subset \barbolc\ell$ for all $\ell$. 
Hence, if $x  \in \barbolc\ell$ and $y \in [\veta_0,x[$, then 
$$[\veta_0,\psi^k(y)[ \subset \psi^k([\veta_0, x[) \subset \barbol{0} \cup [\veta_0,\psi^k(x)[$$ for all $k \ge 0$, by statement (1). 
Therefore,  $\cwpar\ell \notin [\veta_0,\psi^k(y)[$, for otherwise
$x \notin \barbolc\ell$ or $\cwpar\ell \in [\veta_0,\psi^k(y)[ \subset \barbol{0}\setminus \cO$, which is impossible. It follows that $\barbolc\ell$ is a connected subtree.

For (5), assume that $\cwpar\ell \neq \cwpar0$. It is sufficient to show that $\cvpar\ell \in \barbolc\ell$.
Thus,   we may assume that $\cwpar{\ell+1} \neq \cwpar\ell$. If $\cwpar\ell \in [\veta_0, \vphi^k (\cwpar{\ell+1})[$ for some $k \ge 1$, then $\cwpar\ell \in [\vphi^k(\cwpar\ell), \vphi^k (\cwpar{\ell+1})[$, since the number of edges in $[\veta_0, \cwpar\ell]$ is an upper bound for the
number of edges in $[\veta_0,\vphi^k(\cwpar\ell)] \subset \vphi^k([\veta_0,\cwpar\ell])$. 
However, $]\cwpar\ell, \cwpar{\ell+1}[$ contains only vertices in $GO(\veta_0)$. 
By Lemma~\ref{lem:36} (5), we have that $\cwpar\ell \notin GO(\veta_0)$, thus the same holds for its orbit,
hence $\cwpar\ell =\vphi^k(\cwpar\ell)$ is periodic. The periodic points of $\barbol\ell$ belong to $\barbol{0}$ 
and $\cwpar\ell \notin U_0$.
Therefore $\cwpar\ell = \xi_0=\cwpar0$.

Now assume that $\cvpar\ell \in GO(\cO)$ and
$v$ is a vertex of $\barbol\ell$ such that $[\cvpar\ell,v]$ is an edge. 
If $\cvpar\ell \neq \cvpar{\ell-1}$, then there is a unique such a vertex $v$ by Lemma~\ref{lem:36} (3) and $v$ must lie in $[\veta_0, \cvpar\ell] \subset \barbolc\ell$.
If $\cvpar\ell = \cvpar{\ell-1}$ we proceed by contradiction. Suppose that $\cwpar\ell \in [\veta_0, \vphi^{k-1} (v)[$ for some $k \ge 1$.
As before, we conclude that $\cwpar\ell$ lies in  the interval $ [\vphi^{k-1}(\cvpar{\ell-1}), \vphi^{k-1} (v)[$ whose interior
is an edge since $\cvpar\ell = \cvpar{\ell-1}$. Therefore, $\cwpar\ell = \vphi^k(\cwpar\ell)$ is periodic and we obtain a contradiction.

Statement (6) follows from the definition and statement (7) is a consequence of Lemma~\ref{lem:36} (3). 
\end{proof}

Now we consider $\theta \in I_{p/q}$. Given $x \in \ponec \cong \C \cup \{ \infty \}$ denote by $\cvpar\ell_\theta (x)$ the vertex of $\tree\ell(\theta) = \cT(\boldsymbol{\lambda}(\theta)^{(\ell)})$ which contains $x$.
For $t \in \RZ$ we also denote by $\cvpar\ell_\theta (t)$ the vertex containing $\exp(2 \pi i t)$ when no confusion arises.
For $0 \le m \le \ell$, we denote by $\iota : \tree{m}(\theta) \to \tree\ell(\theta)$ the inclusion given by Lemma~\ref{lemma:iell}.

Recall from Lemma~\ref{lem:6}, that $m_2: \tree{\ell+1}(\theta) \to \tree\ell(\theta)$ is a degree two branched map over the interval
$\intervalt\ell\theta =[m_2(\cvpar{\ell+1}_\theta (\infty)), m_2(\cvpar{\ell+1}_\theta (0))]$. Again to simplify notation, let
\begin{eqnarray*}
  \cwpar{\ell+1}_\theta & = & \cvpar{\ell+1}_\theta (0), \\
 \cvpar\ell_\theta  & = & m_2(\cwpar{\ell+1}_\theta), \\
\beta_j & = & \iota \circ m_2^j (\cvpar\ell_\theta(\infty)) \in \tree\ell(\theta)
\end{eqnarray*}
for  $j \ge 0$ (subscripts $\mod q$). Finally let $\alpha_0 \in \tree\ell (\theta)$ be the (inclusion of the) unique $\Gamma$-vertex of level $0$.

\begin{lemma}
  \label{lem:26}
Given $\theta \in I_{p/q}$ let 
$$\treec\ell(\theta) = \{ x \in \tree\ell(\theta) \mid \cwpar\ell_\theta \notin [\alpha_0 , \iota \circ m_2^k (x) [ \mbox{ for all } k \ge 0 \}.$$
For all $\ell \ge 0$, the following statements hold:
\begin{enumerate}
\item For all $x \in \tree{\ell}(\theta)$,
$$\iota \circ m_2 ([\alpha_0,x]) = [\alpha_0, \beta_1] \cup [\alpha_0 , \iota \circ m_2 (x)].$$
\item $\cwpar\ell_\theta \in \treec\ell(\theta)$ and $\iota ( \treec\ell (\theta) ) \subset \treec{\ell+1}(\theta)$.
\item $\iota \circ m_2(\treec\ell) \subset \treec\ell$.
\item $\treec\ell(\theta)$ is a (connected) subtree of $\tree\ell(\theta)$.
\end{enumerate}
\end{lemma}

We omit the proof of the previous lemma since it is, after changing notation, identical  to that of Lemma~\ref{lem:26} (1)--(4).

Denote by $\gamc{\ell}(\theta)$ the set of $\Gamma$-vertices of level $\ell$ 
which belong to  $\treec\ell(\theta)$.

\begin{lemma}
  \label{lem:32}
  For all $\ell \ge 1$ and for all $\theta' \in \cvpar{\ell-1}_\theta$, 
 $$\gamc{\ell}(\theta) =  \gamc{\ell}(\theta').$$
Moreover, if $v$ is a vertex of $\treec{\ell}(\theta)$ such that 
$v \subset \cvpar{\ell-1}_\theta$, then $\partial v$ 
is contained in the union of the vertices of $\gamc{\ell}(\theta)$.
In particular, $v$ is a vertex of $\tree\ell (\theta')$ for all
$\theta' \in \cvpar{\ell-1}_\theta$.
\end{lemma}

\begin{proof}
Let $\theta' \in \cvpar{\ell-1}_\theta$. We first consider the case in which
  $\cvpar{\ell-1}_\theta$ is a $\Gamma$-vertex and then deal with the  case in which is a 
$Y$-vertex.

Assume that  $\cvpar{\ell-1}_\theta$ is a $\Gamma$-vertex. Then 
${\lambda}(\theta) = {\lambda}(\theta')$, by Lemma~\ref{lem:1}.
In particular, $\gamc{\ell}(\theta) =  \gamc{\ell}(\theta').$
Moreover, if $v \subset \cvpar{\ell-1}_\theta$, then $v = \cvpar{\ell-1}_\theta$ and
the lemma follows in this case. 

Assume that $\cvpar{\ell-1}_\theta$ is a $Y$-vertex.
Let $V$ be the connected component of $\ponec \setminus \cwpar\ell_\theta$ that contains
$\alpha_0$. By Lemma~\ref{lem:19}, $W = \ponec \setminus V$ is such that $W \cap \D$ is 
a topological disk convex with respect to the hyperbolic metric. Hence, the diameter connecting
$\theta'/2$ and $\theta'/2+1/2$ is completely contained in $W$.
The vertices  of $\treec\ell(\theta)$ are exactly those that belong to
$V$ and whose iterates also belong to $V$. 
Therefore, if $A$ is a $ {\lambda}(\theta)$-class whose convex hull is contained in
an element of $\Gamma_c(\theta)$, then $A$  is unlinked with $\{ \theta'/2, \theta'/2 +1/2 \}$.
Since $m_2(\Gamma_c(\theta)) \subset \Gamma_c(\theta)$, the same holds for the $m_2^{k}(A)$.
Hence, $A$ is a $ {\lambda}(\theta')$-class. It follows that 
$\Gamma_c(\theta) = \Gamma_c(\theta')$.

Now assume that $v$ is a $Y^{(\ell)}$-vertex which belongs to $\treec\ell(\theta)$
and that is contained in $ \cvpar{\ell-1}_\theta$.
Let $w$ be a $\Gamma^{(\ell)}(\theta)$-vertex such that $\partial v \subset w$.
Since all the iterates $m_2^k(v)$ are contained in the closed set $V$, the same occurs 
with $m_2^{k}(w)$. Therefore, $w \in  \treec\ell(\theta)$ and the lemma follows.
\end{proof}

\subsection{Topological Conjugacy}
\label{sec:TopConj}
The simultaneous proof of Proposition~\ref{bothin} and Proposition~\ref{pro:1} will rely on an inductive
construction of conjugacies defined on the level $\ell$ trees $\barbol\ell$.

The initial step consists of fixing a conjugacy $h_0 : \barbol0 \to \tree0 (\theta)$
between the dynamics of $\psi: \barbol0 \to \barbol0$ and that of $m_2: \tree0 (\theta) \to \tree0 (\theta)$ for all $\theta$ in the characteristic interval $I_{p/q}$. In fact, the vertices of $\barbol0$ are $\{\xi_0, \dots, \xi_{q-1},  \veta_0 \}$ and $\barbol0$ is the starlike tree obtained as the union of
the intervals $[\veta_0, \xi_j]$. Note that $\cwpar0 = \xi_0$. As described before Lemma~\ref{lem:26}, for all $\theta \in I_{p/q}$ there is a unique $\Gamma$-vertex $\alpha_0$ and the $Y$-vertices are $\beta_0, \dots, \beta_{q-1}$ where $\beta_0 = \cwpar0_\theta$. 
It follows that the unique tree isomorphism $h_0: \barbol0 \to \tree0 (\theta)$ such that
$h_0 (\veta_0) = \alpha_0$ and $h_0 (\xi_j) = \beta_j$ is a conjugacy.

\smallskip
Before proving the propositions we establish two necessary lemmas.

\begin{lemma}
  \label{lem:33}
Let $\lambda_*$ be the $\alpha$ lamination of the center of the $p/q$-limb.
Assume that $\cwpar0 = \cdots = \cwpar{\ell-1}$, for some $\ell \ge 1$.
 Then, for all $m=1,\dots,\ell$, there exists tree isomorphism $h_m : \barbol{m} \to \cT (\lam{m}_*)$ such that
$$ h_m \circ \psi = \iota \circ m_2 \circ h_m$$
and the restriction of $h_m$ to $\barbol0$ is $h_0$. 
\end{lemma}

\begin{proof}
  Assume that $h_m$ has been already constructed and that $0 \le m < \ell$. 
Since $\psi : \barbol{m+1} \to \barbol{m+1}$ (resp. $m_2 : \cT (\lam{m+1}_*) \to \cT (\lam{m+1}_*)$)
are degree two branched maps over the point $\xi_1$ (resp. $\beta_1$) and $h_m(\xi_1)=h_0(\xi_1)=\beta_1$,
from Lemma~\ref{lifting.lemma}, we may lift $h_m$ to a conjugacy $h_{m+1}$ with the desired properties.
\end{proof}

\begin{lemma}
  \label{lem:34}
  There exists a nested sequence $$\cvpar0_\cM \supset \cvpar1_\cM \supset \cdots$$
of non-empty subsets of $I_{p/q}$ such that for all $\ell \ge 1$ 
there exists a collection of tree isomorphism $$\{ h_{\ell,\theta}: \barbol\ell \to \tree\ell(\theta) \}_{\theta \in \cvpar{\ell-1}_\cM}$$
with the property that
 $$ \cvpar{\ell}_\cM = h_{\ell, \theta} (\cvpar{\ell})$$
for all  $\theta \in \cvpar{\ell-1}_\cM$.

Moreover, the following statements hold:
\begin{enumerate}
\item The restriction of $h_{\ell,\theta}$ to $\barbol0$ is $h_0$, for all $\theta \in \cvpar{\ell-1}_\cM$.
\item For all $\theta \in \cvpar{\ell-1}_\cM$,
the following diagram conmutes:
$$
\begin{CD}
  \barbol{\ell} @>>\psi> \barbol{\ell} \\
  @VVh_{\ell,\theta}V @VVh_{\ell,\theta}V\\
  \tree{\ell}(\theta) @>\iota \circ m_2>>
  \tree{\ell}(\theta).
\end{CD} 
$$
\item For all  $\theta, \theta' \in \cvpar{\ell-1}_\cM$
$$h_{\ell,\theta} (v) = h_{\ell,\theta'} (v)$$
if $v \in \barbolc\ell \cap GO(\veta_0)$. 
\end{enumerate}
\end{lemma}

\begin{proof}
If $\cwpar\ell = \cwpar0$ for all $\ell$, then the assertion follows from 
Lemmas~\ref{lem:2} and~\ref{lem:33} after declaring $\cvpar{\ell-1}_\cM=\cvpar{\ell-1}_*$ (with the notation of Lemma~\ref{lem:2}) and $h_{\ell,\theta}=h_\ell$ (with the notation of Lemma~\ref{lem:33}) for all $\theta \in \cvpar{\ell-1}_\cM$.

If, for some $\ell_0 \ge 1$, we have that $\cwpar{\ell_0} 
\neq \cwpar0 = \cwpar1 = \cdots =\cwpar{\ell_0-1}$, then for all $m=1,\dots, \ell_0$, 
we let $\cvpar{m-1}_\cM = \cvpar{m-1}_*$ (with the notation of Lemma~\ref{lem:2}) 
and $h_{m,\theta} = h_m$ as in Lemma~\ref{lem:33}. It is not difficult to check that $\cvpar\ell_\cM = h_{\ell,\theta} (\cvpar\ell_\theta)$ and that  properties (1)--(3) hold for all $\ell < \ell_0$.

For $\ell \ge \ell_0$, we proceed by induction. That is, for
all $m \le \ell$, we suppose that 
$\cvpar{m-1}_\cM$ and $h_{m,\theta}$ have already been defined  and that properties (1)--(3) are satisfied.

Note that by (1) and (2), the elements of $GO(\veta_0)$ map onto $\Gamma$-pieces  and the elements of
$GO(\cO)$ map onto $Y$-pieces under $h_{\ell,\theta}$. Also, the conjugacy implies that $h_{\ell,\theta} (\cwpar\ell) = \cwpar\ell_\theta$
for all $\theta \in \cvpar{\ell-1}_\cM$

First we claim that $ h_{\ell, \theta} (\cvpar\ell)$ is independent of  $\theta \in \cvpar{\ell-1}_\cM$.
In fact, under our assumption,  $\cwpar\ell \neq \cwpar0$ for all $\ell \ge \ell_0$.
By Lemma~\ref{lem:11} (5), $\intervalb\ell \subset \barbolc\ell$. In particular $\cvpar\ell \in \barbolc\ell$.
If $\cvpar\ell \in GO(\veta_0)$, then the claim follows from property (3).
If $\cvpar\ell \in GO(\cO)$, then for every $v \in \barbolc\ell$ such that $[\cvpar\ell,v]$ is an edge,
we have that $v \in \barbolc\ell$, by  Lemma~\ref{lem:11} (5). From property (3), every $\Gamma$-vertex which intersects non-trivially $\partial  h_{\ell, \theta} (\cvpar\ell)$
is independent of $\theta \in \cvpar{\ell-1}_\cM$. Hence, $h_{\ell, \theta} (\cvpar\ell)$ is independent of $\theta \in \cvpar{\ell-1}_\cM$.

Pick any $\theta \in \cvpar{\ell-1}_\cM$ and define $\cvpar\ell_\cM = h_{\ell, \theta} (\cvpar\ell)$.
We must show that $\cvpar\ell_\cM \subset \cvpar{\ell-1}_\cM$. In fact, 
by the inductive hypothesis (2), branched values correspond under
$h_{\ell, \theta}$ so we have that $\cvpar{\ell-1}_\theta = 
\pi_\ell \circ h_{\ell, \theta}(\cvpar{\ell-1})$
where $\pi_\ell: \tree\ell(\theta) \to \tree{\ell-1}(\theta)$ denotes 
the retraction of Lemma~\ref{lemma:piell}.
Moreover, $\pi_\ell \circ h_{\ell, \theta}(\cvpar{\ell-1}) =\pi_\ell \circ h_{\ell, \theta}(\cvpar{\ell})$ since the interval $]\cvpar{\ell-1}, \cvpar\ell]$ is 
not contained in $\barbol{\ell-1}$. 
From the definition of $\pi_\ell$ we conclude that 
 $\cvpar\ell_\cM = h_{\ell, \theta} (\cvpar\ell)$ is contained in 
$\cvpar{\ell-1}_\theta = h_{\ell-1,\theta}(\cvpar{\ell-1})=\cvpar{\ell-1}_\cM$.

Now we continue with the construction of 
the isomorphisms $h_{\ell+1,\theta}$ for all $\theta \in \cvpar{\ell}_\cM$.

To define $h_{\ell+1,\theta}$ for $\theta \in \cvpar\ell_\cM$ first we consider the case in which $\cvpar\ell \in GO(\veta_0)$.
It follows that $\cvpar\ell_\cM =h_{\ell, \theta} (\cvpar\ell)$ is a $\Gamma$-piece of level $\ell$. 
By Lemma~\ref{lem:1}, for all $\theta,\theta' \in \cvpar\ell_\cM$, we have 
$ {\lambda}(\theta)= {\lambda}(\theta')$.
Thus, pick an element $\theta_\ell$ of $\cvpar\ell_\cM$, and let $h_{\ell+1,\theta_\ell}: \barbol{\ell+1} \to \tree{\ell+1} (\theta_\ell)$  be the lift
obtained from $h_{\ell,\theta_\ell}$ after applying Lemma~\ref{lifting.lemma}. Declare $h_{\ell+1,\theta}=h_{\ell+1,\theta_\ell}$ for all $\theta \in \cvpar\ell_\cM$.
It is not difficult to check that properties (1)--(3) hold.

Now we assume that $\cwpar{\ell+1} \in GO(\cO)$ and define $h_{\ell+1,\theta}$ for $\theta \in \cvpar\ell_\cM$ 
by studying two cases according to whether
$\cwpar{\ell+1} \neq \cwpar\ell$ or $\cwpar{\ell+1} = \cwpar\ell$.

{\sf Case 1.} Suppose that $\cwpar{\ell+1} \neq \cwpar\ell$. By Lemmas~\ref{lem:11} (7) we have that $\barbolc{\ell+1} = \barbol{\ell+1}$.
Pick $\theta_\ell \in \cvpar\ell_\cM$ and  apply Lemma~\ref{lifting.lemma} to $\psi: \barbol{\ell+1} \to \barbol\ell$ and
$m_2: \tree{\ell+1}(\theta_\ell) \to \tree{\ell}(\theta_\ell)$ to lift $h_{\ell,\theta_\ell}$ to a tree isomorphism $h_{\ell+1,\theta_\ell}: \barbol{\ell+1} \to
\tree{\ell+1}(\theta_\ell)$, since by the choice of $\cvpar\ell_\cM$ we have that critical value intervals correspond under $h_{\ell,\theta_\ell}$.
It follows that $h_{\ell+1,\theta_\ell} (\barbolc{\ell+1}) = \treec{\ell+1}(\theta_\ell)$. Thus, we have that $\tree{\ell+1}(\theta_\ell)
= \treec{\ell+1}(\theta_\ell)$. 
By Lemma~\ref{lem:32}, we conclude that $\tree{\ell+1}(\theta) = \tree{\ell+1}(\theta')$ for all $\theta, \theta' \in \cvpar\ell_\cM$. 
Thus, for all $\theta \in \cvpar\ell_\cM$, we may simply define $h_{\ell+1, \theta}=h_{\ell+1, \theta_\ell}$ and (1)--(3) clearly hold in this case.

{\sf Case 2.} Suppose that  $\cwpar{\ell+1} = \cwpar\ell$. 
From Lemma~\ref{lem:11} we have that $\psi(\barbolc{\ell+1}) \subset \barbolc{\ell}$.
Let 
$$\barbols{\ell+1} = \psi^{-1} (\psi(\barbolc{\ell+1}))$$
and
$$\trees{\ell+1}(\theta) = m_2^{-1} (m_2(\treec{\ell+1}(\theta)))$$
Pick $\theta_\ell \in \cvpar\ell_\cM$. 
We may lift $h_{\ell,\theta_\ell} : \barbolc\ell \to \treec\ell(\theta_\ell)$ to
a tree isomorphism   $h_{\ell+1,\theta_\ell}:\barbols{\ell+1} \to \trees{\ell+1}(\theta_\ell)$.
Since the $\Gamma$-vertices of $\treec{\ell+1}(\theta)$ are independent of $\theta \in \cvpar\ell_\cM$ (Lemma~\ref{lem:32}),
we have that they coincide with the  $\Gamma$-vertices $\treec{\ell+1}(\theta_\ell)$.
We identify  $\treec{\ell+1}(\theta)$ with $\treec{\ell+1}(\theta_\ell)$ via the unique tree isomorphism which preserves
$\Gamma$-vertices. This tree isomorphism clearly extends to one from $\trees{\ell+1}(\theta)$ onto $\trees{\ell+1}(\theta_\ell)$.
Thus, we may define $h_{\ell+1,\theta}:\barbols{\ell+1} \to \trees{\ell+1}(\theta)$
as equal to $h_{\ell+1,\theta_\ell}:\barbols{\ell+1} \to \trees{\ell+1}(\theta_\ell)$, for all $\theta \in \cvpar\ell_\cM$.
Now we may apply Lemma~\ref{lifting.lemma} to  extend $h_{\ell+1,\theta}$ to $\barbol{\ell+1}$ by successive lifts 
to $\left(\psi_{|\barbol{\ell+1}}\right)^{-k}(\barbols{\ell+1})$. Since every element of $\barbol{\ell+1}$, eventually
maps into $\barbols{\ell+1}$, we obtain the desired isomorphism $h_{\ell+1,\theta}:\barbol{\ell+1} \to \tree{\ell+1}(\theta)$.
\end{proof}

\begin{proof}[Proof of Proposition~\ref{bothin} and Proposition~\ref{pro:1}]
  Let $\cvpar\ell_\cM$ and $h_{\ell,\theta}$ be as in the previous lemma. Take a monotone convergent sequence $\{ \theta_n \} \subset \RZ$ 
such that $\theta_n \in \cvpar{n}_\cM$ for all $n$. Let $\theta_\infty$ denote the limit of $\theta_n$. 
If the sequence is eventually constant, then let $\bla = \bla(\theta_\infty)$.
Otherwise, let $\bla = \bla(\theta_\infty)^+$ when the sequence is increasing, and let $\bla = \bla(\theta_\infty)^-$
when the sequence is decreasing. 
It follows that for all $\ell \ge 0$, there exists $n(\ell)$ such that $\lam\ell = \lambda(\theta_n)^{(\ell)}$ for all $n \ge n(\ell)$.
Since the number of isomorphisms between $\barbol\ell$ and $\cT(\lam\ell)$ is finite, we may recursively extract subsequences
$\{ n_k(\ell) \}$ such that
\begin{enumerate}
\item $n_k(0) = k$.
\item $\{ n_k(\ell+1) \}$ is a subsequence of $\{n_k(\ell)\}$.
\item $n_0(\ell) \ge n(\ell)$. 
\item The restriction of $h_{n_k(\ell), \theta_{n_k(\ell)}}$ to  $\barbol\ell$ is constant, say equal to $h_\ell$.
\end{enumerate}
Passing to the inverse limit of $\{h_\ell\}$ we obtain a conjugacy $h: \hatjuliavphi/\sim \to \tree\infty(\bla)$.  
Since $h$ maps $GO(\veta_0)$ onto the inverse limit of the $\Gamma^{(\ell)}(\bla)$, it follows that
$\cwpar{\ell_0} \in GO(\veta_0)$ if and only if $h_{\ell_0}(\cvpar{\ell_0})$ is a $\Gamma$-vertex containing $\theta_n$
for sufficiently large $n$. In this case, the sequence $\theta_\infty$ eventually maps onto the fixed class of $\bla$. 
That is, $\bla$ is critically prefixed. 
\end{proof}

\begin{proof}[Proof of Theorem~\ref{ithr:dynamics}]
  According to Proposition~\ref{classification} we have three possibililies:
  \begin{itemize}
  \item[(a)] $\juliavphi \cap \hl$ is periodic point free.
  \item[(b)] There exists an indifferent periodic orbit $\cO$ in $\juliavphi \cap \hl$.
  \item[(c)] There exists a repelling periodic orbit $\cO$ in $\juliavphi \cap \hl$.
  \end{itemize}
From Proposition~\ref{attracting.fixed.point}, we have that (a) implies that
Theorem~\ref{ithr:dynamics} (1) holds. From Proposition~\ref{pro:14}, we conclude
that (b) implies Theorem~\ref{ithr:dynamics} (2).  From  Proposition~\ref{bothin} and Proposition~\ref{pro:1} we conclude that (c) implies Theorem~\ref{ithr:dynamics} (3) (a) or (b).
\end{proof}
%%% Local Variables: 
%%% TeX-master: "main.tex"
%%% End: 

%\input{OneParameterFamilies}
%\input{Moduli}

%\backmatter
\bibliographystyle{amsalpha}

\providecommand{\bysame}{\leavevmode\hbox to3em{\hrulefill}\thinspace}
%\providecommand{\MR}{\relax\ifhmode\unskip\space\fi MR }
%\providecommand{\MRhref}[2]{%
%  \href{http://www.ams.org/mathscinet-getitem?mr=#1}{#2}
%}
%\providecommand{\href}[2]{#2}

%\bibliography{/home/jkiwi/Papers/index}

\end{document}